\documentclass[10pt]{amsart}
\usepackage{amscd,amsmath,amssymb,amsfonts}
\usepackage[all]{xy}
\usepackage{hyperref}
\usepackage{url}
\usepackage{stmaryrd}
\usepackage{color}

\theoremstyle{plain}
\newtheorem{thm}{Theorem}
\newtheorem{lem}[thm]{Lemma}
\newtheorem{cor}[thm]{Corollary}
\newtheorem{prop}[thm]{Proposition}

\newtheorem{defn}[thm]{Definition}
\newtheorem{conj}[thm]{Conjecture}
\theoremstyle{definition}
\newtheorem{fact}[thm]{Fact}
\newtheorem{ex}[thm]{Example}

\newtheorem{nota}[thm]{Notations}

\newtheorem{claim}[thm]{Claim}
\newtheorem{rmk}[thm]{Remark}
\newtheorem{rmks}[thm]{Remarks}
\numberwithin{thm}{section}

\newcommand{\p}{\partial}

\newcommand{\ml}[2]{\begin{multline}\label{#1}#2 \end{multline}}
\newcommand{\ga}[2]{\begin{gather}\label{#1}#2 \end{gather}}

\newcommand{\surj}{\twoheadrightarrow}

\newcommand{\sA}{{\mathcal A}}

\newcommand{\sD}{{\mathcal D}}
\newcommand{\sE}{{\mathcal E}}
\newcommand{\sF}{{\mathcal F}}
\newcommand{\sG}{{\mathcal G}}
\newcommand{\sH}{{\mathcal H}}

\newcommand{\sL}{{\mathcal L}}
\newcommand{\sM}{{\mathcal M}}

\newcommand{\sO}{{\mathcal O}}

\newcommand{\sT}{{\mathcal T}}

\newcommand{\A}{{\mathbb A}}

\newcommand{\C}{{\mathbb C}}

\newcommand{\F}{{\mathbb F}}
\newcommand{\G}{{\mathbb G}}

\newcommand{\N}{{\mathbb N}}
\renewcommand{\P}{{\mathbb P}}
\newcommand{\Q}{{\mathbb Q}}
\newcommand{\R}{{\mathbb R}}

\newcommand{\Z}{{\mathbb Z}}

\DeclareMathOperator{\GL}{GL}

\begin{document}

\title[Local Systems]{ Lectures on Local Systems in Algebraic-Arithmetic Geometry}
\author{H\'el\`ene Esnault }
\address{Freie Universit\"at Berlin, Arnimallee 3, 14195, Berlin,  Germany}
\email{esnault@math.fu-berlin.de}

\maketitle


\begin{abstract}[BOOK BACK COVER ] The topological fundamental group of a smooth  complex algebraic variety is poorly understood. One way to approach it is to consider its complex linear representations modulo conjugation, that is  complex local systems. 
One fundamental problem is to recognize those coming from geometry, and more generally subloci of the moduli space of local systems with special arithmetic properties. This is the object of deep conjectures. We present in this Lecture Notes some properties of the topological fundamental group, and
some consequences of the conjectures,   notably  density, integrality and crystallinity properties of some special loci.
  \end{abstract}

\section{Lecture 1: General Introduction} \label{sec:intro}

The {\it topological fundamental group} $\pi_1(M,m)$ of a connected finite $CW$-complex $M$ based at a point $m$, as defined by Poincar\'e, is a  finitely presented group. In turn, any finitely presented group is the fundamental group of a connected  finite $CW$ complex.  The finite generation enables one to define a  ``moduli'' (parameter) space $M_B(\pi_1(M, m),r)$ of all its semi-simple complex linear representations $\rho: \pi_1(M,m)\to GL_r(\C)$ in a given rank $r$, modulo conjugation, or equivalently, of all its rank $r$ semi-simple complex local systems $\mathbb L$. It is called the character variety  of $\pi_1(M, m)$, also the {\it Betti moduli space} of $M$ in rank $r$, as conjugation washes out the choice of $m$, and is a scheme of finite type defined over the ring of  integers $\mathbb Z$. 

\medskip 

We are interested in the case when  $M$ consists of the complex points  $X(\C)$ of  a smooth connected algebraic quasi-projective variety $X$
 of finite type over the complex numbers $\mathbb C$, in which case $M$ has the homotopy type of a connected  finite $CW$ complex. 
We know extremely little  about  the restrictions the algebraic  origin of $M=X(\C)$   imposes on  the topological fundamental group $\pi_1(X(\C),x(\C))$. On the other hand, there are naturally defined local systems $\mathbb L$, namely those which upon restriction to some Zariski dense open $U\hookrightarrow X$ are subquotients (equivalently, summands by Deligne's semi-simplicity theorem)  of a local system on $U$ which comes from the variation of the cohomology of the fibres of a smooth projective morphism $g: Y\to U$.  Such $\mathbb L$ are called {\it  geometric}. 
One such example is when $U=X$ and $g$ is in addition  {\it finite} in the topological sense. 
Then the monodromy of $\mathbb L$, that is the group ${\rm Im}(\rho)$, defined up to conjugacy, is finite. 
 By the Riemann existence theorem this  is equivalent to $g$ being finite \'etale, 
 and thus relates $\pi_1(X(\C),x(\C))$ to its profinite completion $\pi_1(X_{\C}, x_\C)$. The latter is  the \'etale fundamental group, defined by Grothendieck, itself related to the Galois group of the field of functions of $X_\C$.

\medskip

 So it is  natural to try to single out  complex points of $M_B(X,r)$ which correspond to geometric  or even finite local systems. More generally, it is natural to try to define a notion of geometric sublocus of higher dimension.  It is clearly an  inaccessible task, which is reminiscent of the Hodge and the Tate conjectures: how can we construct $g$ out of $\mathbb L$?  
 There are several {\it conjectures} relying on various aspects of $M_B(X,r)$. 
 
 \medskip
 
 {\it Grothendieck's $p$-curvature conjecture}: It relies on the {\it Riemann-Hilbert correspondence} which equates the complex points  $\mathbb L$ of $M_B(X,r)$ with algebraic semi-simple integrable connections  $(E,\nabla)$ on $X$ (say $X$ projective for simplicity to avoid boundary growth conditions): we  consider $(E,\nabla)$ mod $p$ for all large primes $p$  and require
  this characteristic $p>0$ connection to be generated by flat sections.  This is the original formulation and should characterize finite local systems $\mathbb L$. More generally, to characterize geometric local systems $\mathbb L$,  we  request $(E,\nabla) $ mod $p$ for all large $p$ to be filtered so that the associated graded is spanned by flat sections. Since the work by Katz 
  which roughly (a bit less) shows that on a geometric $\mathbb L$ we can characterize its finiteness  by the generation of $(E,\nabla)$ mod $p$ by its flat section for almost all $p$, 
  since the later work  by Chudnovsky-Chudnovsky,
  Bost 
  and Andr\'e 
   which handle the solvable case,  and some remarks of the type made with Kisin,
  there has been   essentially  no major progress on this viewpoint.
  
  \medskip
  
   We discuss in Lecture~\ref{sec:pcurv} one possible origin of Grothendieck's $p$-curvature conjecture,  as we understand it, by relating it to the classical rationality criteria of Kronecker. Further in this Lecture,  we write a simplified form of Grothendieck's $p$-curvature conjecture which is equivalent to the general form. We also discuss Katz' proof,
   giving  a slightly different viewpoint.
  
  \medskip

   {\it Gieseker-de Jong conjecture}: It relies simply on the {\it finite generation} of  the topological fundamental group $\pi_1(X(\C),x(\C))$,  which implies the theorem of Mal\v{c}ev-Grothendieck 
   saying that the \'etale fundamental group  $\pi_1(X_\C,x_\C)$ controls the size of $M_B(X,r)$: if $\pi_1(X_\C,x_\C)=\{1\}$ then $M_B(X,r)$ consists of one point, the trivial local system  of rank $r$  (in fact there are no extensions as well). Gieseker's conjecture, 
   solved with Mehta, 
    asserts an analog in characteristic $p>0$ for infinitesimal crystals, while de Jong's conjecture, which is still unsolved  in its generality 
    (we understood  small steps with Shiho)
    predicts an analog for isocrystals. It is also related to the Langlands program: if the ground field is $\bar \F_p$ and the isocrystal is endowed with a Frobenius structure, then the existence of $\ell$-adic companions (as proven with Abe, see also Kedlaya)
     initially predicted by Deligne in Weil II
     proves the conjecture. 
       It would be of interest to understand a generalization of de Jong's conjecture on prismatic crystals which encompasses his initial formulation. 
       
       \medskip
       We mention in Lecture~\ref{sec:Malcev} the proof of the Mal\c{c}ev-Grothendieck theorem. It relies on the finite generation property of the topological fundamental group.  We sketch our proof with Mehta of the Gieseker conjecture when $X$ is smooth projective.  Although the geometric fundamental group of $X$ is topologically finitely generated, the proof does not use 
       directly this property. Instead it uses  the boundedness  of Frobenius divided sheaves.
      
       \medskip
       
     {\it Topological properties of the (tame) fundamental group of a smooth (quasi-) projective variety $X$ defined over an algebraically closed field  $k$ of characteristic $p>0$}:   As already mentioned, if $k$ was the field of complex numbers, then the topological fundamental group of $X$ would be finitely presented. By  Lefschetz theory and Grothendieck's specialization homomorphism from characteristic $0$ to characteristic  $p>0$, it is easy to see that the (tame)  fundamental group of $X$ is  finitely generated as a topological group. 
     But more is true, it is also finitely presented, at least if we assume that $X$ has a good normal crossings compactification.  This is a strong analogy with the classical topological situation.

        \medskip
        
        We present in Lecture~\ref{sec:similarity} our proof of it, joint with Shusterman and Srinivas.  It rests on Lubotzky's remarkable theorem which characterizes cohomologically this property. The characterization has a motivic flavor in that it  says that the $H^2$ of the  (tame) fundamental group with values in rank $r$ continuous representations with $\F_\ell$-coefficients growths linearly in $r$ but does not see $\ell$. Indeed, the proof of the independence of $\ell$ ultimately relies on Deligne's purity from the Weil conjectures. The proof of the  linearity in $r$ for $\ell=p$ is more geometric, and, in case $X$ is not proper, uses the existence of a good compactification to have a numerical characterization of tameness.

       \medskip
       By the Riemann existence theorem over $\C$ and the base change theorem,  the  fundamental group of a smooth (quasi-)projective variety 
       over an algebraically closed field  of characteristic $0$ is the profinite completion of a finitely presented abstract group. 
       We can adapt the notion of finite generation or presentation coming from a discrete group by removing the condition at $p$. This definition is shaped on the properties of Grothendieck's specialization homomorphism.
       It yields the concept of $p'$-finitely generated or presented group.  This property, even solely at the level of the $p'$-finite generation, is an obstruction for $X$ to lift to characteristic $0$.         
       \medskip
       
       We present in Lecture~\ref{sec:diff} our proof of this newly defined obstruction, joint with Srinivas and Stix.  Ultimately it relies again on a motivic property which this time is not always verified: 
       the representation  of the  automorphism group of $X$  on its (first) $\ell$-adic cohomology is  independent of $\ell$ but is not (always) defined over $\Q$.

       \medskip
       
       {\it Density of Special Loci}: For $X$ smooth quasi-projective over $\C$, Drinfeld analyzed some arithmetic properties of the Betti moduli space $M_B(X,r)$ viewed as a scheme over $\Z$. At good closed points of $M_B(X,r)$ of characteristic $\ell $, de Jong's conjecture applied to the mod $p$ reduction of $X$ for $p$ large predicts the existence of many deformations of the residual representation defined by the closed point to an $\ell$-adic sheaf  which is arithmetic, that is which is acted on by a power of the Frobenius.   Due to the arithmetic Langlands program, those sheaves then are pure in the sense of Deligne, so for example they obey the Hard Lefschetz property if they are semi-simple and $X$ is projective. De Jong's conjecture has been proven by  B\"ockle-Khare in special cases and by 
       Gaitsgory in general (for $\ell \ge 3$) using the geometric Langlands program.  It makes it then possible to derive density of certain subloci  of $M_B(X,r)$ using those (deep) arithmetic methods. 
       
       \medskip
       
       We present in Lecture~\ref{sec:qu} a proof, joint with Kerz,  that  the set of complex points of $M_B(X,r)$ corresponding to  semi-simple local systems with quasi-unipotent monodromy at infinity is Zariski dense. It uses Drinfeld's idea. This is an invitation to transpose the definition of arithmetic $\ell$-adic local systems on the mod $p$ reduction of $X$ to a notion of weakly arithmetic  complex local systems on $X$ over $\C$. We present the definition and the proof, joint with de Jong, that the set of weakly arithmetic local systems in $M_B(X,r)$ is Zariski dense.  All weakly arithmetic   local systems have quasi-unipotent monodromies at infinity, so this density is sharper than the previous one. The way to go from a complex local system to an $\ell$-adic \'etale local system in the definition prevents us to conclude that there only countably many weakly arithmetic local systems.

       \medskip
       
       On the other hand,  Biswas-Gupta-Mj-Whang in rank $2$, resp.  Landesman-Litt in any rank proved by topological,  resp. Hodge theoretical methods  that on a geometric generic curve of genus $\ge 2$, semi-simple  complex local systems  of low rank which descend to the universal curve are unitary. 
       
       \medskip
       
        Granted this, we present in Lecture~\ref{sec:qu}  the idea  of their proof in any rank to the effect that in fact not only  those local systems are unitary,  but they have finite monodromy. 
        The proof  ultimately relies on the integrality of cohomologically rigid local systems, a theorem joint with Groechenig, explained in  Lecture~\ref{sec:companions}.  It also shows that local systems on the geometric generic curve which come from the universal curve cannot be dense in the Betti moduli, contrary to a hope I had expressed with Kerz earlier on.

       \medskip

       {\it Companions and integrality}:  A complex local system can be twisted by an automorphism  $\sigma$ of the field of complex numbers: we post-compose the underlying linear representation of the topological local system by  $\sigma$ on the coefficients. If we now consider   representations of a topological group with values in a linear group with coefficients in a topological field, we lose continuity by post-composing. So we cannot define twisted continuous representations. Deligne in Weil II predicted that nonetheless it is possible in the following situation. We assume that the topological group is the geometric fundamental group of a smooth quasi-projective variety $X$ defined over a finite field $\F_q$, the field of coefficients is the algebraic closure  $\bar \Q_\ell$ of the $\ell$-adic numbers for some prime number $\ell$ prime to $p$,  the isomorphism $\sigma$  is  some abstract isomorphism between $\bar \Q_\ell$ and $\bar \Q_{\ell'}$.  If the local system is  stabilized by some power of  the Frobenius,  that is if it is arithmetic, then by the Langlands correspondence the
  characteristic polynomial of the action of the  Frobenius elements at closed points is a polynomial with coefficients being algebraic numbers. Thus $\sigma$ sends them to other algebraic numbers. In addition, a simple $\ell$-adic local system  is recognized by the \v{C}ebotarev density theorem  by those polynomials. Deligne then conjectured the existence of a simple $\ell'$-adic local system with those data, which he called ($\sigma$)-companion.      If we believe that  arithmetic local systems come from geometry, then we know
  at least on the whole Gau{\ss}-Manin system
   how to go from $\ell$ to $\ell'$.  If $X$ has dimension $1$, as a corollary of the Langlands correspondence, L. Lafforgue proved the existence of companions the way Drinfeld did after he proved the rank $2$ case of the Langlands program. In addition, arithmetic local systems come from geometry.   There is no Langlands program in higher dimension.  In absence of any geometric support, Drinfeld proved the existence of companions in higher dimension as well.    
       
       \medskip
       
       We present in Lecture~\ref{sec:companions}  some aspects of Drinfeld's proof. The starting point is the dimension $1$ case due to L. Lafforgue, see above. The problem is how to glue  those $\ell'$-adic sheaves defined on all curves, which agree on intersections. 
       Drinfeld produces a  non-commutative version of the construction performed by Wiesend. He uses a strong form of the \v{C}ebotarev density theorem:  an $\ell'$-adic local system on a curve is recognized on finitely many points.  He also uses Deligne's earlier work showing that for a given $\ell$-adic local system, the coefficients of the characteristic polynomials of the Frobenii at closed points stay in a finite type extension of $\Q$.

       \medskip

  {\it Simpson's motivicity conjecture: Rigid Local Systems}. These are the $0$-dimensional components of $M_B(X,r)$. Simpson   predicted that they  are all geometric.  It relies on the corresponding theorem by Katz 
   when $X$ has dimension $1$, in which case $X$ has to be an open in $\P^1$ (so in the definition of $M_B(X,r)$ one has to fix conjugacy classes of quasi-unipotent monodromy at infinity). It relies also on the  Simpson correspondence, which, when $X$ is projective, equates real analytically $M_B(X,r)$ with the moduli space of semi-stable Higgs bundles with vanishing Chern classes. Those are endowed with a $\C^\times$-flow, thus rigid local systems, viewed on the Higgs side, are fixed by it. This implies, according to   Simpson's  theorem, that they underly a polarizable complex variation of Hodge structures. From there it is one short step to dream of geometricity. 
   
   \medskip
   
    We present in Lecture~\ref{sec:companions}  our proof with Groechenig of a consequence of the motivicity conjecture, called the integrality conjecture, also formulated by Simpson: rigid local systems should be integral. To this aim we descend the $\ell$-adic local system mod $p$ for $p$ very large. The rigidity implies that this $\ell$-adic local system is in fact arithmetic as well, a fact already observed by Simpson.  We then take a companion, and interpret back this $\ell'$-local system topologically on the initial complex variety. To be able to conclude we need that this newly defined topological system is rigid as well. We can prove this only under the extra assumption that the rigid local system we started with was a smooth point in its moduli. This is a cohomological condition which we call cohomological rigidity.

        \medskip

   {\it Integrality of the Betti moduli space}:  An important point is that rigid local systems are arithmetic. On the other hand, as mentioned above as a consequence of  de Jong's conjecture, weakly arithmetic local systems are dense.

   \medskip 
   
     We present in Lecture~\ref{sec:companions}  our proof with de Jong of an integrality property of the whole  open of the Betti moduli corresponding to irreducible local systems. The proof   rests on this notion of weakly arithmetic local systems together with the same idea to prove integrality of (cohomologically) rigid local systems. 
      One consequence of the property is that it yields a new obstruction, visible on the Betti moduli, for a finitely presented group to be the topological fundamental group of a smooth complex quasi-projective variety.

   \medskip
   
    {\it Crystallinity property of Rigid Local Systems}: If we believe in Simpson's motivicity conjecture,  it should have two $p$-adic consequences, one on the local system  $\mathbb L$ itself, the other one on the flat connection $(E,\nabla)$  which is defined using the Riemann-Hilbert correspondence.  Indeed if those data ultimately  come from a splitting 
    of a Gau{\ss}-Manin system of a smooth projective morphism $g: Y\to U$ on some dense open $U\hookrightarrow X$, then for $p$ large all the data have good reduction. 
 So one expects that the $p$-completion of the local  system on  a $p$-adic model of $X$ with good reduction, for $p$ large, is crystalline. At the same time, if $g$ does not have good reduction but $X$ has, one expects the connection, in restriction to the $p$-adic model, to yield an isocrystal with a Frobenius structure. Those two properties, with less precision on the ``good'' $p$,  were proven with Groechenig.

  \medskip
  
    In Lecture~\ref{sec:Fiso} we present a proof of the $F$-isocrystal property which relies on a fact which is nearly documented in the Stack Project and which I learned from de Jong: if the integral  $p$-adic model of $X$ is unramified, and $p\ge 3$, then the Frobenius acts on connections defined on it. This enables one  to  prove a posteriori that the $p$-curvature of the mod $p$ reduction of the connection is nilpotent. Our original proof with Groechenig showed this property first, which was a bit more difficult. For the crystallinity of the $p$-adic local system however, in the present state of knowledge, we have to return to our proof with Groechenig which proceeds  along the completion along $p$ by showing that we have a periodic Higgs-de Rham flow in the sense of Zuo and his coauthors.

    \medskip
    
    It is to be noted that this crystallinity proof   is easily generalized 
 to non-proper smooth varieties under a strong cohomological     condition. We do not reproduce our proof  with Groechenig in those notes. We observe that precisely this property is the one used by 
 Pila-Shankar-Tsimerman to prove the Andr\'e-Oort conjecture on Shimura varieties of real rank $\ge 2$.

    \medskip
    
    \medskip
    
    There are manifold topics close to the ones discussed in the notes which are not addressed. One of them concerns specific subloci of the Betti moduli space which could be called arithmetic and which we defined with Kerz. It generalizes to higher dimensional  subvarieties of the Betti moduli the notion of arithmeticity of an  integral $\ell$-adic  point. In rank $1$ we proved with Kerz that those are translates by torsion points of subtori in the whole moduli  which itself  is a torus. We could not progress in higher rank in this direction, for lack of understanding the more precise geometry of the moduli.  Another direction is what to expect after the work of Petrov who showed that the generalized Fontaine-Mazur conjecture could simply be formulated by predicting that geometrically  irreducible $\ell$-adic local systems which are arithmetic on $X$ smooth quasi-projective in characteristic $0$ should in fact come from geometry. It would be of help to find one  small consequence of this vast conjecture, as a reality check.

    \medskip
   
   Finally in Lecture~\ref{sec:CQ} we formulate a few questions and problems related to the various Lectures.

\newpage

\subsection{Acknowledgements}  \label{ack}  Special heartfelt thanks go to Michael Groechenig, Moritz Kerz and more recently Johan de Jong with whom I developed part of the mathematics presented here. 
Those notes  hugely reflect  ideas  we developed together and further share.  

\medskip

 The influence of the ideas of Pierre Deligne and Vladimir Drinfeld is overwhelming and permeates the whole edifice of ideas exposed in the notes.  I thank in addition Pierre Deligne for a thorough list of comments and remarks on a first draft of those notes.
 
 \medskip 
 
I  thank
Tomoyuki Abe, Lars Kindler, Mark Kisin,  Adrian Langer,  Vikram Mehta, who unfortunately is no longer among us,   Atsushi Shiho,  Mark Shusterman, Vasudevan Srinivas,  Jakob Stix for the joint work reproduced  in those notes.

\medskip

More specifically,  I thank 
Piotr Achinger, Yves Andr\'e,  Benjamin Bakker, Alexander Beilinson, Barghav Bhatt,  Emmanuel Breuillard, Yohan Brunebarbe, Benjamin Church,  Dustin Clausen,  Marco D'Addezio, Pierre Deligne, Johan de Jong, Vladimir Drinfeld, Mikolaj Fraczyk, Michael Groechenig, Michael Harris,   Lars Hesselholt, Moritz Kerz,  
Mark Kisin, Bruno Klingler,  Raju Krishnamoorty, Aaron Landesman, Adrian Langer, Daniel Litt, Alexander Lubotzky,  Akhil Mathew, John Morgan, Alexander Petrov, Jonathan Pila,  Maxime Ramzi, 
Vasily Rogov, Will   Sawin,  Ananth Shankar, Carlos Simpson, Jacob Tsimerman, Vadim Vologodsky.  I had with them exchanges at  various stages, they all contributed in this way to the elaboration of the Lecture Notes. 

\medskip

I thank the mathematicians from Columbia University. They kindly invited me to give the Eilenberg Lectures  in the fall 2020 which are the support of the notes. Because of the pandemic,  I ``used up'' three chairs, Michael Thaddeus, Robert Friedman and Johan de Jong. I am sorry for the work it created for them. Finally I came during Johan's ``reign'' in the fall 2022. It has been   a wonderful time.

\medskip

I thank the mathematicians of the University of Copenhagen.  I gave there an echo of the Eilenberg Lectures in the fall 2021 and the winter 2023. 

\medskip 

I thank all the students, post-docs and Faculty who listened to the lectures and contributed to make them colorful and I hope enjoyable.
I also thank the administrative staff  of both institutions for warmly welcoming me.

\medskip

 I thank  Springer Verlag which allowed me to write the notes in a colloquial way, Lecture 1, Lecture 2 etc.... To understand the details of some proofs, we always have to go back to the original literature. This is all the more true  for those notes. The point to write them is to convey a philosophy and a  program which I developed partly with other mathematicians  within the last 10 or 15 years, and  to share some dreams.  Sharing dreams can only be done  in an intimate style.

\newpage

\tableofcontents
\newpage

\section{ Lecture 2:  Kronecker's Rationality Criteria and Grothendieck's $p$-Curvature Conjecture} \label{sec:pcurv}

 \begin{abstract}  We recall two criteria, say an analytic one and an algebraic one,  for an integral number to be  a root of unity and for an algebraic number to be a rational number. Both go back to Kronecker. We recall Grothendieck's $p$-curvature conjecture and its generalization, avoiding the general definition of the $p$-curvature of a connection in characteristic $p>0$. We show how the two are related and mention Katz's proof using (a generalization of ) the analytic criterion. 
  \end{abstract}

\subsection{Kronecker's criteria}
Let $a\in \C$ be a complex number, and write it as $a={\rm exp}(2\pi \sqrt{-1} b)$ for $b\in \C$ defined modulo the  integers $\Z$. We list Kronecker's criteria for $a\in \mu_{\infty}\subset \C$, that is for $a$  to be a root of unity, or equivalently for $b\in \Q$,  that is for $b$  to be a rational number. 
\subsubsection{Kronecker's analytic criterion, \cite{Kro57}}
Recall that the subring $\bar \Z\subset \C$ of algebraic integers of the complex numbers consists of those complex numbers $a\in \C$
 satisfying an equation $f(a)=a^d+c_1a^{d-1}+\ldots +c_d=0$ with $c_i\in \Z$, and the subfield $\bar \Q\subset \C$ of algebraic numbers consists of those $a$ as before but with $c_i\in \Q$. Any field automorphism $\sigma\in {\rm Aut}(\C)$  of $\C$ leaves $\Z\subset \Q$ invariant.

\begin{prop} \label{prop:Kronecker_analytic}
Assuming $a\in \bar \Z$, then $a\in \mu_\infty$ if and only if for any $\sigma \in {\rm Aut}(\C)$, the complex  absolute value of 
$\sigma(a)$ is equal to $1$. 

\end{prop}
\begin{proof} See {\tiny {\url{https://mathoverflow.net/questions/10911/english-reference-for-a-result-of-kronecker}}}
For $f=f_1$ as above, we write $f_n(X)=\prod_{i=1}^d (X-\alpha^n_i) \in \C[X]$  for all $n\in \N_{>0}$, with $\alpha_1=a$. The coefficients of $f_n$  are symmetric functions in the $\alpha_i$, thus are expressable as polynomials with rational coefficients in the $c_i$, and on the other hand they are in $\bar \Z$, thus they lie in $\Z$, and have bounded norms. Thus there are finitely many such $f_n$, thus the set $\{a^n, n\in \N_{\ge 1}\}$ is finite, thus lies in $\mu_\infty$. 
\end{proof}
\subsubsection{Kronecker's algebraic criterion, \cite{Bau04}}
Assume $b \in \bar \Q$, so $b$ lies in the number field $\Q(b)$, of rank $d$ say. So we can take its reduction modulo  almost all primes,  that is modulo all  except finitely many primes $\frak{p}$ of the number ring  $\sO_{\Q(b)}$.

\begin{prop} \label{prop:Kronecker_algebraic}
If for almost all primes $\frak{p}$ of  $ \sO_{\Q(b)}$, $(b \ {\rm mod} \ \frak{p)} \in \F_p\subset \sO_{\Q(b)}/\frak{p}$, then 
$b\in \Q$.  In words: if $b$ lies in the prime field in characteristic $p$ for almost all $p$, then it does in characteristic $0$.

\end{prop}
\begin{proof}
After multiplying $b$ by an element  in  $\Z \setminus \{0\}$, we may assume that $b\in \sO_{\Q(b)}$, so generates the subring  $\Z[b]\subset \sO_{\Q(b)}$ which is free of  rank $d$ over $\Z$, which is also the rank of $\sO_{\Q(b)}$ over $\Z$. 
Hence for all but finitely many primes $\frak{p}$ of $\sO_{\Q(b)}$, $\sO_{\Q(b)}/\frak{p}$ is spanned by $(b \  {\rm mod} \  \frak{p})$ over $\F_p$. 
So the condition 
 $(b \ {\rm mod} \ \frak{p)} \in \F_p\subset \sO_{\Q(b)}/\frak{p}$ is equivalent to $(b  \ {\rm mod} \ \frak{p)}$ being completely split.  By  Kronecker's density theorem,  this implies $d=1$.

\end{proof}
\subsubsection{Translation of Kronecker's algebraic criterion in terms of differential equations}
We consider the complex algebraic variety $$X={\rm Spec} \  \sO(X), \  \sO(X)= \C [t, t^{-1}]$$ and on it the linear differential equation
\ga{}{(\star) \  \  \  \frac{df}{f}=b \frac{dt}{t} \notag }
for some $b\in \C$. It has  the analytic solution
\ga{}{f_\lambda(t)=\lambda t^b, \ \lambda\in \C.\notag}
Then 
\begin{lem} \label{lem:alg}
For $\lambda\neq 0$, $f_\lambda$ is algebraic over $\sO(X)$ if and only if $b\in \Q$ if and only if $f_\lambda$ is  integral over $\sO(X)$.
\end{lem}
\begin{proof}
If $ \Q\ni b=\frac{m}{n}, \ m,0\neq n  \in \Z$ then $f_\lambda^n \in \sO(X)$ so $f_\lambda$ is integral over $\sO(X)$. Assume  now $f_\lambda$ is algebraic over $\sO(X)$, so in particular $f_\lambda$  is algebraic over the Laurent power series field $\C((t))$ containing  $\sO(X)$. So $f_\lambda$ defines a finite field extension  $$\C((t))\hookrightarrow \C((t))(f_\lambda).$$
As  the embedding $$\cup_{n\in \N} \C((t^{1/n})) \hookrightarrow\overline{\C((t))}$$ is an equality,  $C((t))(f_\lambda)$ must be one of the fields $\C((t^{1/n})) $ thus $b\in \Q$, which also implies that $f_\lambda$ is integral over $\sO(X)$.

\end{proof}
\begin{rmk} For $\lambda\neq 0$, the following conditions are equivalent:
\begin{itemize}
\item[(i)]  $f_\lambda$ is algebraic over the field of fractions ${\rm Frac}(\sO(X))=\C(t)$;
\item[(ii)] $f_\lambda$ is integral  over $\sO(X)$;
\item[(iii)] the monodromy of $(\star)$ is finite. 
\end{itemize}
\begin{proof}
(i) $\Longrightarrow$ (ii):  $f_\lambda$,  as a solution of $(\star)$ is analytic on $X$, and by the condition  (i), $f_\lambda$  lies in $\overline{\C(t)}$. Thus  $f_\lambda$ lies in the integral closure $\overline{\sO(X)} \subset \overline{ \C(t)}$.\\
(ii) $\Longrightarrow$ (iii):  Recall that by Lemma~\ref{lem:alg}, $f_\lambda=\lambda t^b$.  As $f_\lambda$ is integral over $\sO(X)$, it is integral over $\C((t))$. So by Kummer theory, $b\in \Q$. 
On the other hand,   the restriction $\gamma^* f_\lambda$ of $f_\lambda$ to  the path $\big( \gamma: [0  \ 1] \to , \tau \mapsto {\rm exp} (2\pi \sqrt{-1} \tau)\big)$ is the function $\big([0 \  1]\to \C, \ \tau \mapsto \lambda {\rm exp}(2\pi \sqrt{-1} b\tau)\big )$. So the monodromy $a\in \C^\times$ on $\gamma$   is computed as  the ratio $$a=\gamma^*f_\lambda(\tau=1)/\gamma^*f_\lambda(\tau=0)={\rm exp}(2\pi \sqrt{-1}b),$$ 
 which is a root of unity. \\
(iii) $\Longrightarrow$ (i): If $a$ is a root of unity, then $b$ is a rational number so (i) holds.

\end{proof}

\end{rmk}
We now assume $b\in \bar \Q$. For almost all  primes $\frak{p}$ of $\Q(b)$ where it makes sense,  that is for which $b$ is integral in the  corresponding $p$-adic field, we consider the differential equation 
\ga{}{ (\star)_{\frak{p}}  \   \  \  \frac{df}{f}=(b \ {\rm mod} \ \frak{p}) \frac{dt}{t} \notag  }
 which is simply $(\star)$ by viewed over $\sO_{\Q(b)}/\frak{p}$.  We write $\sO_\Z(X)=\Z[t, t^{-1}]$ so $\sO(X)=\sO_\Z\otimes_{\Z}\C$. The only way to make sense of the  solutions 
 $$\lambda t^b \ {\rm  for} \  \lambda \in \sO_{\Q(b)}/\frak{p}$$ is to request them to lie
 in $ \sO_\Z(X)\otimes_{\Z} ( \sO_{\Q(b)}/\frak{p})$. For $\lambda\neq 0$ in  $\sO_{\Q(b)}/\frak{p}$,
  $\lambda t^b$ lies in  $ \sO_\Z(X)\otimes_{\Z} ( \sO_{\Q(b)}/\frak{p})$
  if and only if $(b \ {\rm mod} \ \frak{p}) $ lies in $\F_p\subset \sO_{\Q(b)}/\frak{p}$.
  
  \medskip
  
 We  now  summarize the discussion.
 \begin{thm} \label{thm:gr_gm}
 If $b\in \bar \Q$, the differential equation $(\star)$ has  one non-trivial 
  algebraic solution over $\sO(X)$  if and only if  for almost all primes $\frak{p}\in \sO_{\Q(b)}$, 
 $(\star)_{\frak{p}}$  has one non-trivial  solution  in $\sO_\Z(X)\otimes_{\Z}  (\sO_{\Q(b)}/\frak{p})$.  
 
 \end{thm}
 \begin{rmk}
 The last condition, to have  one non-trivial solution, or equivalently, as the differential equation has rank $1$, to have a full set of solutions in  $\sO_\Z(X)\otimes_{\Z}  (\sO_{\Q(b)}/\frak{p})$, which again   by definition  is equivalent to saying that 
 the  $p$-curvature of the differential equation  $(\star)_{\frak{p}}$ vanishes.  We do not give  here the definition of the $p$-curvature itself, as  we do not need 
 it as such, we only need what it means for it to vanish  (or later in the Lecture  to be nilpotent). 
  
 \end{rmk}
  \begin{proof}[Proof of Theorem~\ref{thm:gr_gm}]  
   $(\star)$ has a full set of algebraic solutions over $\sO(X)$  
   \begin{center}  if and only if \end{center}
    $b\in \Q$  (Lemma~\ref{lem:alg})
    \begin{center}  if and only if  \end{center}
    $(b \ {\rm mod} \ \frak{p}) \in \F_p\subset (\sO_{\Q(b)}/\frak{p})$ for almost all $p$ (Proposition~\ref{prop:Kronecker_algebraic}) 
    \begin{center} if and only if  \end{center}
    $(\star)_{\frak{p}}$ has a full set of solutions in $\sO_\Z(X)\otimes_{\Z}  (\sO_{\Q(b)}/\frak{p})$ (previous discussion). 
  
  \end{proof}
The content of Grothendieck's $p$-curvature  conjecture is to predict that Theorem~\ref{thm:gr_gm} extends to any  smooth complex quasi-projective variety $X$ and any linear differential equation  $(\star)$. We shall explain the formulation and a generalization of it without ever mentioning the definition of the $p$-curvature.

\subsection{Grothendieck's $p$-curvature conjecture}
  Let us first formulate what {\it a posteriori} is equivalent to the $p$-curvature conjecture.
  
  \medskip
  
   We write $X={\rm Spec} \sO(X), \ X={\rm Spec} \C[t,t^{-1}, (t-1)^{-1}]$. Then we pose the linear differential 
equation
\ga{}{(\star) \  \ \  \frac{df}{f}=b \frac{dt}{t} + c \frac{dt}{t-1} \notag }
where now $$b=(b_{ij}), \ c=(c_{ij}) \in M_r(\bar \Q).$$ For almost all  primes $\frak{p}$ of $\Q(b_{ij}, c_{ij}), \  1\le i,j\le r$ where it makes sense, we consider the differential equation 
\ga{}{ (\star)_{\frak{p}}  \   \   \
\frac{df}{f}=(b \ {\rm mod} \ \frak{p}) \frac{dt}{t} + (c \ {\rm mod} \ \frak{p}) \frac{dt}{t-1} \notag  }
 which is simply $(\star)$ but viewed over $\sO_{\Q(b_{ij})}/\frak{p}$.  We write $\sO_\Z(X)=\Z[t, t^{-1}, (t-1)^{-1}]$ so $\sO(X)=\sO_\Z(X)\otimes_{\Z}\C$.
 
 \medskip
 
 We say that $(\star)_p$ has a full set of solutions if it has a set of $r$ solutions 
 $$(f_1,\ldots, f_r), \ f_i \in  \sO_\Z(X)\otimes_{\Z}  (\sO_{\Q(b_{ij})}/\frak{p}),$$ which are linearly independent over $ \sO_{\Q(b_{ij} )} /\frak{p}.$
 Similarly we say that 
  $(\star)$ has a full set of algebraic solutions if it has a set of $r$ solutions 
 $(f_1,\ldots, f_r)$ where $f_i$ is algebraic over $\sO(X)$, and  which are linearly independent over  $\C$.

 \begin{conj} \label{conj:pcurvP}
 The differential equation $(\star)$ has a full set of algebraic solutions over $\sO(X)$   if and only if $(\star)_{\frak{p}}$  has a full set of solutions in 
 $\sO_\Z(X)\otimes_{\Z}  (\sO_{\Q(b_{ij}})/\frak{p})$  (that is its $p$-curvature vanishes).
 \end{conj}
 The initial Grothendieck's conjecture is documented in Katz's article \cite[Introduction]{Kat72}: let $X$ be a smooth quasi-projective variety over $\C$, $(E,\nabla)$ be an integrable algebraic connection. Let $S$ be a scheme of finite type over $\Z$ so $(X,(E,\nabla))$ descends to $(X_S, (E_S,\nabla_S))$ and has good reduction at all closed points $s\in S$. Then the prediction is the following conjecture.
 \begin{conj}[Grothendieck's $p$-curvature conjecture] \label{conj:pcurv}
 $(E,\nabla)$ has a full set of algebraic solutions if and only there is a dense open $S^\circ \subset S$ such that for all $s\in S^\circ$ the reduction $(E_s,\nabla_s)$ on $X_s$   has a full set of solutions  (that is its $p$-curvature vanishes). 
 \end{conj}
We  explain briefly why  Conjecture~\ref{conj:pcurv} is equivalent to its special case Conjecture~\ref{conj:pcurvP}. 

\begin{claim}[\cite{Kat70}, Theorem~13.0] \label{claim:res}
 If there is a dense open $S^\circ \subset S$ such that for all $s\in S^\circ$ the reduction $(E_s,\nabla_s)$  has a full set of solutions, then $(E,\nabla)$ is regular singular and has finite monodromies at infinity.

\end{claim}
\begin{claim} \label{claim:cover}

$(E,\nabla)$ has a full set of algebraic solutions if and only if there is a finite  cover $h':  Y'\to X$ such that $h^{'*}(E,\nabla)$ has a full set of solutions, if and only if there is a  finite  unramified 
cover $h: Y\to X$ such that $h^{*}(E,\nabla)$ has a full set of solutions.

\begin{proof}
If  $h$ exists, the solutions of ($E,\nabla)$  in $\sO(Y)$ are algebraic solutions over $\sO(X)$. Vice-versa, assume given a basis over $\sO(X)$ of algebraic solutions $f_i$ say. The ring $\sO(X)\hookrightarrow \sO(X)[f_i]$ is then finite, thus defines a finite possibly ramified finite cover $h': Y'\to X$ so $h^{'*}(E,\nabla)$ is trivial. So the underlying monodromy representation $\pi_1(X)\to GL_r(\C)$ of $(E,\nabla)$ trivializes on $h'_*\pi_1(Y')$ in $\pi_1(X)$, which has finite index. This defines a finite unramified cover $h:Y\to X$ with $\pi_1(Y)=h'_*\pi_1(Y')$ such that $h^*(E,\nabla)$ is trivial.

\end{proof}

\end{claim}
\begin{claim}
Conjecture~\ref{conj:pcurv} in general is equivalent to its special case when $X$ is a smooth projective curve.
\end{claim}
\begin{proof}
By Claim~\ref{claim:cover}, Conjecture~\ref{conj:pcurv} is true for $(E,\nabla)$  on $X$ if and only if it is for $h^{'*}(E,\nabla)$ for any $h': Y'\to X$ finite cover. So by Claim~\ref{claim:res} we may assume that $(E,\nabla)$ extends to a smooth projective compactification of $X$. Then the Lefschetz hyperplane theorem reduces it to the case of a  smooth projective curve.

\end{proof}
\begin{claim}[\cite{And04}, Th\'eor\`eme 0.6.1]
 \label{claim:yves}
Conjecture~\ref{conj:pcurv}  is equivalent to its special case when $X$ is a smooth projective curve
defined over a number field. 
\end{claim}
In other words,  let  $f: X\to S$ be a smooth morphism between smooth quasi-projective varieties defined over a number field, let $(E,\nabla)$ be a relative flat connection on $X$. If for all closed points $s\in S$, 
the monodromy of  the restriction $(E,\nabla)|_{X_s} $  of $(E,\nabla)$ to the fibre $X_s$ is finite, that is  if there is a finite \'etale cover $h_s: Y_s\to X_s$ such that $h_s^*(E,\nabla)|_{X_s}$ is trivial, then the monodromy of $(E,\nabla)$ restricted to a geometric generic fibre $X_{\bar \eta}$ is finite as well. We remark that the arguments can be adapted to prove a characteristic $p>0$ version of the same theorem (\cite[Theorem~5.1]{EL13}).

\medskip

Applying Belyi's theorem~\cite{Bel80}, we formulate the following claim. 
\begin{claim}
Conjecture~\ref{conj:pcurv}  is equivalent to its special case when $X$ is $\P^1\setminus \{0,1,\infty\}$, that is Conjecture~\ref{conj:pcurvP}.
defined over a number field. 
\end{claim}

\subsection{Back to Kronecker's analytic criterion: Gau{\ss}-Manin connections}
Let $X$ be a quasi-projective smooth variety over $\C$, $g: Y\to X$ be a smooth projective morphism, $\mathbb L$ being the local system $R^ig_*\C$ for some $i$. Then
\begin{thm}[\cite{Kat72}] 
Conjecture~\ref{conj:pcurv} holds for $\mathbb L$. 
\end{thm}
\begin{proof}
The proof consists of  two parts. First, the assumption on the full set of solutions for $(\star)_{\frak{p}}$ implies that the $F$-filtration
$R^ig_*\Omega^{\ge j}_{Y/X} \hookrightarrow   R^ig_*\Omega^{\bullet}_{Y/X}  $ is stabilized by the Gau{\ss}-Manin connection. Then this implies that the monodromy is an automorphism of the integral polarized  Hodge structure of the fibre, which is the intersection of the integral group  ${\rm Aut}(H^i(Y_x, \Z))$ with ${\rm Aut}(H^i(Y_x, \R), (2\pi \sqrt{-1}^i (x, Cy))$ inside of ${\rm Aut}(H^i(Y_x, \R))$, where $C$ is the Weil operator.  A variant of Kronecker's analytic criterion says that this intersection  is finite. 

\medskip

So the stabilization of the $F$-filtration by the connection is the main point. We give a presentation which slightly differs from {\it loc. cit.}
To this, over our model $X_S$, we take a closed point  $s$ of large characteristic   $p$  so the coherent sheaves $R^{i-j}g_*\Omega^j_{Y/X}$ remain locally free, and the  map 
\ga{}{ (*) \  \  \ R^ig_*\Omega^{\ge j}_{Y/X}\to R^ig_*\Omega^\bullet_{Y/X},\notag}  which by Hodge theory is injective, remains injective by base change.  We want to prove that the condition on the existence of a full set of solutions implies that the $F$-filtration is stabilized by the connection, that is that the $\sO_X$-linear Kodaira-Spencer class 
\ga{}{ KS: R^{i-j}g_*\Omega^j_{Y/X}\xrightarrow{\nabla}  \Omega^1_X \otimes R^{i+1-j}g_*(\Omega^{j-1}_{Y/X}\to \Omega^j_{Y/X})\to
 \Omega^1_X\otimes R^{i+1-j}g_*\Omega^{j-1}_{Y/X}\notag}
dies.
The important diagram is then
 \ga{}{\xymatrix{ 
\ar[dr]_g Y_s\ar[r]^F & \ar[d]^{g'} Y_s' \ar[r]^\sigma& Y_s \ar[d]^g\\
 & X_s \ar[r]_{F_{X_s}} & X_s
   } \notag}
where $F: Y_s\to Y'_s$  is the  Frobenius on $Y_s$ relative to $X_s$,  that is on local  functions $\sO(Y_s)=\sO(X_s)[y_1,\ldots, y_n]/I$ it sends the class of 
 $f(y_1,\ldots, y_n)$ to the class of $f(y^p_1,\ldots, y^p_n)$, and $F_{X_s}: X_s\to X_s$ is the absolute Frobenius, so it sends $f\in \sO(X_s)$ to $f^p\in \sO(X_s)$.  
The injectivity of  $(*) $ implies that 
\ga{}{ (**) \ \ \ R^ig_*\tau_{\le j}  \Omega^\bullet_{Y_s/X_s}=R^ig'_*\tau_{\le j} F_* \Omega^\bullet_{Y_s/X_s}  \to R^ig'_*F_* \Omega^\bullet_{Y_s/X_s}
\notag}
is injective as well. By the Cartier isomorphism and base change for $F_{X_s}$, it holds
\ga{}{ R^{i-j} g_*\sH^j =R^{i-j} g'_* \Omega^{j}_{Y'_s/X_s}=F_{X_s}^*  R^{i-j} g_* \Omega^{j}_{Y_s/X_s},     \notag}
where $\sH^j$ is the $j$-th cohomology sheaf of  $F_* \Omega^\bullet_{Y_s/X_s}  $. 
Computing the dimension of de Rham  cohomology $H^i_{dR}(Y_x)$  of the fibres of $g$
yields  that the following  sequence of connections
\ml{}{ (\epsilon) \ \ \ 0\to F_{X_s}^* R^{i-j}g_*  \Omega^{j }_{Y_s/X_s} \to R^ig_*(\tau_{\le j+1}F_*  \Omega^\bullet_{Y_s/X_s}/\tau_{\le j-1}  F_*\Omega^\bullet_{Y_s/X_s}) \\
\to  F_{X_s}^* R^{i-j-1} g_* \Omega^{j+1}_{Y_s/X_s}\to 0 \notag}
is exact. It defines a de Rham extension 
\ga{}{\epsilon \in H^1_{dR}(X_s,  F_{X_s}^* R^{i-j-1} g_*(\Omega^{j+1}_{Y_s/X_s})^\vee \otimes  F_{X_s}^* R^{i-j}g_*  \Omega^{j }_{Y_s/X_s}).\notag}
Now we use the assumption which is equivalent to saying that  the exact sequence  $(\epsilon)$  of connections  is  equal to $F^*_{X_s}$ of an exact sequence
\ga{}{ 0\to R^{i-j}g_*  \Omega^{j }_{Y_s/X_s} \to 
V\to   R^{i-j-1} g_*\Omega^{j+1}_{Y_s/X_s}\to 0 \notag}
of vector bundles on $X_s$, endowed with the canonical connection.  In particular this sequence splits on a dense affine open of $X_s$.  Thus 
$\epsilon$ dies as well on a dense affine open. On the other hand, the cohomology
\ml{}{ H^1_{dR}(X_s, F_{X_s}^*  R^{i-j-1}g_* (\Omega^{j+1}_{Y_s/X_s})^\vee \otimes F_{X_s}^*R^{i-j}g_*  \Omega^{j }_{Y_s/X_s}) =\\
H^1_{dR}(X_s,  ( R^{i-j-1}g_* (\Omega^{j+1}_{Y_s/X_s})^\vee \otimes  R^{i-j}g_*  \Omega^{j }_{Y_s/X_s})\otimes F_{X_s *}(\Omega^\bullet_{X_s})\notag)}
maps, again via the Cartier isomorphism, to 
\ga{}{ H^0(X_s,  ( R^{i-j-1} g_*(\Omega^{j+1}_{Y_s/X_s})^\vee \otimes R^{i-j}g_*  \Omega^{j }_{Y_s/X_s}) \otimes \Omega^1_{X_s}).\notag}
Katz \cite[Theorem~3.2]{Kat72} (in a slightly different presentation) computes that the image of $\epsilon$, viewed as a $\sO_{X_s}$-linear map 
\ga{}{  R^{i-j-1} g_*(\Omega^{j+1}_{Y_s/X_s})\to \Omega^1_{X_s} \otimes R^{i-j}g_*  \Omega^{j }_{Y_s/X_s} \notag}
is precsiely the Kodaira-Spencer map $KS_s$  for $g$.  Thus $KS_s$ dies on a dense open of $X_s$.  As  $KS_s$   is a  $\sO_{X_s}$-linear morphism between vector bundles, we conclude that it dies on the whole of $X_s$. 
Thus $KS$ dies in restriction to all closed points in of a dense open in
 $S$. Thus it dies. This finishes the proof.

\end{proof}
The works by Chudnovsky-Chudnovsky \cite{Chu85}, Bost  \cite{Bos01} and Andr\'e  \cite{And04} handle the solvable case. 
There are also some remarks like \cite{EK18} under a stronger condition than just generation of the differential equation on characteristic $p$.
Else, there is a different  idea in Ananth Shankar's work \cite{Sha18} who showed finiteness of the monodromy on simple loops, which enabled him and Patel-Whang to finish the proof in rank $2$ in \cite{PSW21}. 
Beyond this,   there is essentially  no big progress on this viewpoint. 

\medskip

However, 
 there are two more points. The proof above using the stabilization of the $\tau$- filtration (also called {\it conjugate filtration}) suggests the generalization: 
 \begin{conj}[Generalized Grothendieck's $p$-curvature conjecture, \cite{And89}, Appendix to Ch. V] To characterize geometric local systems $\mathbb L$,  we  request that the underlying vector bundle with a connection $(E,\nabla) $ has the property that  mod $p$ for all large $p$ it is filtered so that the associated graded is spanned by flat sections 
  (that is its $p$-curvature is nilpotent)  on a dense open of $X$ mod $p$.
\end{conj}

Said differently, Katz proves:  if mod $p$ for all large $p$, $(E,\nabla)$ is  filtered so that the associated graded is spanned by flat sections, and in addition $(E,\nabla)$ comes from geometry, and furthermore   this filtration is trivial, then $(E,\nabla)$ has finite monodromy.

\medskip 

There is one other instance 
of a similar situation.  Let $X$ be a smooth projective variety defined over $\C$ and let us assume that $H^0(X, {\rm Sym}^\bullet \Omega^1_X)=0$. I had conjectured that then all local systems have finite monodromy. This has been proved in \cite{BKT13} by {\it analytic methods} as a corollary of Zuo's theorem \cite{Zuo00} to the effect that if the derivative of the period map associated to  an integral  variation of Hodge structure is injective, then $\Omega^1_X$ is big. In \cite{BKT13} one finds a large list of examples of such varieties. 

\medskip

Algebraically, the vanishing $H^0(X, {\rm Sym}^\bullet \Omega^1_X)=0$ implies that the underlying bundles with connection $(E,\nabla)$  have the property that mod $p$ for all large $p$ they are filtered so that the associated graded is spanned by flat sections. 
So according to the generalized Grothendieck  $p$-curvature  conjecture they should be geometric. But here unlike for Katz, we do not know a priori the geometricity, it comes ultimately  as a consequence of the more difficult finiteness result which the authors prove  in {\it loc. cit.}

\newpage

\section{ Lecture 3:   Mal\v{c}ev-Grothendieck's Theorem,  its Variants in Characteristic $p>0$, Gieseker's Conjecture, de Jong's Conjecture, and the One to Come } \label{sec:Malcev}
\begin{abstract}  We recall Mal\v{c}ev's  theorem to the effect that the triviality of the profinite completion of a finitely generated group implies the triviality its algebraic completion. We recall Grothendieck's version of it formulated with $\sD$-modules using the Riemann-Hilbert correspondence, then the Gieseker conjecture, its pendant in characteristic  $p>0$, its solution and generalizations.  \end{abstract}

\subsection{Over $\C$}
Let $\Gamma$ be a finitely generated group, $\hat \Gamma $ be its profinite completion, $\Gamma^{\rm alg}$ be its proalgebraic completion. So $\hat \Gamma= \varprojlim_{\{G\to Q\}} Q$ where $Q$ runs through all finite quotients, $\Gamma^{\rm alg}= \varprojlim_{\{G\to Q\}} Q$ where $Q$ runs through all  algebraic subgroups of $GL_r(\C)$ for some $r\in \N_{>0}$.
\begin{thm}[Mal\v{c}ev \cite{Mal40}]
$\hat \Gamma=\{1\} \Longrightarrow \Gamma^{\rm alg} =\{1\}$.
\end{thm}
\begin{proof}
Indeed, if $\rho: \Gamma \to GL_r(\C)$ is a representation, then $Q={\rm Im}(\rho)$ lies in $GL_r(A)$ where $A$ is a ring of finite type over $\Z$.  Let $\frak{m}\subset A$ be a maximal ideal, so with finite residue field $A/\frak{m}$, and define $\iota: A\hookrightarrow \hat A=\varprojlim_{n\in \N} A/\frak{m}^n$ to be the completion of $A$ at $\frak{m}$.  Then $\rho$ is non-trivial if and only if $\iota\circ \rho: \Gamma\to GL_r(\hat A)$ is non-trivial. As $\iota\circ \rho$ factors through $\hat \rho: \hat \Gamma\to GL_r(\hat A)$, then  $\hat \Gamma=\{1\} \Longrightarrow \rho$ is trivial, i.e. has image equal to the identity matrix  $\mathbb I_r \in GL_r(A)$.

\end{proof}
\begin{rmk}
If $\Gamma$ is not finitely generated,  the implication of the theorem might be  wrong.  For example, let $\Gamma=\Q$, then $\hat \Gamma=\{1\}$, but e.g. the representations
$\chi: \Q\to GL_1(\C), \ n \mapsto {\rm exp}(2\pi \sqrt{-1} a n)$ for some $a\in \C$ or 
 $\rho: \Q\to GL_2(\C),   \   n\mapsto {  \begin{pmatrix}  1 & n\\
0 & 1 \end{pmatrix}}$  are non-trivial.

\end{rmk}
We now assume that $X$ is a smooth quasi-projective variety defined over $\C$. Then the Riemann-Hilbert correspondence \cite{Del70}  is an equivalence between complex local systems  and algebraic integrable connections on $X$ which are regular singular at infinity (\cite{Del70}), or equivalently $\sO_X$-coherent regular singular  $\sD_X$-modules.   On the other hand, the topological fundamental group $\pi_1(X(\C), x)$ based at some complex point $x$  is finitely generated, even finitely presented but this is not used here in the discussion, and by the Riemann Existence Theorem, its profinite completion is Grothendieck's \'etale fundamental group $\pi_1(X_\C,x_\C)$. 
So Mal\v{c}ev's theorem can be rephrased in the context of complex algebraic geometry  by saying
\begin{thm}[Grothendieck \cite{Gro70}] 
 $\pi_1(X_\C,x_\C)=\{1\} \Longrightarrow$ there are no non-trivial $\sO_X$-coherent regular singular $\sD_X$-modules on $X$. 
\end{thm}

\subsection{Over an algebraically closed field of characteristic $p>0$}
I am not aware of a direct proof of Grothendieck's theorem without passing through the Riemann-Hilbert correspondence.  Gieseker \cite[p.8]{Gie75} conjectured the '`same'' theorem over an algebraically closed field of characteristic $p>0$. 
\begin{thm}[\cite{EM10}, Theorem~1.1] \label{thm:EM10}
Let $X$ be a smooth projective variety defined over an algebraically closed field $k$ of characteristic $p>0$. If $\pi_1(X)=\{1\}$, there are no non-trivial $\sO_X$-coherent $\sD_X$-modules. 
\end{thm}
\begin{proof}
Following Katz \cite[Theorem~1.3]{Gie75}, we first identify   $\sO_X$-coherent $\sD_X$-modules with Frobenius divided sheaves.
Then the strategy consists, starting from a non-trivial  Frobenius divided sheaf $E=E_0={\rm Frob}^*E_1, E_1={\rm Frob}^*E_2,\ldots$  of rank $r$, to find a non-trivial vector bundle $M$ of rank $r$  with the property that
 $({\rm Frob}^n)^*(M)\cong M$ for some $n\in \N_{>0}$. Indeed the underlying Lang torsor  under $GL_r(\F_{p^n})$  is finite \'etale over $X$, thus 
has to be trivial by assumption, so $M$ has to be trivial, a contradiction.  To this aim, we observe that the Frobenius divisibility forces all kinds of Chern classes of $E$  to be trivial and $E_i$ itself to be stable for $i$ large if the object $(E_0,E_1,E_2,\ldots)$ itself is simple. So we consider in the moduli of vector bundles of rank $r$ with vanishing Chern classes the Zariski closure of the locus of the $E_n$ for $n$ large, on which Frobenius is a rational dominant map. We can then specialize the situation to $X$ over $\bar \F_p$ and apply Hrushovski's theorem  to the effect that there are (many) preperiodic points. By the above discussion, those yield non-trivial Lang torsors of $X_{\bar \F_p}$, thus by Grothendieck's specialization homomorphism, also a non-trivial  finite \'etale cover of $X$, a contradiction. 
\end{proof}
\begin{rmks}
1) Hrushovski's proof \cite{Hru04}, which relies on model theory, has been meanwhile proven purely in the framework of arithmetic geometry by Varshavski, see \cite{Var18}.\\ \ \\
2) The core of the proof is to find such an $M$ with $({\rm Frob}^n)^*(M)\cong M$ for some $n\in \N_{>0}$.
This unfortunately does not seem to be deductible from Katz' 'Riemann-Hilbert' theorem in characteristic $p>0$ or its (so far) ultimate generalization \cite{BL17}. So in the present state of understanding, we cannot draw a parallel with the complex proof. \\ \ \\

\end{rmks}
The category of  $\sO_X$-coherent $\sD_X$-modules is  equivalent to the one of  crystals in the infinitesimal site, as developed  by Grothendieck. This led J. de Jong to pose the following vast strengthening of Gieseker conjecture (Theorem~\ref{thm:EM10}):
\begin{conj}[de Jong 2010, see \cite{ES18}]
Let $X$ be a smooth projective variety defined over an algebraically closed field $k$ of characteristic $p>0$. If $\pi_1(X)=\{1\}$, there are no non-trivial isocrystals.

\end{conj}
There are (very) partial results on the conjecture, see notably \cite[Theorem~1.2]{ES18}. The main obstruction to our understanding lies in the  prefix ``iso''. All approaches I know start with the mod $p$ reduction of a crystal associated to an isocrystal and try
\begin{itemize}
\item[$\bullet$] to use again an argument as in the proof of Theorem~\ref{thm:EM10} on the existence of $M$, with some sharpening (see  \cite[Proposition~3.2]{ES18}), so ultimately it relies on Hrushovski's theorem; this also requests  
vanishing results for Chern classes of the mod $p$ reduction of crystals, see \cite{ES19}, which have been generalized by Bhatt-Lurie (unpublished);
\\
\item[$\bullet$] so to trivialize the crystal itself  (see e.g. \cite[Corollary~3.8]{ES18}), not only  the isocrystal. 
\end{itemize}
In rank $1$ it is irrelevant  (see  \cite[Lemma~2.12]{ES18}), not however in higher rank. One would wish to find a Frobenius invariant isocrystal on the scratch, without passing through randomly chosen underlying crystals. However Hrushovski's theorem is no longer of help.

\subsection{Over a $p$-adic field} \label{sec:padic}
The wish would be to have a version of de Jong's conjecture for prismatic ``iso''-crystals. Here one problem is that 
the notion of prismatic crystals is documented (\cite{BS21}), not however the one of prismatic isocrystals. One question is what one has to invert to make the problem meaningful.  The formulation should also be compatible with a complex embedding of the $p$-adic field and the initial Grothendieck formulation over $\C$. 

\newpage

\section{Lecture 4: Interlude on some Similarity  between the Fundamental Groups in Characteristic $0$ and $p>0$} \label{sec:similarity}
 \begin{abstract}  We describe two theorems concerning the structure of the (tame) fundamental group of a smooth (quasi)-projective variety in characteristic $p>0$. The first  one \cite{ESS22}  shows a similarity with the complex situation concerning the finite presentation and ultimately relies on a theorem of Lubotzky \cite{Lub01}, the second one \cite{ESS22b} yields a new obstruction for a smooth projective variety in characteristic $p>0$ to lift to characteristic $0$
 and ultimately relies on Grothendieck's specialization homomorphism. In particular, and interestingly, the second one shows that we cannot explain the finite presentation of the first one by a lifting argument.  In Lecture~\ref{sec:similarity}  we treat the finite presentation, in Lecture~\ref{sec:diff}  the obstruction. 
  \end{abstract}
  
  \subsection{Lubotzky's theorem} 
 \begin{defn}[\cite{Lub01}, p.1 of the Introduction]
  A profinite group $\pi$ is said to be 
  \begin{itemize}
  \item[1)] finitely generated if there is a  free profinite group $F$ on finitely many generators and a continuous surjection $F\xrightarrow{\epsilon} \pi$;
  \item[2)] finitely presented if it is finitely generated and ${\rm Ker}(\epsilon)$ is finitely generated as a closed normal subgroup of $F$. 
  \end{itemize}
  \end{defn}
  Concretely, in 1) there are finitely many elements $f_i, i=1,\ldots, N$ in $F$ such that the abstract subgroup $\langle f_1,\ldots, f_N\rangle \subset F$ surjects to 
 any finite quotient  $\pi \to G$ in the profinite structure of $\pi$. For 2) there are finitely many elements $r_j,  j=1, \ldots, M$ in $F$ such that ${\rm Ker} (F\to \pi)$ is spanned by $\langle fr_jf^{-1}, \ j=1,\ldots, M, f\in F\rangle$.  Finite generation and presentation are well defined, they do not depend on the chosen presentation. 
 For the finite generation,  it is very classical, for the finite presentation, a slick way  is simply to apply the following theorem. 
 \begin{thm}[Lubotzky \cite{Lub01}, Theorem~0.3] \label{thm:Lubotzky}
 $\pi$ is finitely presented if and only if it is finitely generated and in addition there is a constant $C\in \R_{>0}$, such that   for any  $r\in \N_{>0}$, for  any  prime number $\ell$, for any continuous linear representation $\rho: \pi\to GL_r(\F_\ell)$, it holds
 $${\rm dim}_{\F_\ell} H^2(\pi, \rho)\le C\cdot r.$$
 \end{thm}
 \begin{ex} \label{ex:fp}
 The standard example is given by $\pi$ being the profinite completion of an abstract finitely presented group, in the classical sense, so with the same definition without topology. 
 \end{ex}
 \subsection{Tame fundamental group} \label{sec:tame}
 Let $X$ be a regular connected scheme.
 For $x$ a closed point of the normal compactification $\bar X$ of $X$, we denote by $k(X)_x$ the
 field of fractions of the completed local ring of $\bar X$ at $x$. 
  For a  connected finite \'etale cover $\pi: Y\to X$, we denote by $\bar Y$ the normalization of $\bar X$ in the field of functions on $Y$. 
  Any closed point $x\in \bar X$  defines for any closed point $y\in \bar Y$ above $x$  a finite field extension $ k(X)_x\hookrightarrow k(Y)_y$.
  
   \begin{defn}[\cite{KS10}, Introduction and Theorem~1.1]
 1) If $X$ has dimension $1$, 
 a finite \'etale cover $\pi: Y\to X$ is tame if and only if all  the field extensions $k(X)_x \hookrightarrow k(Y)_y$ are tame, that is the indexes of ramification are prime to $p$ and the residual extension $k(x)\hookrightarrow k(y)$ are separable. \\
 2) In higher dimension, $\pi$ is tame if and only if for any morphism of a normal dimension $1$ scheme  $C\to X$, the  induced morphism $C\times_XY\to C$ is tame on all the irreducible components. 
  \end{defn}
  \subsection{Grothendieck's specialization homomorphism, \cite{SGA1}, Expos\'e X, Expos\'e XIII} \label{sec:sp}
  Let $X_S\to S$ be a smooth  morphism, where $S$ is any scheme. We consider two field value points ${\rm Spec}(F)\to S$ and ${\rm Spec} (k)\xrightarrow{s} S$ with the property that ${\rm Spec}(k)$ lies in the Zariski closure of ${\rm Spec}(F)$. So there is an irreducible subscheme $Z\subset S$ such that $s\in Z$ is a point and $\sO(Z) \to F$  
is injective.   Let  $\widehat{Z_s}$ is the completion of $Z$ at $s$, and $F\hookrightarrow \hat F_s$ a field extension such that $\hat F_s$  contains $\sO(\widehat{Z_s})$.

  \medskip
  
  If $X$ is not proper, we assume in addition that it admits a compactification $X_S\hookrightarrow \bar X_S$ so $\bar X_S\to S$ is smooth proper, and $\bar X_S\setminus X_S \to S$ is a relative normal crossings divisor with smooth components. We call it a {\it good} compactification (over $S$). 
  So   we have a diagram
\ga{}{  \xymatrix{ {\rm Spec}( \hat F_s) \ar[r] &  \widehat{Z_s} & \ar[l] s
}\notag}
   together with the scheme over it
  \ga{}{  \xymatrix{ \ar[d] X_{\hat F_s} \ar[r] & \ar[d] X_{ \widehat{Z_s}} & \ar[l] \ar[d]X_s\\
   {\rm Spec}(\hat F_s) \ar[r] &  \widehat{Z_s} & \ar[l] s
}\notag}
  We denote by $\overline{\hat F}_s \supset  \hat F_s$ and $\bar k\supset k$  algebraic closures, the latter  defining $\bar s\to s$. 
Then, upon choosing an $S$-point $x_S: S\to X_S$,  one defines \cite[Expos\'e XIII 2.10] {SGA1}   a specialization homomorphism
 \ga{}{sp_{\hat F_s, s}:  \pi_1^t(X_{ {\hat F}_s}, x_{ {\hat F}_s})\to \pi_1^t(X_{ s}, x_{s}) \notag}
 which is the composite of the functoriality homomorphism
 \ga{}{  \pi_1^t(X_{ {\hat F}_s}, x_{ {\hat F}_s})\to 
  \pi_1^t(X_{{ \widehat{Z}_{s}}}, x_{{ s}}) \notag}
   with the inverse of the base change isomorphism
  \ga{}{\pi_1^t(X_{ s})\xrightarrow{\cong}  \pi_1^t(X_{\widehat{Z_{ s}}}). \notag}
  Finally one has the  functoriality homomorphism
  \ga{}{ \pi_1^t(X_{\hat F_s}, x_{\hat F_s}) \to  \pi_1^t(X_F, x_F) \notag}
  which is an isomorphism in restriction to the geometric fundamental groups
    \ga{}{ \pi_1^t(X_{ \overline{\hat F}_s}, x_{ \overline{\hat F}_s}) \xrightarrow{\cong} \pi_1^t(X_{ \bar F}, x_{ \bar F}).\notag}
  Taken together this defines the specialization homomorphism
  \ga{}{sp_{F,s}: \pi_1^t(X_F, x_F)\to \pi_1^t(X_s,x_s)\notag}
  which, while restricted to the geometric fundamental groups, defines the specialization homomorphism
  \ga{}{sp_{\bar F,\bar s}: \pi_1^t(X_{ \bar F}, x_{ \bar F}) \to \pi_1^t(X_{ \bar s}, x_{ \bar s}).\notag}
  The specialization homomorphisms $sp_{F,s}$  and $sp_{\bar F,\bar s}$  are surjective, and 
  $sp_{\bar F,\bar s}$ induces an isomorphism on the pro-$p'$-completion 
  \cite[Expos\'e XIII 2.10,
Corollaire~2.12]{SGA1}.
  
  \subsection{Finite generation}
We apply Grothendieck's specialization homomorphism in the following situation.  
   Let $X$   be smooth projective defined over a characteristic $0$ field together with a good compactification $X\hookrightarrow \bar X$, let $X_S\hookrightarrow \bar X_S$ be a good compactification over $S$ as in Section~\ref{sec:sp}. 
 Upon choosing a $S$-point $x_S: S\to X_S$,  we have the  specialization homomorphism
 \ga{}{sp_{F,s}: \pi_1(X_F, x_F)\to \pi_1^t(X_s, x_s) \notag}
 where $F$ is a geometric point of $X_{k(S)}$ which specializes to a $\bar \F_p$-point  $s$ of $S$. (Beware we slightly change the notation as compared to Section~\ref{sec:sp}
 as we consider only the geometric fundamental groups). 
   Taking $F=\C$ then $\pi_1(X_F, x_F)$ is by Riemann existence theorem 
\cite[Expos\'e XII, Th\'eor\`eme~5.1]{SGA1} the profinite completion of $\pi_1(X(\C), x(\C))$, the latter being itself is an {\it abstract finitely presented group}, see \cite[Theorem~1.1]{Esn17}. Thus $\pi_1^t(X_s,x_s)$ is {\it finitely generated}.
 
 \begin{thm}[\cite{ESS22},  Theorem~1.1] \label{thm:ESS22} Let $X$ be a smooth quasi-projective variety defined over an algebraically closed field  $k$ of characteristic $p>0$, with a good compactification $X\hookrightarrow \bar X$. Then $\pi_1^t(X)$ (based at any geometric point $x$) is finitely presented. 
 
 \end{thm}
 \begin{rmk}
 1) The tameness assumption is necessary. If $X$ is not proper, ``in the rule'' $\pi_1(X)$ is not finitely generated.  For example,
  for each natural number $s$, and a choice of $s$ natural numbers $m_i$ prime to $p$ and pairwise different, 
  the Artin-Schreier covers $y^p - y = z^{m_i}$ of the affine line $\mathbb{A}^1_k =\mathrm{Spec}(k[z])$  yield a surjection
$\pi_1(\mathbb{A}^1_k )\surj  \oplus_{i=1}^s \Z/p\Z$.
\\ \ \\
2)  We need the good compactification assumption in the argument but it should not be necessary as the result does not see the good compactification. In addition  we (perhaps ?)  believe in resolution of singularities. We'll mention in the proof which replacement for resolution of singularities we need, which might be called a ``numerical resolution''. H\"ubner and Temkin announced a solution to the problem. 
 \end{rmk}

 \subsection{Proof of Theorem~\ref{thm:ESS22}, the $\ell\neq p$ part.}
 We could call this part of the proof a ``SGA type proof''. All the ingredients are contained there.  We have to prove the linear growth in $r$ of $H^2(\pi_1^t, \rho)$ for a representation $$\rho: \pi_1^t  \to GL_r(\F_\ell)$$ (which is necessarily continuous). Here we simplify the notation $\pi^t_1(X,x)$ to $\pi_1^t$ and
 $\pi_1(X,x)$ to $\pi_1$. We denote by $X^\rho\to X$ the cover defined by ${\rm Ker}(\rho)$ which by assumption factors the universal tame cover $X^t\to X$. 
  We split the universal cover based at $x$
 \ga{}{\xymatrix{ \ar@/^1pc/[rr]^{\pi_1} \tilde X  \ar[r]_K & X^t  \ar[dr] \ar[r]_{\pi_1^t} & X \\
 & & X^\rho \ar[u] }\notag }
 We denote by $M$ the underlying vector space $\oplus_1^r \F_\ell$ of the $\rho$-representation. 
 The Hochshild-Serre spectral sequence yields the exact sequence
 \ga{}{ (H^1(K, \F_\ell)\otimes M)^{\pi^t_1} \to  H^2(\pi^t_1, M)\to H^2(\pi_1, M). \notag}
 By definition there is no $\F_\ell$-abelian quotient of $K$, thus 
 \ga{}{ (\star) \ \ \ 0=H^1(K, \F_\ell)=H^1(K^{ab}, \F_\ell),\notag}
 from which we derive $H^2(\pi_1^t, M)\hookrightarrow H^1(\pi_1, M)$. On the other hand, 
the Hochshild-Serre spectral sequence induces the exact sequence
\ga{}{ (H^1(\tilde X, \F_\ell)\otimes M)^{\pi_1} \to  H^2(\pi_1, M)\to H^2(X, M)\notag}
 and $H^1(\tilde X, \F_\ell)=0$. So $H^2(\pi_1, M)\hookrightarrow H^1(X, M)$ is injective as well. Taken together this yields that the composed linear map
 \ga{}{ H^2(\pi_1^t, M)\hookrightarrow H^2(\pi_1, M)\hookrightarrow H^2(X, M)\notag}
 is injective. 
 We show the existence of the linear bound on the a priori larger group $H^2(X, M)$. 
 
 \medskip
 
 By the Lefschetz theorem \cite[ Theorem 1.1 b)]{EK16}
 we may assume that $X$ is a surface. Using  an ``ample''  curve $C\rightarrow X$  in good position so the compactification of $C$ and the boundary of $X$ form a strict normal crossings divisor and $X\setminus C$ is affine,  purity yields the exact sequence $$H^0(C, M)\to H^2(X, M)\to H^2(X\setminus C, M).$$   As ${\rm dim}_{\F_\ell} H^0(C, M)\le r$,  we  thus may assume that $X$ is affine.  On the other hand, by Deligne's theorem~\cite[Corollaire~2.7]{Ill81}, tameness implies $$\chi(X, \mathbb M)=r \chi(X, \F_\ell)$$  and we always have $$ \chi(X, \F_\ell)=\chi(X, \Q_\ell).$$
 The latter group is independent of $\ell\neq p$ by Deligne's purity  \cite[Th\'eor\`eme~3.3.1]{Del80}. 
 
 \medskip
 
 This   reduces the problem to bounding linearly the growth of $H^1(X, M)$.  From the exact sequence $$0 \to H^1(\pi_1^t, M)\to H^1(\pi, M)\to (H^1(K, \F_\ell)\otimes M)^{\pi_1^t},$$
 again  using $ (\star)$,  we see that $H^1(X, M)$,  which is equal to $H^1(\pi_1, M)$,  is equal to $H^1(\pi_1^t, M)$. As a $1$-cocycle  $\varphi: \pi_1^t\to M$ fulfils 
 $\varphi(ab)= a \varphi(b)+\varphi(a)$,  it holds $${\rm dim}_{\F_\ell} H^1(\pi_1^t, M)\le \delta r$$ where $\delta$ is the number of topological generators of $\pi_1^t$.  This finishes this computation, which in fact can be performed even if we do not have a good model, using alterations \`a  la de Jong-Gabber-Temkin.
 
  \subsection{Proof of Theorem~\ref{thm:ESS22}, the $p$ part.}
  We denote by $j: X\hookrightarrow \bar X$ the good compactification. We denote by   $\underline{M}$  the local system associated to $M$. 
The key point is to relate $H^2(\pi_1^t, M)$ to a cohomology group which comes from cohomology of some extension on $\bar X$  of a local system $\underline{M}$ on $X$. Being able to use local systems rather than abstract representations yields more flexibility as we have at disposal  the cohomology of all kinds of constructible extensions of the local systems as opposed to just group cohomology. 
 \begin{thm} \label{thm:groupj}
There is an injective $\F_p$-linear map $$H^2(\pi_1^t(X), M)\to H^2(\bar X, j_*\underline{M}).$$

\end{thm}
\begin{proof} First let us fix what we have to prove: it holds
\ga{}{H^2(\pi^t_1, M)=\varinjlim_{U \subset \pi^t {\rm open \  normal}} H^2(G_U, M^U), \ G_U=\pi_1^t/U. \notag } 
For each such $U$ 
we define the fibre square
\ga{}{\xymatrix{ \ar[d]_{h_U} X_U \ar[r]^{j_U} & \bar X_U \ar[d]^{\bar h_U}\\
X\ar[r]^j& \bar X
} \notag}
where $h_U$ is the tame Galois cover defined by $ G_U$  and $\bar X_U$ is the normalization of $\bar X$ in the field of functions $k(X_U)$ of $X_U$. 
As $j$ is a normal crossings compactification, $\bar h_U$ is {\it numerically tame} in the sense of \cite[Section~5]{KS10}, see \cite[Theorem~5.4 (a)]{KS10}.  Thus by \cite[Corollaire p.204]{Gro57} one has 
\ga{b}{ R^{>0}   \bar h^{G_U}_{U*} j_{U*} \underline{M}^U=0 , \notag \\
H^{n}(\bar X_U, G_U, j_{U*} \underline{M}^U)=H^n(\bar X, j_*\underline{M}). \notag
}
We conclude that the spectral sequence  converging to  equivariant cohomology 
\ga{}{E_2^{ab}=H^a(G_U, H^b(\bar X_U,j_{U*} \underline{M}^U)) \Rightarrow 
H^{a+b}(\bar X_U, G_U, j_{U*} \underline{M}^U),\notag
 }
 for each $U\subset {\rm Ker} (\rho)$ yields a short exact sequence 
\ga{seq}{ H^0(G_U, H^1(\bar X_U, j_{U*} \underline{M}^U)) \to H^2(G_U, H^0(\bar X_U,j_{U*} \underline{M}^U)) \to H^2(\bar X, j_*\underline{M})\notag}
where 
\ga{}{  H^1(\bar X_U, j_{U*}  \underline{M}^U)^{G_U}= (H^1(\bar X_U, \F_p)\otimes_{\F_p} M)^{G_U}. \notag}
On the other hand, $H^1(\bar X_U, \F_p)\to H^1(X_U, \F_p)$ is injective and 
again by $(\star)$ it holds ${\rm lim}_U H^1(\bar X_U, \F_p) =0$. 

\end{proof}
\begin{rmk}
We see that the main point is that a good normal crossings compactification implies that the tower $h_U$ is numerically tame. The problem is then how, out of a random normal compactification, to construct one which in the tame tower is numerically tame. 
\end{rmk}
Once there, we can cut down again by a  Lefschetz type argument (\cite[Theorem 1.1 b)]{EK16})  to surfaces and on them, using  the Artin-Schreier exact sequence 
\ga{}{0\to j_*\underline{M}\to \sM\xrightarrow{1-F} \sM\to 0\notag}
where $\sM$ is a locally free sheaf on $\bar X$ with restriction to $X$ equal to $\underline{M}\otimes_{\F_p}\sO_X$,
and the classical fact that this sequence remains exact on cohomology, we are  reduced to bounding above ${\rm dim}_kH^2(\bar X, \sM)$ linearly in $r$, which is the same as ${\rm dim}_k H^0(\bar X, \sM^\vee\otimes \omega_{\bar X})$.  We are back to cohomology of coherent sheaves. 

\medskip

 Now the main point is the following. If $h: Y\to X$ is the Galois cover defined by  the monodromy representation of $\underline{M}$,  with Galois group $G$, thus $h^* \underline{M}$ is trivial. Let  $\bar h: \bar Y\to \bar X$   be its compactification as above, then
$\sM=(\bar h_*\sL)^G$, 
 so
$\bar h^*\sM\subset \sL:=H^0( Y, h^* \underline{M})\otimes_{\F_p} \sO_{\bar Y}$ and as a consequence of Abhyankar's lemma,  tameness  implies 
$\sL\otimes \bar h^*\sO_{\bar X}(-D)\subset \bar h^*\sM$ for $D=(\bar X\setminus X)_{\rm red}.$ Those two informations together yield
\ga{}{ \sL\otimes \bar h^*\sO_{\bar X}(-D)\subset \bar h^*\sM \subset \sL\notag}
which is a boundedness statement. 
 This enables us
to conclude  that the restriction map  $$H^0(\bar X, \sM^\vee\otimes \omega_{\bar X}) \to 
H^0(C\cap C', \sM^\vee\otimes \omega_{\bar X})$$ is injective for $C, C'$ generic curves in the linear system of $\sH'=\sH\otimes \omega_{\bar X}(D)$,  where $\sA$ is chosen to be  very ample and we request $\sH'$ to be very ample as well.  Note $\sH$ and $\sH'$ depend only on $X$ and $r$. On the other hand, $C\cap C'$ is the union of $c_2(\sH')$-points so the right hand side of the inequality is equal to $c_1(\sH')^2\cdot r$.   This finishes the proof.

\newpage
 
  \section{Lecture 5: Interlude on some Difference between the Fundamental Groups in Characteristic $0$ and $p>0$} \label{sec:diff}
  \label{sec:ob}
 \begin{abstract} See the Abstract of Lecture~\ref{sec:similarity}: we show  here the existence of an obstruction to lift a smooth (quasi-)projective variety defined over an algebraically closed field $k$ of characteristic $p>0$ to characteristic $0$ which relies purely on the shape of its (tame)  fundamental group.

  \end{abstract}

\subsection{Serre's construction} 
This problem has been addressed for the first time by Serre  \cite{Ser61}. We use the notation of Section~\ref{sec:tame} but assume more concretely that  $k$ is an algebraically closed field of characteristic $p>0$,  $ X_k=:X $ is smooth projective and obtained as follows. There is  a finite Galois \'etale cover $Y \to X$ such that $Y\hookrightarrow \P^n$ is a smooth complete intersection of dimension $\ge 3$. So in particular $\pi_1(Y)=\{1\}$. In addition, the Galois group $G$ 
is the restriction of a linear action $\rho: G\to GL_{n+1}(k)$. Then  \cite[Lemma]{Ser61} shows that if $X$ lifts to $X_R$ with $S={\rm Spec}(R)$, $R$ a noetherian local complete ring, then $\rho$ lifts to $\rho_R: G\to PGL_{n+1}(R)$, which is not possible if the cardinality of the $p$-Sylow subgroups of $G$ is large. 

\subsection{Various obstructions} We mention three major directions of obstructions which have been settled since Serre's work. Clearly this list is not exhaustive. 

\medskip

 Deligne-Illusie in \cite{DI87} proved that  $X$ smooth proper, lifting to  $W_2(k)$, $k$ perfect 
of characteristic $p>0$, has the property that its Hodge to de Rham spectral sequence degenerates in $E_1$. 
This yields an obstruction to lift to $W_2(k)$ as examples for which the spectral sequence does not degenerate were previously known \cite{Ray78}. This  has been the basis of vast  further developments.

\medskip 

 Achinger-Zdanowicz construct in \cite{AZ17}  
 specific   varieties  which are non-liftable to characteristic $0$ by blowing up the graph of Frobenius which is assumed to be non-liftable in a rigid variety with no corner piece of the $F$-filtration.  It is remarkable that their example has cohomology of Tate type. 

\medskip

 van Dobben de Bruyn proved that  if $X\subset C^3$ is a smooth ample divisor where $C$ is a supersingular genus $\ge 2$ curve over $\bar \F_p$, then $X$ does not lift to characteristic $0$, nor does any smooth proper variety which is rationally dominant over $X$, \cite[Theorem.~1]{vDdB21}. The main property  \cite[Theorem.~2]{vDdB21} is that if $X_S\to S$ lifts $X$, and $X$ admits  a  morphism to a smooth projective genus $\ge 2$ curve $C$, then after possibly base changing with  an inseparable cover of $C$, 
  the morphism lifts to characteristic $0$.

 \subsection{ An abstract obstruction to lift to characteristic $0$, based on the structure of the fundamental group}
 
 The Hodge-de Rham obstruction singled out by Deligne-Illusie  is of theoretical nature,  that is  it does not depend on a concrete way to construct the variety. However it  is  {\it not} an obstruction to lift to characteristic $0$; there are schemes which lift to characteristic $0$, yet in a ramified way,  not over $W_2$. A classical example is a supersingular Enriques surface over $k$ algebraically closed of characteristic  $2$, see \cite[Proposition II.7.3.8]{Ill79}. The other obstructions to lift to characteristtic $0$ rely on the construction of the variety.

    \medskip 
    
 The aim of the remaining part of Chapter~\ref{sec:ob}  is to show that there is an essential  difference between the prime to $ p$ 
quotient of the  fundamental group  of varieties in characteristic  $p > 0$  and the one  in characteristic $0$. It 
 provides a {\it conceptual} obstruction.  It is in {\it contrast}  with the similarity we explained in Lecture~\ref{sec:similarity}, where we showed that the  (tame) fundamental group of a smooth (quasi-)projective variety defined over an algebraically closed field of characteristic $p>0$ (admitting a good compactification) is finitely presented, as it is in characteristic $0$. 
 It is also in  contrast  with the foundational theorem by Achinger \cite[Theorem~1.1.1]{Ach17}  after which every connected affine scheme of positive characteristic is a  $K(\pi, 1)$ space for the \'etale topology.  His theorem notably says that affine varieties  over an algebraically closed field in characteristic $p>0$  are  analog to  Artin neighborhoods in characteristic $0$, see  \cite[Expos\'e XI]{SGA4}.
The theorem means precisely  that for any 
  locally constant \'etale sheaf of finite abelian groups $\sF$ on $X$, the homomorphisms
$$ H^i(\pi_1(X, x), \sF_x) \to  H^i(X, \sF ) $$ coming from the Hochshild-serre spectral sequence are  isomorphisms for all $i$. Here  $x\to X$ is a geometric point. It has not  really been documented in the literature, but we could think of Achinger's theorem as the building block of the theory of \'etale cohomology, reducing it to continuous group cohomology. 
\subsection{Main definition}
\begin{defn}[See \cite{ESS22b}, Definition~A]\label{defnABC:pprime}
A profinite group $\pi$ is said to be  $p'$-discretely finitely generated (resp.\ $p'$-discretely finitely presented)   if there is a finitely generated (resp.\ presented)  discrete group $\Gamma$ together with a group homomorphism 
$
\gamma: \Gamma \to \pi
$
such that
\begin{enumerate}
\item the profinite completion 
$\hat \gamma: \hat \Gamma \to \pi$ is surjective;
\item
for any open subgroup $U\subset \pi$ 
with $\Gamma_U := \gamma^{-1}(U)$ 
the restriction 
$\gamma_U: \Gamma_U \to U$ induces a continuous group isomorphism on pro-$p'$-completions 
\[
\gamma_U^{(p')}: \Gamma_U^{(p')}\to U^{(p')}.
\]
\end{enumerate}
\end{defn}
Grothendieck's specialization homomorphism \ref{sec:sp} together with the lifting property
 \cite[Th\'eor\`eme~18.1.2]{EGAIV(4)}  imply that if $\pi=\pi_1(X,x)$ is the fundamental group of  a smooth proper variety $X$, then it is $p'$-discretely finitely presented (in particular it is $p'$-discretely finitely generated), see \cite[Proposition~2.7]{ESS22b}. Property (1) is also true for $\pi_1^t(X,x)$ when there is a good compactification which comes from characteristic $0$. But to check it in the tower as requested in (2)  is more subtle. We thank the referee of \cite{ESS22b} for kindly noticing  this. Nonetheless the property is true, see \cite[Example~2.8]{ESS22b}. 
 
 \begin{ex} \label{ex:fgp}
 Of course if as in Example~\ref{ex:fp}, $\pi$ is the profinite completion of a finitely presented (resp. generated) group, then $\pi$  is $p'$-discretely finitely presented (resp. generated). 
 
 \end{ex}
  
 \begin{thm}[\cite{ESS22b}, Theorem~C]  \label{thm:C}There are smooth projective varieties  $X$ defined over an algebraically closed field $k$ of characteristic $p>0$ such that $\pi_1(X,x)$ is not $p'$-discretely finitely generated. In particular this notion is an obstruction to liftability to characteristic $0$.

 \end{thm}

\begin{rmk} We remark that 
in view of Example~\ref{ex:fgp} and of 
 the finite presentation of Theorem~\ref{thm:ESS22}, Theorem~\ref{thm:C} implies in particular that there  are smooth projective varieties  $X$ defined over an algebraically closed field $k$ of characteristic $p>0$ such that 
 $\pi_1(X,x)$ is not the profinite completion of a finitely presented group. In fact, this property is easier to see than Theorem~\ref{thm:C} itself. 
\end{rmk}

  \subsection{Independence of  $\ell$ and Schur rationality}

Let $\pi$ be a profinite group, and let $\varphi : \pi \surj G$ be a continuous finite quotient with kernel $U_\varphi = {\rm Ker}(\varphi)$. We denote by $U_\varphi^{ab}$ its abelianization. Then conjugation induces a commutative diagram
\ga{}{
\xymatrix@M+1ex{\pi \ar[r] \ar@{->>}[d]^{\varphi} & {\rm Aut}(U_\varphi) \ar@{->>}[d]  \ar[r] & {\rm Aut}(U_\varphi^{ab})
 \\
G \ar[r] & {\rm Out}(U_\varphi) \ar[ru] & } \notag}
If $\pi$ is finitely generated, then $U_\varphi$ is finitely generated so
 $U_\varphi^{ab}$ is a finitely generated $\hat \Z$-module. We set
 \ga{}{
\rho_{\varphi,\ell} : G \to \GL(U_\varphi^{ab} \otimes \Q_\ell), \notag}
for the induced representation, 
with character 
\ga{}{ 
\chi_{\varphi,\ell} = {\rm Tr}(\rho_{\varphi,\ell}) : G \to \Q_\ell. \notag }
The first main point is the following proposition.
\begin{prop}[See \cite{ESS22b}, Propositions~3.4,~3.5] \label{prop:ind}
1) If $\pi$ is $p'$-discretely finitely generated, then for all $\ell\neq p$, $ \chi_{\varphi,\ell}$ has values in $\Z$ and is independent of $\ell$. \\
2) If $X$ is a smooth projective variety defined over an algebraically closed  characteristic $p>0$ field, and $\varphi: \pi\surj G$ is a finite quotient (thus in particular 
$\varphi$ is 
continuous), then for all $\ell\neq p$,  $ \chi_{\varphi,\ell}$ has values in $\Z$ and is independent of $\ell$.
\end{prop}
\begin{proof}
The property 1) comes essentially from the definition:  setting $$\Gamma_\varphi={\rm Ker} (\Gamma (\to \pi)\to G)$$ we have for $\ell \neq p$ the relation 
$$\Gamma^{ab}_\varphi\otimes_{\Z}\Q_\ell=U^{ab}_\varphi\otimes_{\Z} \Q_\ell.$$ The property 2) is more interesting in view of the final result, see Theorem~\ref{thm:obst}. Let $Y\to X$ be the Galois cover with group $G$.  Then $$U^{ab}_\varphi\otimes_{\Z}\Q_\ell= H^1(Y, \Q_\ell)^\vee.$$  As a consequence of the Weil conjectures, see \cite{KM74}, 
 the characteristic polynomial of  $g\in G$ acting on $H^1(Y, \Q_\ell)$ lies in $\Z[T]$ and does not depend on $\ell$.

\end{proof}
  \begin{rmk}
  In particular 2) tells us that this independence of $\ell\neq p$ property cannot be our sought obstruction. On the contrary, we shall use it now in order to define a rationality obstruction. 
  \end{rmk}
  
  The second main point is the following Proposition. 
  \begin{prop}[See \cite{ESS22b}, Proposition~3.7]
  If $\pi$ is $p'$-discretely finitely generated, then for any continuous finite quotient $\varphi: \pi\surj G$, there is a $\Q$-vector space  $V_\varphi$ and a representation $\rho_\varphi: G\to GL(V_\varphi)$ such that for every $\ell\neq p$, the relation $$\rho_{\varphi, \ell}=\rho_\varphi \otimes_{\Q} \Q_\ell$$ holds.
  \end{prop}
  \begin{proof}
  Indeed,  $V_\varphi=\Gamma_\varphi^{ab}$ and $\rho_\varphi: G\to GL( \Gamma_\varphi^{ab} \otimes \Q)$. 
  
  \end{proof}
  \begin{defn}
  We say that for $\ell\neq p$, $\rho_{\varphi, \ell}$ is Schur rational.
  \end{defn}
  \begin{rmk}
  In fact, as we see in the proof, $(V_\varphi, \rho_\varphi)$ even has an integral structure, so is {\it Schur integral}, but our obstruction shall disregard this integrality property. 
  \end{rmk} 
  \begin{thm} \label{thm:obst}
  The Schur rationality is an obstruction for  a smooth projective variety defined over an algebraically closed field of characteristic $p>0$ to lift to characteristic $0$.
  \end{thm}
  So we have to exhibit an example of a smooth projective variety  $X$ defined over an algebraically closed field $k$ of characteristic $p>0$ and a quotient $\varphi: \pi\surj G$ such that $\rho_{\varphi, \ell}$ is not Schur rational. 

\subsection{The Roquette curve, combined with Serre's construction}
The Roquette curve is  the smooth projective curve $C$ over $\F_p$ for $ \ p\ge 3$ which is the normal
compactification of the affine curve with equation
\ga{}{\ \  y^2=x^p-x.\notag}
It is defined in \cite{Roq70}, has genus $g=(p-1)/2$, is supersingular,  and is the only curve for $p\ge 5$, (so $g\ge 2$ and  $\rho_\ell$ is faithful),  and $p> g+1$  with the property that  the cardinality of its group of automorphisms $G$ is larger than the Hurwitz bound $84(g-1)$. Precisely it is $2p(p^2-1)$ and all automorphisms are defined over $\F_{p^2}$.   The equation of $C$ presents it as an Artin-Schreier cover of $\P^1\setminus \P^1(\F_p)$. This realizes $\Z/p=:N$ as a subgroup of $G$, which  thus is a $p$-Sylow which is in fact normal.   It is not difficult to compute (see \cite[Appendix~A]{ESS22b})) that 
$\rho_\ell|_N$ is non-trivial, thus $\rho_\ell$ is absolutely irreducible.  The action of $\Q_\ell[G] \to {\rm End}(H^1(C,\Q_\ell))$ has values in $\Q_\ell[G] \to {\rm End}_{\rm Frob}(H^1(C,\Q_\ell))$,  and is surjective by the absolute irreducibility.   The Tate conjecture  identifies it with the   action
 $\Q_\ell[G] \to  {\rm End}^0(C)\otimes \Q_\ell$. The quaternion algebra  ${\rm End}^0(C)$ is ramified at $p$ and $\infty$, which prevents $\rho_{\varphi, \ell}$ to be rational, see \cite[Proposition~4.6]{ESS22b}. 
 
\begin{proof}[Proof of Theorem~\ref{thm:obst}]
 
 We now take a smooth connected projective variety $P$ defined over $\F_p$ of dimension at least $3$,  so $P$  is simply connected over $\bar \F_p$,  and on which $G$ acts without fixpoints. We do Serre's construction setting $X=(C\times_{\F_p}P)/G$ where $G$ acts diagonally. It yields an exact sequence 
 $$1\to \pi_1(C)\to \pi_1(X)\xrightarrow{\varphi} G\to 1$$ 
 which can be understood as the Galois sequence for  the Galois \'etale cover $C\times_{\F_p} P\to X$ or equivalently as the homotopy exact sequence for $X\to P/G$.  The action of $G$ on $H^1(C\times_{\bar \F_p} P, \Q_\ell)= H^1(C_{\bar \F_q}, \Q_\ell)$ is the same as the outer action studied above. So it is not Schur rational. Consequently  $X$ does not lift to characteristic $0$. This finishes the proof. 

\end{proof}

\newpage
\section{Lecture 6:  Density of Special Loci} \label{sec:qu}
\begin{abstract}
By a theorem of  Clemens and Landman,  see  \cite[Theorem.~3.1]{Gri70} in complex geometry and Grothendieck \cite[XIV~1.1.10]{SGA7.2} in arithmetic geometry, geometric (complex or $\ell$-adic) local systems have quasi-unipotent monodromies at infinity. We explain in this section why  this property is good for going  from complex models to models over finite fields, and why in the Betti moduli space the complex local systems with quasi-unipotent monodromy at infinity are Zariski dense.  Further, we report on ~\cite{BGMW22}, ~\cite{LL22a}, ~\cite{LL22b}  showing that the arithmetic local systems  on geometric generic curves in low rank cannot be Zariski dense in their Betti moduli, contrary to what was expected in  \cite[Question~9.1 (1)]{EK20} and \cite[Conjecture~1.1]{EK23}.  Finally we report on the concept of {\it weakly arithmetic local systems} defined in \cite[Section~3]{dJE22}, 
which in particular have quasi-unipotent monodromies at infinity,
and show that they are Zariski dense in their Betti moduli.  So for certain problems  (to be defined) one would then wish to follow Drinfeld's method \cite{Dri01}  to conclude that it is enough to check them on weakly arithmetic local systems.
 
\end{abstract}

\subsection{Definitions}
For the definition we only need $X$ to be a normal variety and $X\hookrightarrow \bar X$ to be a normal compactification. This defines the codimension $1$ components $D_i$ in $\bar X\setminus X$ and small loops $T_i$ around there. A representation $$\rho: \pi_1(X,x)\to GL_r(\C)$$ has {\it quasi-unipotent monodromies at infinity} if $\rho(T_i)$ is quasi-unipotent. This property does not depend on the choice of the compactification~\cite[Thm.3.1]{Kas81}. See \cite[Section~3]{EG21} where the concept is used also for $GL_r$ replaced by any linear algebraic group as well. 

\medskip
Set $\pi=\pi_1(X(\C), x(\C))$ for the topological fundamental group based at some complex point.
The {\it framed character variety} $ Ch(\pi, r)^\square$ is defined to be the affine variety defined over $\Z$ by the moduli functor which takes affine rings $R$ to the set ${\rm Hom}(\pi, GL_r(R))$. It is a fine moduli functor and the resulting scheme is also called the {\it framed Betti moduli} $M_B(X,r)^\square$, also defined over $\Z$. The group $GL_r$ acts by conjugation (gauge transformations) on $Ch(\pi, r)^\square$. Its categorial quotient $$Ch(\pi, r)=Ch(\pi, r)\sslash  GL_{r,\C}$$  defined by $$\sO(Ch(\pi, r))=\sO(Ch(\pi,r)^\square)^{GL_r}$$ is the {\it character variety}, also called the {\it Betti moduli space}  $M_B(X,r)$. Its complex points are  isomorphism classes of semi-simple  local systems  of rank $r$. The fibres of $M_B(X,r)^\square\to M_B(X,r)$ are the closures of $GL_r$-orbits. Such an orbit is closed over an irreducible complex local system.

\medskip

 \subsection{Why quasi-unipotent  monodromies at infinity} \label{sec:why}
The first reason is that {\it geometric} local systems   have  quasi-unipotent monodromies at infinity. 
We indicate Brieskorn's  complex proof  \cite[p.125]{Del70}: by base change   we may assume that 
 $Y\xrightarrow{g}  U\hookrightarrow X$ is defined over a number field. So the eigenvalues  $\lambda_i$ of the residues of the Gau{\ss}-Manin connections lie in $\bar \Q$. On the other hand, the Gau{\ss}-Manin local system is defined over $\Z$, as this is the variation of the Betti cohomology of $g$, so the eigenvalues of the monodromy at infinity lie in $\bar \Q$.
  By \cite[Corollaire~5.6, p.96]{Del70}  $\mu_i={\rm exp}(2\pi \sqrt{-1}\lambda_i)$. We conclude by Gelfond's theorem that $\lambda_i\in \Q$.

 A second reason is as follows. 
  Let $S$ be an affine  scheme of finite type over $\Z$ with $\sO(S)\subset \C$, such that $X\hookrightarrow \bar X$ and a given complex point $x\in X$  have a model $X_S\hookrightarrow \bar X_S$  as a relative good compactification and $x_S$ as an $S$-point of $X_S$, and the orders  the eigenvalues of the $T_i$ are invertible  on 
$S$.  For any closed point $s\in |S|$ of residue field $\F_q$ of characteristic $p>0$, with a $\bar \F_p$-point $\bar s$ 
above it, we denote by 
\ga{}{sp_{\C,\bar s}: \pi_1(X_\C, x_\C)\to \pi_1^t(X_{\bar s}, x_{\bar s}) \notag}
the continuous  surjective specialization  homomorphism to the tame fundamental group 
\cite[Expos\'e XIII 2.10,
Corollaire~2.12]{SGA1} and Section~\ref{sec:sp}.  Precomposing with the profinite completion homomorphism
\ga{}{  \pi_1(X(\C), x(\C))\to \pi_1(X_\C, x_\C) \notag}
yields
\ga{}{ sp_{\C, \bar s}^{\rm top}:   \pi_1(X(\C), x(\C))\to \pi_1^t(X_{\bar s}, x_{\bar s}) \notag}
which  is compatible with the local fundamental groups, see \cite[Section~1.1.10]{Del73}. This enables one to transpose the quasi-unipotent monodromy condition to $X_{\bar s}$. 

\medskip

Finally we understand well how the Galois group $\Gamma$ of $F={\rm Frac}(\sO(S))$ acts on the image $T_i^{\rm \acute{e}t}$  of the $T_i$ in $\pi_1(X_\C, x_\C)$, see \cite[XIV.1.1.10]{SGA7.2},  \cite[Lemma~2.1]{EK23}.

\begin{lem}\label{lem.cyclchar}
For each $1\le i\le s$ the action of $\gamma\in \Gamma $ on $\pi$ maps $T^{\rm \acute{e}t}_i$ to
$(T_i^{\rm \acute{e}t})^{\chi(\gamma)}$, where $\chi\colon \Gamma\to \widehat{\mathbb Z}^\times$ is the cyclotomic character.
\end{lem}

\subsection{Density Theorem for quasi-unipotent local systems}
\begin{thm}[\cite{EK23},  Theorem~1.3] \label{thm:dens}
The set of  $\rho\in Ch(\pi,r)^\square (\C) $ with quasi-unipotent monodromy at infinity
is Zariski dense in $ Ch(\pi,r)^\square(\C)$.
\end{thm}
\begin{proof}
Let $Q\subset Ch(\pi,r)^\square(\pi)(\C)  $ be the Zariski closure of the set of quasi-unipotent
representations. Assume $Q\neq Ch(\pi, r)^\square(\C)$.

\medskip

For each local monodromy at infinity $T_i\subset \pi$, choose $g_i\in T_i$. 
  We have a morphism
   \ga{}{\xymatrix{ & & M=\G_m^s \ar[d]^\varphi\\
   Ch(\pi,r)^\square  \ar[rr]^{\psi^\square }& & N= \prod_{i=1}^s  (\A^{r-1}\times \G_m) 
    }\notag}
 of affine schemes of finite type over $\C$
defined for each $i=1,\ldots, s$ by the coefficients
$(\sigma_1(\rho(g_i)),\ldots,\sigma_r(\rho(g_i)))\in N$ of the characteristic polynomials
  \ga{}{     {\rm det}( T\cdot \mathbb I_r -\rho(g_i)) = T^r-\sigma_1(\rho(g_i)) T^{r-1} +
    \ldots + (-1)^r \sigma_r( \rho(g_i)) \notag}
  of a  representation $\rho\colon \pi \to GL_r (\C)$. (The morphism $\psi^\square$ factors through $Ch(\pi, r)$, but we do not use $Ch(\pi,r)$ in this lecture). There is a scheme $B$ of finite type over $\Z$ with factorization $\Z\to \sO(B)  \to  \C$ over which the diagram $(\psi^\square,\varphi)$ and the inclusion $Q\hookrightarrow Ch(\pi,r)^\square$ are defined.  We write $Q_B\hookrightarrow Ch(\pi,r)^\square_B$ and 
   \ga{}{\xymatrix{ & & M_B=\G_{m,B}^s \ar[d]^{\varphi_B}\\
   Ch(\pi,r)_B^\square  \ar[rr]^{\psi^\square_B} & & N_B= \prod_{i=1}^s  (\A_B^{r-1}\times \G_{m,B})
    }\notag}
  
\medskip

Using the section  $\Gamma \to \pi_1(X_F, x_\C)$ given by $x_F$, the Galois group $\Gamma$ acts  by conjugacy  on $\pi_1(X_\C, x_\C)$, thus on the set of closed points $|Ch(\pi, r)^{\square}|$. Since a closed point $z\in |Ch(\pi, r)^{\square}|$ has finite monodromy, its stabilizer $\Gamma_z\subset \Gamma$ is an open subgroup. It acts on the completion $(\widehat{Ch(\pi,r)^\square})_z$ at $z$. Said differently, the action $\Gamma\to  {\rm Aut} (|Ch(\pi, r)^\square|)$ is continuous.  Furthermore, Lemma~\ref{lem.cyclchar} enables one to extend the action of $\Gamma$ on the diagram $(\psi^\square_B, \varphi_B)$ in a compatible way with the action on 
$|Ch(\pi, r)^\square|$.

\medskip
Set $T_B$ to be the reduced Zariski closure of ${\rm Im}(\psi_B)$ and $S_B=\varphi^{-1}_B(T_B)$.  As $Ch(\pi, r)^\square (\C)
\setminus Q \neq \emptyset$, $Ch(\pi,r)^\square_B\setminus Q_B$ dominates $B$.  So in particular, $T_B$ and thus  $S_B$ dominate $B$ as well. By generic smoothness, the smooth locus $S_B^{\rm sm}$ over $B$ dominates $B$. By generic flatness for $\psi_B$ restricted to $Ch(\pi,r)^\square_B\setminus Q_B$, 
 its image meets $\varphi_B(S_B^{\rm sm})$ (recall $\varphi_B$ is finite).  So there is a closed point 
 $z\in |Ch(\pi,r)_B^\square\setminus  Q_B|$ such that
\begin{itemize}
\item $\psi^\square$ is flat at $z$,
  \item $y=\psi^\square_B(z)\in \varphi_B(S_B^{\rm sm})$.
\end{itemize}
We also fix a closed point $x\in S^{\rm sm}_B\cap \varphi_B^{-1}(y)$. Let $\Gamma'$ be the
intersection of stabilizers $\Gamma_x\cap\Gamma_z,$  which is thus open in $\Gamma$, and let $b\in B $ be the image of
the points $x,y,z$.
So 
the closed subscheme $(\widehat{S_B})_x\hookrightarrow (\widehat{M_B})_x$ is $\Gamma'$-stable. We abuse notation and set $\Gamma'=\Gamma$. 

\medskip
We now take $s\in |S| $ a non-ramified closed point with residue field $\F_q$ of  characteristic $p>0$ different from the residual characteristic $\ell$ of $x$.
Then the Frobenius $Frob_s$ lies in $\Gamma$ and by Lemma~\ref{lem.cyclchar} it acts on $M_B$ and $(\widehat{M_B})_x$ by multiplication with $q$.  By  (as we are over $ (\widehat{B})_b$, a variant of)
de Jong's conjecture~\cite[Conjecture~1.1]{dJ01} solved by  B\"ockle-Khare ~\cite{BK06}   in specific cases and by Gaitsgory~\cite{Gai07}  in general for $\ell \ge 3$, there are $Frob_s$-invariant points in $(\widehat{S_B})_x$,  flat over $(\widehat{B})_b$.  So there are  points for which the coordinates in the group scheme $M_B$  are $(q-1)$ roots of $1$ which are  flat over $(\widehat{B})_b$.  By flatness  of $\psi_B$ restricted to  $(\widehat{Ch(\pi,r)_B^\square\setminus  Q_B})_z$, there is a point in $(\widehat{Ch(\pi,r)_B^\square\setminus  Q_B})_z$, flat over $(\widehat{B})_b$.  This yields a  complex topological local system outside of $Q$ with eigenvalues of the $T_i$ being $(q-1)$ roots of $1$, a contradiction.  See \cite{EK23} for more details.

\end{proof}

\subsection{Remarks}  \label{rmk:LL} 1) When I lectured on zoom  in December 2020, at the pic of corona, on our density  theorem~\ref{thm:dens} with Moritz Kerz,  Ben Bakker and Yohan Brunebarbe listened to the talk. They later on explained to us that they had a Hodge theoretical proof of the result. This would be nice, as it would add one 
stone to  the line of similarities between complex  and arithmetic methods. 

\medskip

2) It would also be nice to single out subloci of the one consisting of complex local systems with quasi-unipotent monodromies at infinity for which density is preserved. One 
element of answer is provided by \cite[Section~3]{dJE22}, see  Section~\ref{sec:wa},  in which we define the notion of weakly arithmetic local systems and prove their density. Weakly arithmetic local systems have in particular quasi-unipotent monodromies at infinity.

\medskip

3)
Of course, we should remark that if $X$ was projective to start with, our Theorem~\ref{thm:dens} would be void. Still if we think of varieties defined over number fields and Belyi's theorem, $X=\P^1\setminus \{0,1,\infty\}$ is the key scheme for many problems and so clearly the monodromies at infinity span the whole fundamental group. In addition if we think of analogies with the number theory case, where we look at Galois representations, we have ramification at bad primes.  Unlike what was hoped for in \cite[Question~9.1 (1)]{EK20} and  \cite[Conjecture~1.1]{EK23}, the sublocus of arithmetic points is not dense, thanks to the work of Biswas-Gupta-Mj-Whang~\cite{BGMW22} in rank $2$ and 
 Landesman-Litt~\cite{LL22a}, ~\cite{LL22b} in any rank, see also \cite{Lam22}. Some points of their construction is discussed  in Section~\ref{sec:LL}. On the other hand, on the Mazur-Chenevier deformation formal schemes for smooth quasi-projective varieties over finite fields,  it would truly be bad if the sublocus of arithmetic points  was not dense. In rank one it is, see \cite[Theorem~1.3]{EK21}.   Over $\C$, as already mentioned in the previous paragraph, weakly arithmetic complex local systems are dense. 
 

\subsection{Some other dense or not dense loci}

We use the notation of Theorem~\ref{thm:dens} and fix some finite order  conjugacy classes $T_i$ of monodromies at infinity. 
Let $\sT\subset  M_B(X,r)(\C)$ be the set of complex points with finite monodromy and 
$\sT(T_i)\subset 
M_B(X,r,T_i)(\C)$ be its intersection with $M_B(X,r,T_i)$. If $\sT(T_i)$ is not empty, all the $T_i$ must have finite oder.   We denote by $ \bar \sT$, resp. $\overline{\sT(T_i)}$ the Zariski closure of $\sT$ in $M_B(X,r)$, resp. $\sT(T_i)$ in $M_B(X,r, T_i)$ and similarly by    $ \bar\sT^{an}$  resp.   $\overline{\sT(T_i)}^{an}$ the analytic closures. 
  For any smooth  complex variety $Y$, we denote by $j: Y\hookrightarrow \bar Y$  a smooth projective compactification such that $\bar Y\setminus Y$ is a normal crossings divisor.

\begin{prop} \label{prop:yves}
 \begin{itemize}
\item[1)]  If all the $T_i$ have finite order, there is a finite Galois \'etale cover $h_1: Y_1\to X$  such that for any $\mathbb L\in M_B(X,r, T_i)(\C)$,  $h^* \mathbb L$ extends to a local system on $\bar Y_1$.
\item[2)] There is a finite Galois \'etale cover $h_2: Y_2\to X$ such that \\
a)  for any $\mathbb L\in \sT$, $h^* \mathbb L$  is  a sum of $r$ torsion  rank $1$  local systems; \\
b) for any $\mathbb L\in \bar \sT$, $h^* \mathbb L$  is  a sum of $r$  rank $1$  local systems.
\item[3)]   If all the $T_i$ have finite order, there is a finite Galois \'etale cover $h_3: Y_3\to X$ such that \\
a)  for  any $\mathbb L\in \overline{ \sT(T_i)}$, $h^* \mathbb L$ extends to a sum of $r$ rank $1$  local systems on $\bar Y_3$;\\
b) for  any $\mathbb L\in \overline{ \sT(T_i)}^{an}$, $h^* \mathbb L$ extends to a sum of $r$ rank $1$ unitary local systems on $\bar Y_3$.

\end{itemize}
\end{prop}

\begin{proof}

We prove 1). We consider  the  affine closed subschemes
\ga{}{ \xymatrix{\ar[d]^qM_B(X,r,T_i)^\square \ar[r] & \ar[d]_q M_B(X,r)^\square\\
M_B(X,r, T_i)  \ar[r] & M_B(X,r)} \notag}
As $M_B(X,r,T_i)^\square$ is  a fine moduli space, we have the universal representation
\ga{}{\rho_{univ}: \pi_1(X(\C), x(\C))\to GL_r( \sO(M_B(X,r)^\square)).\notag}
As $\pi_1(X(\C), x(\C))$ is finitely generated, Selberg's theorem (as used in \cite[Sections~8]{And04}) implies that there is a finite \'etale  Galois cover $h_1: Y_1\to X$ such that 
$\rho_{univ}({ \pi_1(Y_1(\C), y_1(\C))})$  is torsion-free. Here $y_1 \in Y_1$ lies above $x$. As the $T_i$ where assumed to have finite order, this implies that 
$\rho_{univ}|_{({ \pi_1(Y_1(\C), y_1(\C))})}$ factors through $\pi_1(\bar Y_1(\C), y_1(\C))$. Any complex point $\rho\in M_B(X,r,T_i)^\square(\C)$ is defined as the local system to  $\iota\circ \rho_{univ}$, where $\iota: \sO(M_B(X,r,T_i)^\square)\to \C$ is a complex point. 
We conclude that all  complex points of $M_B(X,r,T_i)$, restricted to $Y_1$, extends to $\bar Y_1$, that is they have no ramification.  This proves 1).

\medskip

We prove 2)a).  We use Jordan's theorem as in \cite[Section~10]{And04}: there is a finite \'etale Galois cover $h_2: Y_2\to X$ such that for all $\mathbb L\in \sT$, $h^*\mathbb L$ is a sum of $r$ torsion rank $1$ local systems. This proves 2)a).

\medskip

We prove 2)b).   The morphism $h_2^*: M_B(X,r)\to M_B(Y,r)$ is finite onto its image, thus $h^*(\bar \sT)$ is the Zariski closure of $h^*(\sT)$ in $M_B(Y,r)$. On the other hand, the locus in $M_B(Y,r)$ consisting sums of $r$ rank $1$ local systems is Zariski closed, see Lemma~\ref{lem:sum}.  This proves  2)b).

\medskip

We prove 3)a).  See als \cite[Theorem~5.1]{EL13} for an analog statement for vector bundles. 

 Taking now $h_3: Y_3\to X$ finite étale Galois dominating both $h_1$ and $h_2$,  1) holds for $h_3$. 
The extension to $\bar Y_3$ defines the factorization 
\ga{}{h_3^*: M_B(X,r,T_i)\xrightarrow{c^*} M_B(\bar Y_3,r) \xrightarrow{j^*} M_B(Y_3,r) \notag}
where $h_3^*$ is finite and its image is Zariski  closed.  So $c^*$ is finite and its image is  Zariski closed. 
As $j^*$ is furthermore Zariski closed, the Zariski closure of $ h_3^*(\sT(T_i))$ in  $M_B(Y_3,r) $ is the image by $j^*$ of the Zariski closure of 
$c^*(\sT(T_i))$ in $M_B(\bar Y_3,r)$ which itself is equal to  $c^*(\overline{\sT(T_i)})$. We apply Lemma~\ref{lem:sum} to conclude that $c^*(\sT(T_i))$
consists of local systems which split as a  sum $\oplus_1^r \sL_i$  of $r$ rank $1$ local systems.  This proves 3)a).

\medskip

We prove 3)b). 
We have to see that if $\mathbb L\in \overline{\sT(T_i)}^{an}$ then the $\sL_i$ are unitary. 
The analytic closure of torsion rank $1$ local systems on $M_B(\bar Y_3,1)$ is the space $U(\bar Y_3)$  of unitary rank $1$ connections, which is compact. 
Thus via the algebraic morphism $M_B(\bar Y_3,1)^r\to M_B(\bar Y_3, r) (\hookrightarrow M_B(Y_3,r))$  the image of $U(\bar Y_3)^r \to M_B(\bar Y_3,r) (\hookrightarrow M_B(Y_3,r))  $ is analytically closed. Thus $c^*(\overline{ \sT(T_i)}^{an})$ lies in this image. This  proves 3b)). 

\end{proof}
\begin{rmk}
  In Proposition~\ref{prop:yves} we cannot replace in 3)  $\overline{\sT(T_i)}$ by $\bar \sT$ and similarly for the analytic version. 
  Assume $X=\G_m$ then $\pi_1(X(\C), x(\C))=\Z$ and $M_B(X,1)=\G_m$, where for $\mu\in \G_m(\C)=\C^\times$, the associated local system $\mathbb L$  is defined by $\Z\to \C^\times, \ 1\mapsto \mu$. Then $\sT= \mu_\infty \subset  \C^\times$,  the subgroup of roots of unity, its analytic closure is $S^1\subset \C^\times$, the circle, its Zariski closure is the whole of $\C^\times$.   However,   $h:Y\to \G_m$ has the property  that $\bar Y$ is isomorphic to $\P^1$, thus is simply connected and non-torsion rank $1$ local systems cannot be trivialized by a finite \'etale cover. 
See also  \cite[Remarque~7.2.3]{And04} for the same example used in a similar spirit. 

\end{rmk}

The following lemma ought to be well known, we write a short proof as we could not find it in the literature.

\begin{lem} \label{lem:sum}
For any finitely generated group $\pi$,  and natural number $r\ge 1$, the locus of complex points  $\mathbb L$ in $Ch(\pi, r)$  such that $\mathbb L$ is a sum  of rank $1$ representations up to isomorphism is closed. 
\end{lem}
\begin{proof}
The structure morphism $q: Ch(\pi,r)^\square\to Ch(\pi,r)$  endows $Ch(\pi,r)$ with the quotient topology, so Lemma~\ref{lem:sum} is equivalent to saying that the locus of complex points $\rho$ in $Ch(\pi,r)^\square$ which stabilize a complete flag, i.e. the associated graded splits $\C^r$ as a sum of rank $1$ vector subspaces, is Zariski closed. 
Let $(1= \gamma_1, \gamma_2, \ldots, \gamma_s)$  be generators of $\pi$. Let $Fl$ be the variety of complete flags on $\C^r$. Consider the algebraic map
\ga{}{ Ch(\pi, r)^\square \to {\rm Hom} (Fl, (Fl)^s), \ (\rho \mapsto [ Fl \ni F\mapsto  (Fl)^s\ni (\gamma_1(F),\ldots, \gamma_s(F)]).\notag}
Let $\Delta\subset (Fl)^s$ be the diagonal. Then ${\rm Hom}(Fl, \Delta)\subset {\rm Hom}(Fl,  (Fl)^s)$ is a closed embedding, so its inverse image in $Ch(\pi, r)^\square$ is closed as well.

\end{proof}

 Proposition~\ref{prop:yves} 2)b) is a convenient way to see that some $\mathbb L$ cannot be in  $\bar \sT$, see Example~\ref{ex:nss} and  Corollary~\ref{cor:LL}.
 \begin{ex} \label{ex:nss}
 Let $\rho: \pi_1(X(\C), x(\C))\to GL_r(\C)$ be a  representation with the property that $\rho( \pi_1(X(\C), x(\C)))$ is Zariski dense in $G \subset GL_r(\C)$, where
$G=GL_{r'} (\C)$ or $G=SL_{r'}(\C)$ for some $r'$ with  $2\le r'\le r$. 
  Then  $\rho$ is semi-simple and its moduli point does not lie in $\bar \sT$.
 \end{ex}

\subsection{ Arithmetic local systems are not dense: the work of Biswas-Gupta-Mj-Whang and   Landesman-Litt} \label{sec:LL}
We sketch  some aspects of the construction referred to in Remark~\ref{rmk:LL} 2)  by Aaron Landesman and Daniel Litt, which generalizes in higher rank 
Theorem~A and Lemma~3.2 of \cite{BGMW22}.  While the latter uses the specificity of $SL_2$ representations, which have finite monodromy if and only if they do on specific loops, the former uses Hodge theory. 
We are grateful to the two authors for their swift, precise  and friendly 
answers to our questions. 
 We  just develop some points of the arguments  in a simplified geometric situation and refer to their original articles \cite{LL22a}, \cite{LL22b} for the complete statements and proofs. 

 \subsubsection{Statement} \label{sec:stat}
 
 Let $f: X\to S$ be a smooth family of genus $\ge 2$ curves together with a dominant morphism $h: S\to M_g$, where $M_g$ is the coarse moduli space of genus $g$ curves.  We shall assume that $f$ has a section $S\to X$ which is irrelevant for the statements and their proofs. 
Let $s \in S$ be a complex generic point. We denote by $X_s$ the fibre of $f$ above $s$. For $x\in X_s$ a geometric point,  for example the one on the section of $f$, the homotopy sequence
 \ga{}{1\to \pi_1(X_s(\C), x(\C))\to \pi_1(X(\C), x(\C))\to \pi_1(S(\C), s(\C))\to 1\notag}
 is exact.
 As     the image of $ \pi_1(X_s(\C), x(\C))\to \pi_1(X(\C), x(\C))$ is normal, 
  the restriction $\mathbb L|_{X_s}$  to $X_s$ of any semi-simple complex local system $\mathbb L$  on $X$ is semi-simple. 
 
\begin{thm}[\cite{LL22b}, Theorem~1.2.1] \label{thm:LL}
If a semi-simple local system $\mathbb L$ on $X$ has rank $r< \sqrt{g+1}$, then $\mathbb L|_{X_s}$ has finite monodromy.

\end{thm}
 If $r=1$, so in particular if $g=1$, 
  Theorem~\ref{thm:LL} 
 is due to \cite[Lemma~3.2]{BGMW22} by analyzing the effect of the Dehn twists on the images of the generators of $H_1(X_x, \Z)$ under the character representation.  If $r=2$ this is due to \cite[Theorem~A]{BGMW22}, by analyzing the monodromy on specific loops. 
  Also the theorem is valid if $\mathbb L$ is not semi-simple. We won't discuss this last step. 
\begin{cor}[\cite{LL22a}, Corollary~1.2.10, Lemma~7.5.1] \label{cor:LL}
If $g\ge 2$ and  $2\le r <\sqrt{g+1}$, the image  of   the restriction morphism $M_B(X, r) \to M(X_s, r)$  is not dense. 

\end{cor}

\begin{proof}[Proof of Corollary~\ref{cor:LL}] We have to show that the moduli points corresponding to  local systems with finite monodromy are not Zariski dense. 
Let $a_1, b_1, \ldots, a_g,b_g, \ g\ge 2$ be generators of $\pi_1(X_s(\C), x(\C))$ with the single relation $\prod_{i=1}^g [a_i,b_i]=1$. 
Define $\rho: \pi_1(X_s(\C), x(\C)) \to GL_r(\C)$ by sending $a_1, a_2$ to  the block matrices
\ga{}{\begin{pmatrix} 1 & 1&  0\\
0 & 1 & 0 \\
0 & 0 & {\mathbb I}_{r-2}
\end{pmatrix} \  \ 
\begin{pmatrix} 1 & 0&  0\\
1 & 1 & 0 \\
0 & 0 & {\mathbb I}_{r-2}
\end{pmatrix} 
\notag}
and all the other generators $a_j$ for $ j\ge 3$ and $b_i$ for $ i\ge 1$ to $\mathbb I_r$, where $\mathbb I_n$ is the unit square matrix of size $n\times n$.   The monodromy group is dense in $SL_2(\C) \subset GL_r(\C)$. We apply 
Example~\ref{ex:nss}. 

\end{proof}
\begin{rmk} As 
$M_B(\pi_1(X(\C), x(\C)), r)^\square\to M(X, r)$ is surjective,  we could equally say in Corollary~\ref{cor:LL} that the image  of 
$M_B(\pi_1(X(\C), x(\C)), r)^\square\to M(X_s, r)$ is not dense. Said in words, a semi-simple local system on $X_s$ which lifts to $X$ lifts as a semi-simple local system, and those are not dense for $r$ small.

\end{rmk}
We keep the same notation as above.  The morphism $f:X\to S$ descends to $f_\circ: X_\circ\to S_\circ$ over $F_\circ$ where $F_\circ/\Q$ is a field of finite type over $\Q$. Thus $F_\circ(S_\circ)\subset \kappa(x) \cong \C$, the residue field of $x$. 
Recall that an $\ell$-adic local system  $\mathbb L_\ell$  on $X_s$  is 
arithmetic if it descends to $X_F$, where  $ F_\circ(S_\circ)\subset F (\subset \kappa(x)) $ is a finite type extension. 
Geometrically,  it means that there is  a dominant morphism $T_1\to S_\circ \otimes_{F_\circ} F_1$ defined over $F_1$, a finite extension of $F_\circ$ such that $F=F_1(T_1)$, 
 and such that
 $\mathbb L_\ell$ is defined over the fibre  of  the pull-back morphism $f_\circ\otimes_{S_\circ} T_1=: f_{1}:  X_\circ\otimes_{S_\circ} T_1=:X_{1}\to T_1$ over the generic point
${\rm Spec}(F)$.
 As the notion of arithmeticity does not depend on the choice of such an $F$, see 
\cite[Remark~3.2]{dJE22}, we'll assume henceforth that $f_{1}$ admits a section $T_1\to X_{1}$, and then we abuse notation setting $f=f_1$, so also $F=F_\circ(S)$.  In particular, this guarantees that the homotopy sequence of topological fundamental groups 
\ga{}{1\to \pi_1(X_s(\C),x(\C)) \to \pi_1(X(\C),x(\C)) \to \pi_1(S(\C),s(\C))\to 1\notag}
is exact, and by \cite[Proposition~3]{And74},  if $g\ge 2$ which implies that the center of $ \pi_1(X_s(\C),x(\C)) $ is trivial, 
that the induced homotopy sequence of \'etale fundamental groups
\ga{}{ (\star) \ \ 1\to \pi_1(X_s,x) \to \pi_1(X,x) \to \pi_1(S,s)\to 1\notag}
is exact as well. 

\begin{cor}[\cite{LL22b}, Theorem~8.1.2]
If $g\ge 2$ and  $2\le r<\sqrt{g+1}$, the arithmetic local systems in $M_B(X_s,r)( \bar \Q_\ell)$ are not Zariski dense.

\end{cor}
\begin{proof}
The action of ${\rm Gal}(\bar F/F)$ on $\ell$-adic local systems comes from 
the homotopy  exact sequence 
\ga{}{ 1\to \pi_1(X_x,x)\to \pi_1(X_{F}, x)\to {\rm Gal}(\bar F/F )\to 1 \notag}
and the action of ${\rm Gal}(\bar F/F)$ on $\pi_1(X_x,x)$ via outer automorphisms (which can be lifted to automorphisms thanks to the section). 
An $\ell$-adic arithmetic local system   $\mathbb L_{\ell, \C}$ descending to $\mathbb L_{\ell, F'}$, for $F\subset F' (\subset \bar F)$ a finite Galois extension, 
is a fixpoint under this action restricted to ${\rm Gal}(\bar F/F')$. We may assume $F$ is equal to $F'$,  and $S'$, the normalization of $S$ in $F'$, is  equal to $S$.  
We restrict this action to ${\rm Gal}( \overline{\C(S)}/\C(S))$ via the base change morphism ${\rm Gal}(\overline{\C(S)}/ \C(S)) \to {\rm Gal}(\bar F/F)$ 
from $F=F_\circ(S)$ to $\C(S)$. By left exactness of $(\star)$,   the action  factors through  $\pi_1(S,s)$. Thus in fact $\mathbb L_{F \otimes_{F_\circ} \C} $ descends to an $\ell$ adic local system on $X$. Viewed as a topological local system with $\bar \Q_\ell$-coefficients, its restriction to $X_s$ has finite monodromy by choosing randomly a field isomorphism between $\bar \Q_\ell$ and $\C$ and applying Theorem~\ref{thm:LL}.  We apply Corollary~\ref{cor:LL}. This finishes the proof. 
\end{proof}

\subsubsection{Rigidity}
The notation is as in Section~\ref{sec:stat}. We had already discussed that an irreducible local system $\mathbb L$ has finite determinant. Moreover, by 
\cite[Lemma~2.1.1]{LLSS20},   \cite[Proposition~2.4]{AS16} and Lefschetz theory which implies that the whole monodromy is recognized on good curves, 
it has quasi-unipotent monodromies at infinity. 

\medskip 

Recall that an irreducible complex local system $\mathbb L$ with finite determinant  and quasi-unipotent monodromies at infinity 
 on a smooth variety $X$ is called {\it strongly cohomologically rigid}  if $$H^1(X,   \sE nd^0(\mathbb L))=0.$$ 
 Here the upper index $0$ indicates the trace free endomorphisms.  
 This  strong rigidity notation was introduced and discussed  in
 \cite[Definition~A.1.b)]{EG21}, as it is verified for any irreducible local system on a Shimura variety of real rank $\ge 2$. 
 It implies in particular that $\mathbb L$ is {\it cohomologically rigid}  as the map $$IC^1(X, \sE nd^0(\mathbb L)) \to H^1(X,  \sE nd^0(\mathbb L))$$ 
 from  intersection cohomology to cohomology is injective. 
 Recall the integrality theorem \cite[Theorem~1.1]{EG18}, see  the report on it in Section~\ref{sec:intEG}.
 \begin{thm} \label{thm:int}
 If $\mathbb L$ is cohomologically rigid, then it is integral, so in particular so is its restriction $\mathbb L|_{X_s}$.
 
 \end{thm}

 \subsubsection{Proof of Theorem~\ref{thm:LL}}
 
 To show finiteness of  the monodromy of $\mathbb L|_{X_s}$, we  thus only need two steps:
 \begin{itemize}
 \item[I)] For $r< \sqrt{g+1}$,  $\mathbb L|_{X_s}$ is a direct sum of irreducible unitary local systems $U_i$. 
 \end{itemize}
  If so, as by  the theorem of Krull-Schmidt this decomposition is unique,
  the action of $\pi_1(S(\C),s(\C))$ on the set  $\{U_i\}_i$ is finite, so after a finite  \'etale  base change $S'\to S$, $U_i$  lifts to $\mathbb U_i$ to  $X':=S'\times_S X$. So for the problem, we may assume $S'\to S$ is the identity.
  \begin{itemize}
 \item[II)] $\mathbb U_i$ is strongly cohomologically rigid. 

 \end{itemize}
 If so, then applying Theorem~\ref{thm:int} to the conjugates $\mathbb L^\sigma$  on the coefficients of the monodromies of $\mathbb L$, the local system $(\oplus_{\sigma} \mathbb L^\sigma)|_{X_s}$  is then unitary and has by Theorem~\ref{thm:int}
 monodromy defined over $\Z$, thus has finite monodromy. 
 
 \medskip
 
 Ad II): once we have the $\mathbb U_i$, as $U_i$ is irreducible, $f_* \sE nd^0(U_i)=0$ thus the Leray spectral sequence yields
 $$H^1(X, \sE nd^0(\mathbb U_i))=H^0(S, R^1f_*\sE nd^0(\mathbb U_i))$$
 and II) is equivalent to 
 $$  H^0(S, R^1f_*\sE nd^0(\mathbb U_i))=0.$$
 This vanishing is  going to be handled in I)c) below.
 
 \medskip
 
 Ad I): starting with $\mathbb L$, as  the subgroup $\pi(X_s(\C), x(\C))\subset \pi_1(X( \C), x( \C))$ is normal, the restriction $\mathbb L|_{X_s}$ is semi-simple. So one wants to prove that the summands are all unitary.  The way the authors proceed is as follows. 
 \begin{itemize}
 \item[I)a)]  Show I) is true if $\mathbb L$ is a polarized complex variation of Hodge structures.  This is the content of  \cite[Theorem~1.2.12]{LL22a}
for which the bound is slightly better, one needs only $r < 2 \sqrt{g+1}$. 
 \item[I)b)] If $\mathbb L$ is not from the beginning a polarized complex variation of Hodge structures, we know by Simpson-Mochizuki~\cite[Theorem~10.5]{Moc06} that the connected component of $M_B(X,r)$ which contains  the moduli point to $\mathbb L$ contains also the moduli point of one such  polarized complex variation of Hodge structure which we denote by $\mathbb M$. 
 \item[I)c)] Show that $\mathbb L|_{X_s}=\mathbb M|_{X_s}$. For this, one needs in the proof the sharper bound $r < \sqrt{g+1}$. 
 
 \end{itemize}
 
  \medskip
  
Let $(E,\nabla)$ be the algebraic flat connection associated to $\mathbb L$. The $\mathbb L|_{X_s}$ is unitary if and only if $E|_{X_s}$ is  (slope) semi-stable.  The proof of I)a) computes semi-stability in low rank. 

Ad I)c):  recall that the tangent map of the restriction $M_B(X,r)\to M_B(X_s,r)$ at the point $\mathbb M$ is identified with the restriction map $$H^1(X, \sE nd^0(\mathbb M))\to H^1(X_s, \sE nd^0(\mathbb M)).$$
This map factors through 
$$H^0(S, R^1f_*\sE nd^0(\mathbb M)) \subset H^1(X_s, \sE nd^0(\mathbb M)).$$
So the first step for proving I)c) is to show the vanishing of $H^0(S, R^1f_*\sE nd^0(\mathbb M))$,   which in particular implies the
vanishing needed for II) and so settles II). This is performed using Deligne's fix part  theorem (see \cite[Theorem~1.7.1]{LL22b}).   If both $\mathbb M$ and $\mathbb M|_{X_s}$ were irreducible, the formality theorem of Goldman-Millson~\cite[Proposition~4.4]{GM88} applied to $X_s$ and to $X$ (by a wishful generalization  to the non-proper case) would enable to conclude that the homomorphism of completed  local rings $\hat{\sO}_{M_B(X_s,r), \mathbb M|_{X_s} }\to
\hat{\sO}_{M_B(X,r), \mathbb M}$  factors through  $\hat{\sO}_{M_B(X_s,r), \mathbb M|_{X_s}} \xrightarrow{\mathbb M|_{X_s}}  \C$, the closed point corresponding to $\mathbb M|_{X_s}$, and consequently the restriction map $M_B(X,r)\to M_B(X_s,r)$ contracts the connected component containing $\mathbb M$ to its restriction $\mathbb M|_{X_s}$.

\medskip

The proof proceeds by showing the vanishing
$$H^0(S, R^1f_*\sE nd^0(\mathbb M_A))=0$$
for any extension of $\mathbb M$ to a local system $\mathbb M_A$ with coefficients in any local Artinian ring $A$ and for rank $\sE nd^0(\mathbb M)<g$, that is $r  <  \sqrt{g+1}$. This is an essential point of the proof and is  performed using the Kodaira-Spencer class  combined with Deligne's fix part theorem  (see \cite[Theorem~6.2.1]{LL22b}). This implies that the tangent map of  $M_B(X,r)\to M_B(X_s,r)$ 
 at the $A$-point $\mathbb M_A$ dies. As this is true for all local Artinian rings $A$,  this shows that $\hat{\sO}_{M_B(X_s,r), \mathbb M|_{X_s} }\to
\hat{\sO}_{M_B(X_s,r), \mathbb M|_{X_s}}$  factors through  $\hat{\sO}_{M_B(X_s,r), \mathbb M|_{X_s}} \xrightarrow{\mathbb M|_{X_s}} \C$ (see \cite[Proposition~3.2.1]{LL22b}).

\subsection{Weakly arithmetic complex local systems} \label{sec:wa}
In this section  we report on the notion of weakly arithmetic local systems introduced  with Johan de Jong  in \cite[Section~3]{dJE22}
which is weaker than the notion of arithmetic local systems.  
With this weaker notion of arithmeticity, we can prove density
in the Betti moduli space.  We refer to  {\it loc. cit.} for complete statements and proofs.

\subsubsection{Definitions}

Let $X$ be a  smooth complex quasi-projective variety. We fix a projective compactification $X\hookrightarrow \bar X$ such that $\bar X\setminus X$ 
is a strict normal crossings divisor. 
The compactification $X \hookrightarrow \bar X$ is defined over a field of finite type $F\subset \C$, and $$(X\hookrightarrow \bar X)=(X_F \hookrightarrow \bar X_F)\otimes_F \C.$$   It defines the algebraic closure $F\subset \bar F\subset \C$ of $F$.
Recall that the homotopy exact sequence
\ga{}{1\to \pi_1(X_{\bar F}, x_{\bar F})\to \pi_1(X_F, x_{\bar F})\to {\rm Gal}(\bar F/F)\to 1\notag}
defines an action of ${\rm Gal}(\bar F/F)$ on $\ell$-adic local systems coming from the outer action of $  {\rm Gal}(\bar F/F)$
on $\pi_1(X_{\bar F}, x_{\bar F})$. An {\it arithmetic} $\ell$-adic local system on $X_{\bar F}$ is one which has a finite orbit under this action. 
Equivalently it descends to $X_{F'}$ where $F\subset F'\subset \bar F\subset \C$ is a finite extension.  Clearly for this definition we may replace $F$ by a finite extension in $\bar F$ so assume that $X$ has a rational point. Then the outer action becomes just an action. More generally, we may replace $F$ by a finite type extension $F\subset F' $  in $\bar F$ as then 
the image of ${\rm Gal}(\bar F/F')\to {\rm Gal}(\bar F/F)$ is open so the definition does not depend on the $F$ chosen  (see e.g. \cite[Remark~3.2]{dJE22}).

\medskip

We choose  a model $X_S$ of $X$ over $S$ a scheme of finite type  over $\Z$, with field of fractions $F$, such that 
$$ (X_F\hookrightarrow \bar X_F, x_F) =( X_S\hookrightarrow \bar X_S, x_S) \otimes_S F $$
where $X_S\hookrightarrow \bar X_S$ is smooth relative compactification, $\bar X_S\setminus X_S$ is a relative normal crossings divisor. 
Then in fact an arithmetic local system $\mathbb L_\ell$  is {\it unramified almost everywhere} on $S$, that is there is a dense open  subscheme $S^\circ \subset S$ such that 
$\mathbb L_\ell$, defined on $X_F$, descends to $X_{S^\circ}$, see e.g. \cite[Proposition~6.1]{Pet22} where the argument
is performed for $S$  the spectrum of a number ring, and  see references therein.

\medskip

We  have the specialization homomorphism
\ga{}{sp^{\rm top} _{\C, \bar F} : \pi_1(X(\C), x(\C))\to \pi_1(X_{\bar F}, x_{\bar F}) \notag}
which stems from the Riemann Existence Theorem equating  $\pi_1(X_{\C}, x_{\C}) $
with the profinite completion of $\pi_1(X(\C), x(\C))$, then the base change continuous isomorphism
equating  $\pi_1(X_{\C}, x_{\C}) $ with the geometric \'etale fundamental group  $ \pi_1(X_{\bar F}, x_{\bar F})$, which we can then
postcompose with 
the continuous embedding  $ \pi_1(X_{\bar F}, x_{\bar F}) \to \pi_1(X_F, x_F)$
 from the the geometric \'etale fundamental group 
to the arithmetic one. In total this defines
\ga{}{sp^{\rm top} _{\C,  F} : \pi_1(X(\C), x(\C))\to \pi_1(X_{ F}, x_{F}). \notag}

\medskip

For any closed point $s\in S$ with geometric point $\bar s\to S$ above it, we have a specialization homomorphism 
\ga{}{ sp_{\bar F,\bar s}: \pi_1(X_{\bar F}, x_{ \bar F}) \to \pi_1^t( X_{\bar s}, x_{\bar s}) \notag}
which extends to 
\ga{}{ sp_{F, s}: \pi_1(X_F, x_{\bar F}) \to \pi_1^t(X_s, x_{\bar s}). \notag}
By precomposing with $sp^{\rm top}_{\C, \bar F}, \ sp^{\rm top}_{\C, F}$, we obtain
\ga{}{sp^{\rm top}_{\C, \bar s}: \pi_1(X(\C), x( \C))\to \pi_1^t(X_{\bar s}, x_{\bar s}), \ 
sp^{\rm top}_{\C,  s}: \pi_1(X(\C), x( \C))\to \pi_1^t(X_{ s}, x_{\bar s}) \notag.
}
where the upper script $^t$ indicates the tame fundamental group.
\medskip

If $\tau : K_1 \to K_2$ is a homomorphism of fields, and $\mathbb L_1$
is a topological local system defined over $K_1$   by the homomorphism $ \rho: \pi_1(X(\C), x(\C))\to GL_r(K_1)$, we denote by
$\mathbb L_{K_1}^\tau$ the local system defined over $K_2$
obtained by post-composing  $ \rho$ by the homomorphism $ GL_r(K_1) \to GL_r(K_2)$ defined by  $\tau$. 

\medskip

Choosing quasi-unipotent conjugacy classes $T_i $ in $GL_r(\C)$ and a torsion rank $1$ local system 
$\sL$ on $X(\bar \C)$, we recall that $M_B(X,r,\sL, T_i)$ is defined over a number ring $\sO_K \subset \C$.
We say that a complex point $\mathbb L_\C$ is {\it weakly arithmetic} if made $\ell$-adic for some $\ell$, and descended to a tame 
$\ell$-adic local system via $sp^{\rm top}_{\C, \bar s}$ for some closed point $s\in $ of large characteristic, it is on $X_{\bar s}$ arithmetic, that is it descends to $X_{s'}$ for $s'\to s$ a finite extension below $\bar s$.  Formally, we write the definition as follows. 
  
\begin{defn}[See \cite{dJE22}, Definition~3.1] \label{defn:wa}
\label{defn:arithm} 
A point $\mathbb L_{\C} \in M_B(X,r,\sL, T_i)(\C)$ is said to be
{\it weakly arithmetic} if there exist
\begin{itemize}
\item[i)]  a prime number $\ell$ and a field isomorphism $\tau : \bar\Q_\ell \to \C$  (so  $\tau$  
defines  an $\sO_K$-algebra structure on  $\bar\Q_\ell$
via $\sO_K \subset \C$);
\item[ii)] a finite type scheme $S$ 
such that $(X \hookrightarrow \bar X, x,  \sL, T_i)$
has a model over $S$;
\item[iii)] a closed point $s \in |S|$ of residual characteristic
different from $\ell$;
\item[iv)] a tame arithmetic $\ell$-adic local system
$\mathbb L_{\ell, \bar s}$ on $X_{\bar s}$ with determinant
$\sL$ and monodromy at infinity in $T_i$;
\end{itemize}
such that $\mathbb L_\C$ and
$\big((sp^{\rm top}_{\C, \bar s})^{ -1}(\mathbb L_{\ell, \bar s})\big)^\tau$
are isomorphic.

\end{defn}
 
\begin{rmk} \label{rmk:noncountable}
The choice of $\tau$ prevents the set of complex points in $M_B(X,r,\sL, T_i)$ which are the moduli points of 
 weakly arithmetic local systems to be a priori  countable, while the set of complex points in $M_B(X,r,\sL, T_i)$ which are the moduli points of 
 arithmetic local systems is really countable. So in the present state of understanding, this is a ``poor'' approximation.

\end{rmk}

\begin{nota}
1) Fixing $(X,r, \sL, T_i)$, we denote by 
 \ga{}{ W(X,r,\sL, T_i) \subset M_B(X,r,\sL, T_i)(\C) \notag}
 the locus 
 of weakly arithmetic local systems.\\ 
 2)   Fixing $(X,r, \sL)$, we denote by 
 \ga{}{ W(X,r,\sL) = \cup_{\{ T_i\}} W(X,r,\sL, T_i)  \subset M_B(X,r, \sL)(\C) \notag}
  the locus of all  weakly arithmetic local systems of rank $r$ and determinant $\sL$.
\end{nota}
\subsubsection{Density}

\begin{thm}[See \cite{dJE22}, Theorem~3.5] \label{thm:wa}
\label{thm:wa} 
1) Fixing $(X,r,\sL, T_i)$,  $W(X,r,\sL, T_i)$ is dense in $M_B(X,r,\sL, T_i)(\C)$. \\
2) Fixing $(X,r, \sL)$, $W(X,r,\sL)$ is dense in $M_B(X,r,\sL)(\C) $.
\end{thm}

\begin{proof}
Ad 1):  Let $T_\C$ be  the Zariski closure 
  of $W(X,r,\sL, T_i)$ in $M_B(X,r,\sL, T_i)(\C)$ and $T_{\sO_K}$ be the Zariski closure of $T_\C$ in $M_B(X,r,\sL, T_i)$.  As $T_\C$ is invariant under the group ${\rm Aut}_{\sO_K}(\C)$ of field automorphisms of $\C$ over $\sO_K$,  $T_{\sO_K}$ is the Zariski closure of $W(X,r,\sL, T_i)$  in $M_B(X,r,\sL, T_i)$.

If $T_\C$ was not equal to $ M_B(X,r,\sL, T_i) (\C)$, 
we could choose a 
closed point $z \in M_B(X, r, \mathcal{L}, T_i)$ of characteristic $\ge 3$ which has irreducible monodromy.  The quotient morphism
 $$M_B(X,r, \mathcal{L}, T_i)^\square   \to 
M_B(X,r, \mathcal{L}, T_i)$$ in a neighbourhood of $z$ is smooth, so it makes sense to talk on the irreducibility of the monodromy. 
We assume in addition that it does not lie
on $T_{\sO_K}$, and that it is  smooth for the projection $M_B(X,r,\sL, T_i) (\C)_{\rm red}  \to {\rm Spec}(\sO_K)$.   We choose a closed point $s\in S^\circ$ of characteristic $p$ larger than the cardinality of the monodromy group of $z$, the order of $\sL$ and of the eigenvalues of the $T_i$.   So the torsion local system corresponding to $z$ descends via $sp_{\bar F,\bar s}$  to $\mathbb L_{z, \bar s}$,  a local system on $X_{\bar s}$.  
Via the identification of Mazur's deformation ring $D_{z,\bar s}$ (\cite[Proposition~1]{Maz89})  with the formal completion of $M_B(X,r, \mathcal{L}, T_i)$ at $z$,  simply comparing the moduli functors, see \cite[Proposition~2.2]{dJE22}, we can apply Drinfeld's argument \cite[Lemma 2.8]{Dri01} to conclude:
 de Jong’s conjecture  \cite[Conjecture 2.3]{dJ01}, proved by B\"ockle-Khare \cite{BK06} in specific cases and by
 Gaitsgory [Gai07] for  $\ell\ge 3$,  implies that $\mathbb L_{z, \bar s}$ admits a lift as a $\bar \Z_\ell$-local system on $X_{\bar s}$ which is arithmetic. 
 Then $(sp_{\C,\bar s}^{\rm top})^{-1}( \mathbb L_{\ell, \bar s, z})$ does not lie in $T_{\sO_K} (\bar \Q_\ell)$. By invariance, for  any $\tau$, 
 $((sp_{\C,\bar s}^{\rm top})^{-1}( \mathbb L_{\ell, \bar s, z}))^\tau$ does not lie on $T(\C)$, a contradiction to the definition of weak arithmeticity. 

\medskip
\noindent
Ad 2): By  Theorem~\ref{thm:dens} (\cite[Theorem~1.3]{EK23}), for a torsion rank $1$ local system $\sL$ given,  
 \ga{}{ \cup_{T_i}  M_B(X, r, \sL, T_i) \notag}  is dense in 
$M_B(X,r,\sL) (\C)$. Combined with 1),  this yields 2).

\end{proof}

\begin{rmk} \label{rmk:dri}
The notion of weakly arithmetic local systems in Definition~\ref{defn:wa} is forced upon us by Drinfeld's  proof of Kashiwara's conjecture in complex geometry using de Jong's conjecture over finite fields in  \cite{Dri01}. Without really explicitly developing the concept, he proves density, and without explicitly mentioning it, he gives an arithmetic proof of Simpson's Hard-Lefschetz theorem for a semi-simple complex  local system on a smooth complex projective complex variety $X$
(\cite[Lemma~2.6]{Sim92}). We do not detail this last point here, we refer to \cite[Theorem~1.1]{EK21} where we implement this strategy for the proof in rank $r=1$ over $\bar \F_p$.
There is no way we can generalize the $r=1$ proof to higher rank. It uses strongly the shape of $M_B(X,1)$, a commutative group scheme which is the extension of a finite group by a torus. 
In light of the proof of Theorem~\ref{thm:wa}, one key point which is missing is the possibility to go from $\ell$ to $\ell'$, even in a non-canonical way (so without arithmeticity and companions, see Section~\ref{sec:comp}).  May be it is also related to the notion of $p'$-discrete generation (presentation) developed in Section~\ref{sec:ob}.

\end{rmk}

\newpage
\section{Lecture 7: Companions,  Integrality of Cohomologically Rigid Local Systems and of the Betti moduli space} \label{sec:companions}
\begin{abstract} \label{sec:comp}
The notion of companions has been defined by Deligne in \cite[Conjecture~1.2.10]{Del80} who predicted its existence. We report on the ($\ell$-adic version) of it (L. Lafforgue~\cite[Th\'eor\`eme~VII.6]{Laf02} in dimension $1$, Drinfeld~\cite[Theorem~1.1]{Dri12} in higher dimension in the smooth case), explain how we used Drinfeld's theorem in the proof of Simpson's integrality conjecture for cohomologically rigid complex local systems in \cite[Theorem~1.1]{EG18}, also why we proceeded like this, and how we combined this idea together with de Jong's conjecture in order to define and obtain an integrability property  of the Betti moduli space \cite[Theorem~1.1]{dJE22}. 
\end{abstract}

\subsection{Motivation on the complex side}

 Given a field automorphism $\tau$ of $\C$, we can postcompose the underlying monodromy representation of a complex local system $\mathbb L_\C$ by  $\tau$ to define  a {\it conjugate}  complex local system  $\mathbb L_\C^\tau$. Given a field automorphism $\sigma$ of $\bar \Q_\ell$, which then can only be continuous if this is the identity on $\mathbb \Q_\ell$, or more generally given a field isomorphism $\sigma$ between $\bar \Q_\ell$ and $\bar \Q_{\ell'}$  for some prime number $\ell'$, the postcomposition of a {\it continuous } monodromy representation is no longer continuous (unless $\ell=\ell'$ and $\sigma$ is the identity on $\Q_\ell$), so {\it we cannot define a conjugate} $\mathbb L^\sigma_\ell$ of an $\ell$-adic local system by this simple postcomposition procedure.

 \medskip
 
 Returning to the complex side, as a consequence of  $\pi_1(X(\C), x( \C))$ being finitely generated (we do not even need the finite presentation here), there are finitely many elements \ga{}{(\gamma_1, \ldots, \gamma_s)  \in \pi_1(X(\C), x(\C)) \notag}
  such that the characteristic polynomial map
 \ml{}{  M_B(X,r)^\square_\C  \xrightarrow{\psi^\square} N_\C= \prod_{i=1}^s  (\A^{r-1}\times \G_m)_\C  \notag\\
 \rho \mapsto ( {\rm det}(T-\rho(\gamma_1)), \ldots, {\rm det}(T-\rho(\gamma_s))), \notag}
 defined in Section~\ref{sec:qu}, which factors through
  \ga{}{  M_B(X,r)_\C   \xrightarrow{\psi} N_\C,  \notag}
has the property that $\psi$ is a {\it closed embedding}.  The reason is that  
a semi-simple representation is determined uniquely up to conjugation by the characteristic polynomial function on all $\gamma\in \pi_1(X(\C), x(\C))$. By finite generation of $\pi_1(X(\C), x(\C))$, finitely many suitably chosen  ones among them are enough to recognize a semi-simple representation up to conjugation. 
 So denoting  $\tau\circ {\rm det}( T- \rho(\gamma))$ by  ${\rm det}( T- \rho(\gamma))^\tau$ to unify the notation, we can summarize the discussion as follows:
 
 \medskip
 
 An automorphism $\tau$  of $\C$ yields a  commutative diagram
 \ga{}{ (\star) \  \  \ \xymatrix{ \ar[d]_{\mathbb L_\C \mapsto \mathbb L_\C^\tau}  M^{irr}_B(X,r)(\C) \ar[r]^{\rm incl.} & \ar[d]_{\mathbb L_\C \mapsto \mathbb L_\C^\tau} M_B(X,r)(\C) \ar[r]^\psi &  N(\C) \ar[d]^{ (-)^\tau}\\
M^{irr}_B(X,r)(\C) \ar[r]^{\rm incl.}&  M_B(X,r) ( \C) \ar[r]_{\psi}& N( \C)
 } \notag
 }
The upper script $^{irr}$ means the irreducible locus.  
 \subsection{Analogy over a finite field}
 
 Let us now assume that $X$ is smooth quasi-projective over a finite field $\F_q$.  We denote by $p$ the characteristic of $\F_q$. We fix a prime $\ell$ different from $p$. We denote by $M^{irr} _\ell (X_{\bar \F_p},r)$ the  {\it set} of all  rank $r$   $\ell$-adic local systems $\mathbb L_\ell$  defined over $X_{\bar \F_p}$ which 
 \begin{itemize}
 
 \item[$\bullet$] are arithmetic, that is descend to some $X_{\F_{q'}}$ for some finite extension $$\F_q\subset \F_{q'}(\subset \bar \F_p);$$
 \item[$\bullet$ $\bullet$ ]  on $X_{\F_{q'}}$ are irreducible over $\bar \Q_\ell$.
 \end{itemize}
 For $r=1$ the second condition is of course void. 
 \begin{claim} \label{claim:tor}
 
 $M^{irr} _\ell (X_{\bar \F_p},1)$ consists of rank $1$ $\ell$-adic local systems on $X_{\bar \F_p}$  of finite order prime to $p$. 
 
 \end{claim}
 \begin{proof}
 The datum of  a rank $1$ $\ell$-adic local system of   finite order prime to $p$ is equivalent to the datum  of  a Kummer cover of $X_{\bar \F_p}$ of order prime to $p$, which itself is defined over some $\F_{q'}$.  This shows one direction. The other one is classical, see \cite[Th\'eor\`eme~1.3.1]{Del80}, which is proved using Class Field Theory. It is also pleasant to think of it with a swift argument using weights:
 An element  $\sL \in M^{irr} _\ell (X_{\bar \F_p},1)$
  lies in $H^1(X_{\bar \F_p},  \sO_E^\times)$ for some  $\ell$-adic field $E$.
 We write $$\sO_E^\times= k_E^\times\cdot (1+\frak{m}\sO_E)$$ as multiplicative groups via the Teichm\"uller lift where $\frak{m}\subset \sO_E$ is the maximal ideal of $\sO_E$.  So $\sL=\tau \otimes \sL'$ where $\tau \in H^1(X_{\bar \F_p},  |k_E^\times|) \subset H^1(X_{\bar \F_p}, \sO_E^\times)$ is a torsion rank $1$ local system of order  $m \ge 1$ prime to $p$
 as $|k_E|^\times$ itself has order prime to $p$, and $\sL'\in H^1(X_{\bar \F_p}, (1+\frak{m}\sO_E))$.
 Thus  $\sL^m=(\sL')^m \in H^1(X_{\bar \F_p}, (1+\frak{m}\sO_E))$ descends to $X_{\F_{q'}}$.
  The $p$-adic logarithm ${\rm log}_p$  identifies the multiplicative  group  $ (1+\frak{m}\sO_E)$  with the additive one $\sO_E$. Thus ${\rm log}_p(\sL^m)  \in H^1(X_{\bar \F_p}, \sO_E)\subset 
 H^1(X_{\bar \F_p}, \bar \Q_\ell) $ descends $X_{\F_{q'}}$, 
 which  by a weight argument  forces
 ${\rm log}_p(\sL^m) $ to be equal to $0$. 
This shows the other direction.  \end{proof}

 \medskip
 
Deligne deduces in \cite[Proposition~1.3.14]{Del80} from Claim~\ref{claim:tor} and from the structure of the Tannaka group of an irreducible  $\ell$-adic local system defined over $X_{\F_{q'}}$ that it is  equivalent  for the definition of $M^{irr}_\ell(X_{\bar \F_p},r)$ to replace the condition $\bullet$ by  the condition
\begin{itemize}
\item[$\bullet'$] descend to a Weil sheaf to some $X_{\F_{q'}}$ for some finite extension $$\F_q\subset \F_{q'}(\subset \bar \F_p).$$

 \end{itemize}
 He shows indeed in  {\it loc. cit.} that an irreducible Weil sheaf on $X_{\F_{q'}}$ becomes \'etale after a twist by a character of ${\rm Gal}(\bar \F_p/\bar \F_{q'})$.

 \medskip 
 
The set $M^{irr} _\ell (X,r)$
 is a replacement for $M^{irr}_B(X,r)(\C)$ on the left side of $(\star)$.   While over $\C$ we have a scheme structure underlying this set, over $\bar \F_p$ we just consider the set. 
A replacement  for  $M_B(X,r) (\C)$
should take into account the various factors and  on each factor, the various twists by a character of ${\rm Gal}(\bar \F_p/\bar \F_{q'})$.
 This is the reason why in $(\star)$ we restrict the discussion to $M^{irr}_B(X,r)(\C)$.  

\medskip

 We give ourselves an abstract field isomorphism $\sigma: \bar \Q_\ell \xrightarrow{\cong} \bar \Q_{\ell'}$. The only ``continuous''' information it contains is that it sends a number field $K\subset \bar \Q_\ell$ to another one $K^\sigma\subset \bar \Q_{\ell'}$. 
 So  the right vertical arrow $ (-)^\sigma$  makes sense only on a $\gamma\in \pi_1(X_{\F_{q'}}, x_{\bar \F_p})$,  where $\F_q\subset \F_{q'} (\subset \bar \F_p)$ is a finite extension,  which has the property that  {\it the  characteristic polynomial of
 $\gamma$
 has values in a number field.  }
 Furthermore, we wish the set of such $\gamma$ to be able to {\it recognize completely  $M^{irr}_\ell(X,r)$}  as $\psi$ does over $\C$.  
 We know the following fact.
 \begin{fact} 
 The set of {\it conjugacy classes of the Frobenii at all closed points $|X|$  of $X$  } has those two  properties
by   \cite[Th\'eor\`eme~7.6]{Laf02} and \v{C}ebotarev's theorem.  
\end{fact}
\medskip

So we set $N^\infty =\prod_{|X|} (\A^{r-1}\times \G_m) $ and $\psi^\infty$  for the characteristic polynomial map on those Frobenii of all closed points. 
 Then {\it Deligne's companion conjecture may be visualized on the diagram}
 
 \ga{}{ ( \star \star) \  \  \  \xymatrix{
 \ar@{.>}[d]|-{?\exists  \ \mathbb L_\ell \mapsto \mathbb L_{\ell }^\sigma}
   M^{irr}_\ell (X,r)  \ar[r]^{\psi^{\infty}} &  N^{\infty}( \bar \Q \subset \bar \Q_\ell) 
    \ar[d]^{ (-)^\sigma}\\
M^{irr}_{\ell '}(X,r)  \ar[r]_{\psi ^{\infty}} & N^{\infty}(\bar \Q\subset  \bar \Q_{\ell '})
 } \notag
 }
 The wished $\mathbb L_\ell^\sigma$ is called the {\it companion} of $\mathbb L_\ell$ for $\sigma$. See the $\ell$-adic part in \cite[Conjecture~1.2.10]{Del80} for Deligne's precise formulation. 

\subsection{Geometricity} \label{sec:geom}
The formulation $(\star \star)$ hides certain properties. For the fact that the eigenvalues of Frobenius on our irreducible $\ell$-adic sheaves are {\it algebraic numbers}, we quote {\it loc. cit.} which postdates Deligne's conjecture. This is because if $X$ is a smooth curve over $\F_q$,  L. Lafforgue proves  as a corollary of the Langlands program that an $\mathbb L_\ell$ as before is even  {\it pure}. In fact Deligne in \cite{Del80} and Deligne-Beilinson-Bernstein-Gabber in   \cite{BBD82} prove that geometric local systems are pure or mixed, and Lafforgue proves geometricity on curves.  Up to now,  {\it we do not know geometricity} of arithmetic  local systems over $X$ smooth quasi-projective of higher dimension over $\F_q$. But we know thanks  to Drinfeld  (\cite{Dri12})  how to extend the existence of companions from curves to such an  $X$. Let us remark that even on curves, {\it we cannot } lift the geometricity property from $\bar \F_p$ to $\C$, see Section~\ref{sec:LL}.


\medskip

Nonetheless, in view of Simpson's geometricity conjecture~\cite[p.9]{Sim92} for irreducible rigid complex local systems, the fact that, in absence of geometricity, Drinfeld shows how to prove some integrality property for $\ell$-adic local systems in the sense that once we have a geometrically  irreducible $\ell$-adic local system, we have all  geometrically irreducible $\ell'$-adic local systems with the  ``same'' characteristic polynomials at closed points,  for any $\ell'\neq \ell$, was 
the philosophical reason behind our proof in \cite[Theorem~1.1]{EG18}. 

\subsection{On Drinfeld's proof}
We sketch some aspects of it and refer to the original article \cite{Dri12} for the complete proof.
First let us comment that Drinfeld's proof unfortunately does not solve the conjecture on $X$ normal over $\F_q$ as originally formulated by Deligne, but  ``only'' (if this is possible to say ``only'' here) on $X$ smooth. One problem is that the restriction of simple arithmetic local systems on a non-normal subvariety is not understood, in particular  when / whether it is semi-simple. Not even for complex varieties is this problem of preservation of semi-simplicity by pullback understood, see on this \cite[Theorem~25.30]{Moc07}, where there is a typo, (the proof truly does assume that the domain and the target are smooth and it can be extended to normal varieties). See also \cite[Section~7.3]{dJE22} for an arithmetic proof of the semi-simplicity statement  \cite[Theorem~25.30]{Moc07}  for  normal varieties, and  \cite[Theorem~1.2.2]{D'Ad20} for further studies in the non-normal case. 

\medskip

\medskip

We fix $\mathbb L_{\ell}\in M_\ell^{irr}(X,r)$, an isomorphism $\sigma: \bar \Q_\ell \xrightarrow{ \cong } \bar  \Q_{\ell'}$. By Deligne's Theorem~\cite[Th\'eor\`eme~3.1]{Del12}, $\psi^\infty(\mathbb L_{\ell}) \in N^\infty(\sO)$ where $\sO$ is a number ring in $\bar \Q_\ell$. This fixes the $\ell$-adic completion $\sO_E\subset \bar \Q_\ell$ where $E$ is an $\ell$-adic field, together with    the number ring $\sigma(\sO)\subset \bar \Q_{\ell'}$ and its completion to an $\ell'$-adic ring $A\subset  \bar \Q_{\ell'}$. We denote by $ \frak{m}_E\subset \sO_E, \   \frak{m}\subset A$ the maximal ideals.
The starting point of the proof is the datum (via \cite{Laf02}, {\it loc. cit.})  for any smooth curve $C/\F_q$, any morphism $f_C: C\to X$ which is generically an embedding,  of  a semi-simple $\ell'$-adic local system $\mathbb L_{A, C}$  defined over $A$ on $C$ with the property that on intersections points $C\times_XC'$ the pull-backs from $C$ and $C'$ are isomorphic.  Moreover there are such curves $C$ so  that $\mathbb L_{A,C}$ is geometrically irreducible.  Those local systems $\mathbb L_{A,C}$ are simply obtained as the $A$-forms of  the companions  $\mathbb L_\ell|_C^\sigma$ of 
$\mathbb L_{\ell}|_C$. 

\medskip

The goal is  to construct an irreducible  $A$-adic local system $\mathbb L_{A}$  on $X$ with restrictions on those curves $C$
 isomorphic to $\mathbb L_{A,  C}$.
This will be  the companion  $\mathbb L_\ell^\sigma$ of $\mathbb L_\ell$ for $\sigma$,   forgetting in the notation  that it is definable over $A$. 
 If we can do this, by \v{C}ebotarev's density theorem, $\mathbb L_{A}$ is unique. 

 \subsubsection{Reductions}
 
To this aim, we first construct a companion on  some smooth variety $X'$ endowed with a finite \'etale  Galois cover $h:X'\to X$ with Galois group $\Gamma$. 
In this situation there is a smooth geometrically connected curve $\iota: C\to X$ such that the composed morphism
 $$q: \pi_1(C, x)\xrightarrow{\iota_*} \pi_1(X,x)\xrightarrow{q_0 }\Gamma$$ 
  is surjective, where $x \to C$ is a geometric point, see e.g. \cite[Proposition~B.1]{EK12} and references therein.  Write $h_C: C'=C\times_X X'\to C$ and $\iota': C'\to X'$  for the pulled back morphisms. 
\begin{claim} \label{claim:red}
Let $\mathbb L_{X'}$ be an $\ell$-adic local system, such that $\iota{'^*}\mathbb L_{X'}=:\mathbb L_{C'}$ descends to  an $\ell$-adic local system $\mathbb L_C$ on $C$, i.e. $\mathbb L_{C'}=h_C^*\mathbb L_C$.  Then $\mathbb L_{X'}$ descends to an $\ell$-adic local system $\mathbb L_X$ on $X$, i.e. $\mathbb L_{X'}=h^*\mathbb L_X$.

\end{claim}
\begin{rmk} Drinfeld does not directly write Claim~\ref{claim:red}. The reduction is hidden in his proof of the boundedness result below. 
We also remark, even if this is useless here,  that  if $X$ was defined over $\C$ instead, then $\iota$ would exist as well and the Claim would be true with complex local systems replacing $\ell$-adic ones, with the same proof.

\end{rmk}
\begin{proof}
By conjugating a representation $\rho_{X'}: \pi_1(X',x)\to GL_r(\bar \Q_\ell)$ corresponding to $\mathbb L_X$  by an element in $GL_r(\bar \Q_\ell)$, we may assume that $(\iota')^* \rho_{X'}= (h')^*\rho_C$ where $\rho_C: \pi_1(C,x)\to GL_r(\bar \Q_\ell)$ is a representation corresponding to $\mathbb L_C$. 
The surjectivity of $q$ implies the surjectivity of 
\ga{}{ {\rm Ker} \big( \pi_1(C',x)\to \pi_1(X',x))\big)\to {\rm Ker} \big(\pi_1(C,x)\to \pi_1(X,x))\big) \notag}
which implies that t $\rho_C$  factors through $\bar \rho_C: \bar \pi_1(C,x)\to GL_r(\bar \Q_\ell)$ where $$\bar \pi_1(C,x)=\iota_*\pi_1(C,x) \subset \pi_1(X,x).$$  For each element $\gamma\in \Gamma$ we choose an element $c_\gamma\in  \bar  \pi_1(C,x)$ such that $q_0(c_\gamma)=\gamma$.  Then any element $g\in \pi_1(X,x)$ can be uniquely written as $
g=c_\gamma\cdot h_*( \alpha)$ for some $\gamma\in \Gamma$ and $\alpha\in \pi_1(X',x)$. We set
\ga{}{ \rho_X(g) =\bar \rho_C(c_\gamma) \cdot \rho_{X'}(\alpha).\notag}
It is straightforward to check that $\rho_X$ defines a continuous representation.

\end{proof}

We construct $h$ as follows. First we ``tamify'' $\mathbb L_X$:  We define the finite \'etale Galois cover $h_0: Y_0\to X$  associated to the residual representation in $GL_r(\sO_E/\frak{m}_E)$ if $\ell \ge 3$ else to $GL_r(\sO_E/\frak{m}^2_E)$ if $\ell=2$.  Then, after replacing $\F_q$ by a finite extension $\F_{q'}$, there is a Zariski dense open on
$ U\subset Y_0\otimes_{\F_q} \F_{q'}$  and an elementary fibration on $U$  in the sense of Artin~\cite[Expos\'e XI, Proposition~3.3]{SGA4}. This means that there is a fibration $f: U\to S$ 
of relative dimension $1$, which is smooth, with a relative compactification $\iota: U\hookrightarrow \bar U$ such that $\bar U\setminus U\to S$ is finite \'etale. 
We take $h: Y\to X$ to be the Galois hull of the composition of $h_0$ with $Y_0\otimes_{\F_q} \F_{q'}\to Y_0$. 
The $\ell$-adic local system $h^*\mathbb L_\ell$ is now semi-simple as an \'etale sheaf, all of its summands are tame. 
 So we may assume $X=Y$,  $\mathbb L_\ell$ is tame and 
there is an elementary fibration on a dense open $U$ of $X$.  
\medskip

We can replace $X$ by  $U\subset X$.   This follows from the ramification results by Kerz-Schmidt  \cite[Lemma~2.1]{KS09} and \cite[Lemma~2.9] {KS10}  based on Wiesend's work \cite{Wie06}:  let $\mathbb L_{\ell, U}^\sigma$  be the companion of $\mathbb L_\ell$ on $U$. If it ramifies on $X\setminus U$ then there is a curve $C\to X$ such that $\mathbb L^\sigma_{\ell, U}|_{C\times_X U}$ is ramified along $C\setminus C\times_X U$. By \cite[Th\'eor\`eme~9.8]{Del72}, companions on curves preserve ramification.  This is a contradiction.  So we may assume that $X$ admits an elementary fibration $f:X\to S$ and $\mathbb L_\ell$ is tame. Applying again Claim~\ref{claim:red} we may assume that the elementary fibration has a section $\sigma: S\to X$, and in addition, we can shrink $S$ to a dense open and $X$ to the inverse image. 

\subsubsection{Boundedness}

The main step (\cite[Proposition~2.12]{Dri12}) in Drinfeld's proof  is a boundedness result, which is a generalization  to the non-abelian case of  Wiesend's result \cite[Proposition~17]{Wie06}.
 After possibly shrinking $S$ to a dense open and $X$ to the inverse image,
   Drinfeld constructs  an infinite sequence
\ga{}{ \ldots \to X_{n+1}\to X_n\to \ldots \to X_1\to X\notag}
of finite \'etale covers, such that
\begin{itemize}
\item[(i)] for any $n\in \N_{>0}$, for any $C\to X$ as above,  the $G_n=GL_r(A/\frak{m}^n)$-torsor defined by
$\mathbb L_{A,C\times_X X_n}$ is trivial;
\item[(ii)] for $x\to X$ a geometric point,  setting $H_n=\pi_1(X_n, x)$, the group $ \cap_n H_n\subset \pi_1(X, x)$ contains a subgroup $H$ which is closed and normal in $\pi_1(X, x)$ such that
$\pi_1(X, x))/H$  is almost pro-$\ell$, that is an extension of a finite group by a pro-$\ell$-group.

\end{itemize}
Property (ii) does not depend on the choice of $x$.

\medskip

If ${\rm dim}(X)=1$,  we apply the existence of companions on curves,  then $ H$ is  defined to be the kernel of the  monodromy representation, $H_n$ is defined to be the kernel of the 
 truncated representation $\pi_1(X_{\F_q}, x_{\bar \F_q})\to GL_r(A)\to G_n$ so $H=\cap_n H_n$, and $X_n\to X$  is the finite Galois \'etale cover defined by $H_n$.
 
 \medskip
 
We assume ${\rm dim}(X)\ge 2$ and that  $X$ is an elementary fibration  denoted  $f: X\to S$,
 so $X\to S$ is a relative curve with a good relative compactification  $\iota: X\hookrightarrow \bar X$ above $S$,  such that $\bar X\setminus X \to S$ is finite \'etale, with a section $\sigma: S\to X$.    
 
 \subsubsection{Moduli} We are grateful to  Dustin Clausen for a thorough discussion on Claim~\ref{lem:mod}. The idea of the  proof presented here, based on Drinfeld's idea, is due to him. 
 \begin{claim}[\cite{Dri12}, Lemma~3.1, Proof] \label{lem:mod}
  Replacing $S$ by a dense open, there is a tower of finite \'etale surjective morphisms 
\ga{}{ \ldots \to T_{n+1}\to T_n\to \ldots \to T_1\to S \notag}
 such that $T_n$ is the fine moduli space of tame $G_n$-torsors on $X$ which are trivial above the section $\sigma(S)$.  
 \end{claim}
  \begin{proof}
 We consider the three \'etale sheaves in groupoids  over $S$. For $S'\to S$ 
 \begin{itemize}
 \item[1)] $\sF_n(S')$ is the groupoid of tame $G_n$-torsors  $X'\to X$ equipped with a trivialization over the section $\sigma'=\sigma \times_S S': S'\to X'=X\times_S S'$;
 \item[2)] $\sG_n(S')$ is the groupoid of tame $G_n$-torsors over $X'$;
 \item[3)] $\sH_n$ is the groupoid of $G_n$-torsors over $S'$.
 \end{itemize}
 By \cite[Expos\'e XIII, Corollaire~2.7]{SGA1}, $\sH_n$ and $\sG_n$ are $1$-constructible and compatible with base change. On the other hand, $\sF_n$ is the equalizer of 
 the restriction map $(\sigma')^* : \sG_n(S')\to \sH_n(S')$ with the map $\sG_n(S')\to \sH_n(S')$ assigning to any $G_n$-torsor on $X'$ the trivial one on $S'$. Thus by 
 \cite[Expos\'e XIII, Lemme~3.1.1]{SGA1}, $\sF_n$ is $1$-constructible and compatible with base change as well. In addition, $\sF_n$, viewed as a presheaf, is already  a sheaf as the trivialisation on the section forces 
 ${\rm Aut}(\sF_n(S'))=\{1\}$. So there are  dense open subschemes  $S_{n+1}\subset S_n \subset S$ such that $\sF_n$ restricted to $S_n$ is represented by a scheme $T_n$ which is finite \'etale over $S_n$. By the reductions we may assume $S_1=S$. So tame $G_1$-torsors on $X'$  together with a trivialization over $\sigma'$ are identified with $S'$-points of $T_1\to S$.
 
 \medskip
 
 We explain now why  $\sF_{n+1}\to \sF_n$ is  \'etale finite surjective for $n\ge 1$.
 We denote by $V_{n+1}$ the kernel of $G_{n+1}\to G_n$.  It is a commutative group, in fact 
 a  finite dimensional vector space over $ \F_{\ell'}$.
 It is not central, which is to say that the induced homomorphism $\chi_n: G_n \surj \Gamma_n\subset {\rm Aut}(V_{n+1})$ defined by lifting to $G_{n+1}$ and conjugating on $V_{n+1}$ is not trivial.  The obstruction to lift any torsor $P_n \in H^1(X, G_n)$ to a $G_{n+1}$-torsor on $X$   lies in $H^2(X, V_{n+1}^{\chi_n})$,  where $V_{\ell+1}^{\chi_n}$ 
 if the $V_{\ell+1}$-torsor defined by pushing $P_n$ along $\chi_n$.  By cohomological dimension of $f$ there is an exact sequence
 \ga{}{ H^2(S, f_* V_{n+1}^{\chi_n})\to H^2(X, V_{n+1}^{\chi_n}) \to H^1(S, R^1f_*V_{n+1}^{\chi_n}). \notag}
 As $\ell'$ is prime to $p$, all $V_{n+1}$ are tame, thus there are only finitely many such local systems $ V_{n+1}^{\chi_n}$. As $f$ is an elementary fibration,
  $ V_{n+1}^{\chi_n}$ is tame, the constructible sheaves  $R^if_*  V_{n+1}^{\chi_n}$ for $ \ i\ge 0$  are local systems. Thus there is a finite \'etale cover $ T_\circ \to S$, defining $b_\circ: X_{T_\circ}=X\times_S T_\circ \to X$,  such that the image of 
  $ H^2(X, V_{n+1}^{\chi_n})$ in $ H^2(X_{T_\circ} , b_\circ^* V_{n+1}^{\chi_n}) $ is zero.

  \medskip
  
  Once a $G_n$-torsor on $X$ lifts to a $G_{n+1}$-torsor on  $X_{T_\circ}$, to trivialize it on the section $\sigma_{T_\circ}=\sigma\otimes_S {T_\circ}$, we argue similarly: there are finitely many $V_{n+1}$-torsors on $T_\circ$ so after a finite \'etale base change
$  T\to T_\circ$ defining $b: X_T=X\times_ST\to X$, they are all trivial.  We conclude that for any $S'\to S$ and any $S'$-point $S'\to T_n$ over $S$, there is a lift $S\otimes_S T\to T_{n+1}$ thus $T_{n+1}\to T_n$ is finite \'etale for $n\ge 1$.  This finishes the proof.




 \end{proof}
 So there is a universal $G_n$-torsor over
 $ T_n\times_S X$ with Weil restriction $Y_n\to X$ to $X$. 
This yields a tower of  finite \'etale surjective morphisms 
\ga{}{ \ldots \to Y_{n+1}\to Y_n\to \ldots \to Y_1\to X. \notag}
On the other hand, by the induction assumption, we have a tower 
\ga{}{ \ldots \to S_{n+1}\to S_n\to \ldots \to S_1\to S \notag}
which verifies (i) and (ii) for $X$ replaced by  $S$. 
Then
$X_n=Y_n\times_S S_n$
is inserted in a tower 
\ga{}{ \ldots \to X_{n+1}\to X_n\to \ldots \to X_1\to X \notag}
of finite \'etale surjective morphisms. The claim is that this is the sought-after tower. 

\medskip

Property (i) is clear as 
by definition, for any geometric point $s\to S$,  
 \ga{}{{\rm Hom}_S(s, Y_n)= \{ {\rm tame} \ G_n{\rm -torsors \ over} \  U_{ s}\}\notag}
so  for such a geometric point $s$,
 $ (Y_n)_{ s}\to X_{s}$   is   the   universal  finite \'etale  cover which trivializes every tame $G_n$-torsor over $X_ s$ .

\medskip
We address Property (ii). 
Let $\eta\to S$ be a generic point. We take $x$ to be the image by the section $\sigma$ of a geometric point  $\bar \eta$ above $\eta$. 
As $X$ is normal,  $\pi^t_1(X_\eta, x)\to  \pi^t_1(X, x)$ is surjective, so we can replace $X$ with $X_\eta$ for (ii). 
This yields the homotopy exact sequence of fundamental groups
\ga{}{ 1\to K:= \pi^t_1(X_{\bar \eta}, x)\to \pi^t_1(X_\eta, x)\to  \Gamma:=\pi_1( \eta, x)\to 0\notag}
where the upper index $^t$ indicates  tameness  with respect to $\iota$ and 
  the lower index indicates the fibre.  
We  define
\ga{}{H_n={\rm Ker} \big( \pi^t_1(X_\eta, x)\to {\rm Aut} ( (X_n)_x) \big).\notag}
Then $H_n$ is inserted in an exact sequence
\ga{}{1 \to K_n\to H_n\to \Gamma_n \to 1\notag}
where 
\ga{}{K_n= H_n\cap  K  =\cap_{\rho}  {\rm Ker} \big(  \pi^t_1(X_{\bar \eta}, x)    \xrightarrow{\rho} G_n \big) \notag}
\ga{}{\Gamma_n={\rm Ker}\big( \Gamma  \to {\rm Aut}  K/K_n \big).\notag}
Since $K$ is topologically  finitely generated, there are finitely many $\rho$'s in the definition of $K_n$ thus 
$K  /\cap_n K_n$ is almost pro-$\ell$.  Finally, $\Gamma/\cap_n \Gamma_n$ is a quotient of ${\rm Im} \big( \Gamma \to 
{\rm Aut}(K/\cap_n K_n) \big)$ while ${\rm Aut}(K/\cap_n K_n)$ is almost pro-$\ell$ as well. Indeed an automorphism has to respect the maximal pro-$\ell$ sub and
in it the kernel to its first $\F_\ell$-homology,  which itself is pro-$\ell$  and of finite index.  This proves (ii) for $H=\cap_n H_n$ and finishes the proof of this first part for $X$ with an elementary fibration $X\to S$. 

\subsubsection{Gluing}
 The main point is then to prove that the preceding property suffices to glue the $\mathbb L^\sigma_{\ell, C}$. The proof follows an idea of  Kerz,  see \cite[Section~4]{Dri12}, Drinfeld himself initially had a more difficult argument. First we do a Lefschetz  type argument to find a curve $\varphi: C\to X$ containing $x$ such that  $\varphi_*: \pi_1(C,x)\to \pi_1(X,x)/H$ is surjective.  This is possible as the surjectivity is 
recognized on the finite quotient which is the extension of the residual quotient of $\pi_1(X)/H$ (if $\ell \ge 3$, if $\ell=2$ we go mod $\frak{m}^2 $ here)  by the maximal $\F_\ell$-vector space quotient, see \cite[Lemma~8.2]{EK11} (see also above where a similar property was already used in order to show that ${\rm Aut}(K/\cap_n K_n)$ is almost pro-$\ell$).   By (i) and (ii)  the restriction of $\mathbb L^\sigma_{\ell, C}$ to  ${\rm Ker}(\varphi_*)$ is unipotent, but it has to be semi-simple as well as ${\rm Ker}(\varphi_*) \subset \pi_1(C,x)$ is normal and $\mathbb L_{\ell, C}^\sigma$ itself is semi-simple. Thus 
the monodromy group of  $\mathbb L_{\ell, C}^\sigma$ is a quotient of $\pi_1(C,x)/{\rm Ker}(\varphi_*)=\pi_1(X,x)/H$, and this {\it defines} a representation $\pi_1(X,x)\to GL_r(A)$ {\it which a priori depends on $C$}.  Denote by  $\mathbb L_{\ell' ,X}$ the underlying local system. By definition its restriction to $C$ is equal to $\mathbb L_{\ell, C}^\sigma$. To compute the restriction
of $\mathbb L_{\ell', X} $  to all closed points, we take other curves $\varphi_{C'}: C'\to X$ with the same property so they fill in all the closed points. If we make sure that $C$ and $C'$ meet in sufficiently many points,  then $(\mathbb L_{\ell', X}|_C=\mathbb L_{\ell, C}^\sigma , \mathbb L_{\ell', X}|_{C'}, \mathbb L_{\ell, C'}^\sigma)$  are recognized by their value on the intersection points by \cite[Satz~5]{Fal83}.  This finishes the proof  for $X$ with an elementary fibration $X\to S$.

\subsection{Cohomologically rigid local systems are integral, see \cite{EG18}, Theorem~1.1} \label{sec:intEG}
We sketch the proof with Michael Groechenig in {\it loc.cit.}: let $X$ be a smooth quasi-projective variety defined over the field of complex numbers, let $T_i\subset GL_r(\C)$ be conjugacy classes of quasi-unipotent matrices, and $\sL$ be a torsion rank $1$ local system.  We refer to  Section~\ref{sec:geom} for the heuristic of the following theorem. 
\begin{thm} \label{thm:EG18}
Irreducible cohomologically rigid  local systems in $M_B(X,r, \sL, T_i)$ are integral. 
\end{thm}

\begin{proof}

 Let $X$ be a smooth quasi-projective variety over $\C$, fix a torsion character $\sL$ and quasi-unipotent conjugacy classes $T_i\subset GL_r(\C)$. 
Let $\mathbb L_{\C}$ be a point of $M_B(X,r,\sL, T_i)(\C)$ which is isolated. It is defined by a representation $\rho_\C: \pi_1(X(\C), x(\C))\to GL_r(A)$ where $A$ is a ring of finite type over $\Z$. Let $a: A \to \bar \Z_\ell$ be a $\bar \Z_\ell$-point of $A$. It defines an $\ell$-adic local system $\mathbb L_{\ell, \C}$ by the factorization
\ga{}{ \xymatrix{ \ar[d]_{=}\pi_1(X(\C), x( \C)) \ar[r] &\ar[d]^a  GL_r(A) \\
\ar[d]_{\rm profinite \  compl.} \pi_1(X(\C), x( \C)) \ar[r] & GL_r(\bar \Z_\ell)\\
\pi_1(X_\C, x_\C)  \ar[ru]_{ \mathbb L_{\ell, \C}}
} \notag}
Using the notation of \ref{sec:why},  for a closed point $s\in |S|$ of characteristic prime to the order of  the residual representation of $\mathbb L_{\ell, \C}$, the order of $\sL$ and of the eigenvalues of the monodromies at infinity, and integral for the (finitely many) cohomologically rigid local systems, the following properties hold. 
\begin{itemize}
\item[1)] $\mathbb L_{\ell, \C}$ descends to $\mathbb L_{\ell,  s'}$ (\cite[Theorem~4]{Sim92}) for some $s'\to s$ finite below $\bar s$ (one has to complete the argument to take care of the conditions at infinity) keeping $\sL$ and $T_i$ (\cite[Section~1.1.10]{Del73});
\item[2)]  The companions $\mathbb L_{\ell, s'}^\sigma$ still have determinant $\sL$ and the semi-simplification of the monodromies at infinity is the one of the $T_i$'s   (\cite[Th\'eor\`eme~9.8]{Del72});
\item[3)]  The intersection cohomology verifies: $$IH^1(X, \sE nd^0(\mathbb L_{\ell, \bar s}))=0=IH^1(X, \sE nd^0(\mathbb L_{\ell, \bar s}^\sigma))$$ by the weight argument \cite[Proof~of~Theorem~1.1]{EG18};
\item[4)] So $\mathbb L_{\ell, \C} \mapsto sp_{\C, \bar s}^{-1}( \mathbb L_{\ell, \bar s}^\sigma))$ is a bijection from the set of cohomologically rigid $\ell$-adic local systems to the set   of 
cohomologically rigid $\ell'$-adic local systems.

\end{itemize}
So there cannot be a non-integral place $\ell'$. This finishes the proof. 
\end{proof}

\begin{rmks} 1)
The ``same'' proof works if $GL_r$ is replaced by a reductive group $G\subset GL_r$, see  \cite{KP22} by Klevdal-Patrikis. \\ \ \\
2) It is not the case that all rigid local systems are cohomologically rigid, see \cite{dJEG22}. In the present state of knowledge, the examples known are all cohomologically rigid in a smaller reductive group $G\subset GL_r$ as in 1).  However, I have no philosophical argument why rigid local systems should be cohomologically rigid in some  smaller reductive group $G\subset GL_r$ or why not.

\end{rmks}
\subsection{Integrality of the whole Betti moduli space, see \cite{dJE22}, Theorem~1.1} \label{sec:dJE}
In  Theorem~\ref{thm:wa} we proved using de Jong's conjecture that weakly arithmetic local systems are dense in their Betti moduli.  In Theorem~\ref{thm:EG18} we proved 
 using the existence of the $\ell$-adic companions that cohomologically rigid local systems are integral.  Combining the two essential ideas in those theorems we define  a  (weak) notion of integrality for the whole Betti moduli space of irreducible local systems and prove  the following theorem. 
 
 \medskip
 
Let $X$ be a smooth quasi-projective variety defined over the field of complex numbers, let $T_i\subset GL_r(\C)$ be conjugacy classes of quasi-unipotent matrices, and $\sL$ be a torsion rank $1$ local system.  We denote by $M^{irr}_B(X,r, \sL, T_i)$ the Betti moduli spaces of {\it irreducible} local systems. See \cite[Section~2.1]{dJE22} and references therein for the definition: if  $\mathbb L$ is a point and is defined over a ring $A$, then $\mathbb L$ is irreducible if by definition  $\mathbb L\otimes_A \kappa$ is irreducible for all field  valued points $A\to \kappa$. 

 \begin{thm} \label{thm:dJE22}
If there is one complex point of $M_B(X,r, \sL, T_i)$ which is irreducible, then for all prime numbers $\ell$, there is an irreducible $\ell$-adic local system of rank $r$ with determinant $\sL$ and monodromies at infinity being conjugate to $T_i$ up to semi-simplification.   
 \end{thm}
\begin{proof}
The existence of a complex point $\mathbb L_\C$ which is irreducible is equivalent to saying that the morphism $M^{irr}(X,r, \sL, T_i)\to {\rm Spec}(\sO_K)$ is dominant, where $K$ is the number field over which $\sL$ and the $T_i$ are defined.   By generic smoothness, we find a closed point $z$ of characteristic $\ell >0$ on $M^{irr}_B(X,r,\sL,T_i)$ which is smooth for the morphism 
$M^{irr}(X,r, \sL, T_i)_{red} \to {\rm Spec}(\sO_K)$. We choose $s\in S$ a closed point of the scheme of definition of  a good compactification $X\hookrightarrow \bar X$,  $\sL$ and $T_i$ such that  $z$ descends to $X_{\bar s}$. Applying now the argument as in the proof of Theorem~\ref{thm:wa}, there is on $X_{\bar s}$ an irreducible rank $r$  arithmetic $\bar \Q_\ell$-local system with $z$ as a residual local system, and with monodromies at infinity which are in the conjugacy class of $T_i$ up to semi-simplification.  Choosing $\sigma: \bar \Q_\ell\cong \bar \Q_{\ell'}$ now as in the proof of Theorem~\ref{thm:EG18}, the $\sigma$-companion on the same $X_{\bar s}$ has the right determinant and monodromies at infinity.  We  pull it back to $X_\C$  via $sp_{\C, \bar s}$  as in the proof of Theorem~\ref{thm:EG18}. This finishes the proof.

\end{proof}

\begin{rmk}
We can enhance Theorem~\ref{thm:dJE22} in various ways. Rather than assuming there is {\it one} irreducible complex local system, we can ask for {\it  infinitely many } pairwise non-isomorphic $\mathbb L_\C$, then the conclusion follows, there are infinitely many such irreducible $\mathbb L_\ell$ for all $\ell$ (\cite[Theorem~1.4]{dJE22}). We can also ask for the algebraic monodromy of $\mathbb L_\C$ to e.g. contain $SL_r(\C)$, then the conclusion follows for $\mathbb L_\ell$ (\cite[Theorem~1.5]{dJE22}) etc. But the main point which would be wishful to understand is not understood: with this method  initiated in \cite{EG18} to use companions to obtain integrality, we do not know whether on $X_{\bar s}$ in characteristic $p$ where we perform this construction, the ``dimension of the components''  of $M^{irr}_B(X,r, \sL, T_i)$ over $\bar \Z_\ell$  is preserved as we go over  to 
$\bar \Z_{\ell'}$ via the companions.
Stated like this it makes no proper sense. One should specify that at $\ell$ we complete $M^{irr}_B(X,r, \sL, T_i)$ at a closed point which is smooth for  $M^{irr}_B(X,r, \sL, T_i) \to {\rm Spec}(\sO_K)$, then the component containing $z$ is well defined and had a dimension $d$. One could dream that the $\sigma$-companion is also on a dimension $d$ component of
 $M^{irr}_B(X,r, \sL, T_i)$. For example for $d=0$ this would prove entirely Simpson's integrality conjecture.

\end{rmk}
\subsection{Obstruction} 
Forgetting the conditions at infinity, if we fix a finitely presented group $\Gamma$, a natural number $r\ge 1$, and a rank $1$ torsion complex character $\chi:  \Gamma \to  \C^\times$, we posed in \cite[Definition~1.2] {dJE22} the following definition.

\begin{defn} 
$\Gamma$ has the weak integrality property with respect to $(r,\chi )  $ if, assuming there is an irreducible representation $\rho : \Gamma \to GL_r(\C) $ with determinant $\chi$, then for any prime number $\ell$, there is a representation  $\rho_\ell : \Gamma \to GL_r(\bar \Z_\ell)$  with determinant $\chi$  which is irreducible over $\bar \Q_\ell$. 
\end{defn}

Then, by the work of Becker-Breuillard-Varj\`u \cite{BBV23}, there are pairs $(\Gamma, \chi)$ which do {\it not have the weak integrality property}. 
One example is given by $\Gamma$  being the group spanned by two letters  $a,b$ with one relation $b^2a^2=a^2b$,   $\chi$  is the trivial character and $r=2$.
 The example is not integral at $\ell=2$, see \cite[Example~7.4]{dJE22}. Given Theorem~\ref{thm:dJE22}, one easily computes that if $\Gamma$ is the topological fundamental group of $X(\C)$ for $X$ smooth quasi-projective over $\C$ and $\chi$ is any torsion character, then it has the  weak integrality property (see \cite[Theorem~1.3]{dJE22})).  So we obtain the following theorem. 
 \begin{thm}
 Let $\Gamma$ be a finitely presented group, $\chi$ be a torsion character of $\Gamma$, $r \ge 1$ a natural number. If $\Gamma$  does not have the integrality property with respect to $(r,\chi)$,  $\Gamma$ cannot be the topological fundamental group of a smooth complex quasi-projective variety. 
 
 \end{thm}
 
 The proof  relies on Theorem~\ref{thm:dJE22}. This so defined obstruction is of course by no means coming from easy theorems, at least in the present state of knowledge. As explained before, Theorem~\ref{thm:dJE22} relies both on the Langlands duality for a curve over a finite field and on the geometric Langlands duality. But it has the  advantage  that in the definition 
 of the obstruction we do not need to specify which elements of $\Gamma$ are supposed to come from the local fundamental groups at infinity: none of those  can have this property.

\newpage

\section{Lecture 8:  Rigid Local Systems and $F$-Isocrystals  } \label{sec:Fiso}

\begin{abstract} Rigid connections over  a variety  $X_{\C}$ smooth projective  over $\C$, while restricted to the formal $p$-completion $\hat X_W$  of  a non-ramified projective $p$-adic model $X_W$  of $X_\C$, yield $F$-isocrystals. This is proved in \cite[Theorem~1.6]{EG20}, using the theory of Higgs-de Rham flows on the mod $p$ reduction of $X_W$. We give here a $p$-adic proof of this theorem, obtained with Johan de Jong, which relies on the fact that  for $p\ge 3$, the Frobenius pull-back of a connection on $\hat X_W$ is well defined, whether 
 the $p$-curvature of the mod $p$ reduction is nilpotent or not. However this  proof so far does not give the crystallinity property proved in  \cite[Theorem~5.4]{EG20} which can be expressed by saying that  the $p$-adic local systems obtained on the $p$-adic variety $X_{\overline{{\rm Frac}(W)}}$  for $p$ large descend to a crystalline local system  over $X_{{\rm Frac}(W)}$. The version under the strong cohomological rigidity of the same theorem is worked out in \cite{EG21}. We defer this discussion to Lecture~\ref{sec:crysLS}.

\end{abstract}

\subsection{Crystalline site, crystals and  isocrystals}
We refer to \cite[Section~2.6]{EG20} for the details of the definitions, and for the appropriate references.  

\medskip

Given $X$ smooth over a perfect  characteristic $p>0$ field $k$, we denote by $W=W(k)$ the ring of Witt vectors and by $W_n=W_n(k)=W/p^n$ its $n$-th truncation.  We define the site $(X/W_n)$ with objects the $PD$-thickenings  $(U,T, \delta)$ over $W_n$, where 
$U\hookrightarrow T$ is a closed embedding defined by an ideal $I \subset \sO_T$ on which $\delta$ defines a $PD$-structure over $(W_n, pW_n)$.  So for all $x_i,  \ i=1,\ldots, s$ over $(W_n, pW_n)$
it holds 
 $$m x_1^{[{n_1]}}\cdots x_s^{[{n_s]}}=0$$ for some powers $n_i, m \in \N$ (see \cite[I 1.3.1  ii), p.56] {Ber74}).

 This yields functors $(X/W_n) \hookrightarrow (X/W_{n+1})$ and the {\it crystalline site}  $(X/W)$ as the colimit over $n$ of the $(X/W_n)$.  The Homs are the obvious ones respecting the whole structure. 
A {\it crystal}  $\sF/W$   in coherent modules  on $X/W$ is the datum for all $(U, T, \delta)$ of a coherent sheaf  $\sF_T$ so the transition functions are isomorphisms. The category of {\it  isocrystals} has the same objects and the Hom sets, which are abelian groups,  are tensored over $\Z$ by $\Q$.  For us the main concrete description is \cite[Theorem~2.19, Corollary~2.20]{EG20} according to which, {\it  if } $\hat X_W$ {\it is an essentially smooth formal scheme over} $W$  {\it lifting} $X$, a crystal is
\begin{itemize}
\item[i)]  a flat formal  connection $(\hat E_W, \hat \nabla_W)$ on $\hat X_W$;
\item[ii)] such that $(\hat E_W, \hat \nabla_W)\otimes_Wk$ is filtered by subconnections so the associated graded is spanned over a dense open of $X$  by a full set of algebraic solutions  (that is its $p$-curvature is nilpotent, see Section~\ref{sec:pcurv}).
\end{itemize}
By the classical theorem, i) implies that that $\hat E_K:=\hat E_W\otimes_W K$ where $K={\rm Frac}(W)$ is locally free. This is not the case for $\hat E_W$ even with the condition ii), see \cite[1.3]{ES15}. It is also not known whether or not there is  always find a locally free lattice in $\hat E_K$ which is stabilized by $ \hat \nabla_K$.

\subsection{Nilpotent crystalline site,  crystals and  isocrystals}
We refer to \cite[Section~1-2]{ES18} and references in there.  

\medskip

The {\it nilpotent crystalline site} has objects $(U,T,\delta)$ as  for the  crystalline site,  but in addition $I$  itself is locally  nilpotent. 
So for all $x_i,  \ i=1,\ldots, s$ over $(W_n, pW_n)$,
one has  $$x_1^{{[n_1]}}\cdots x_s^{{[n_s]}}=0$$ for some powers $n_i\in \N$.  This yields a continuous functor $$(X/W)_{Ncrys}\to (X/W)_{crys}$$ of topoi. 
Crystals and isocrystals are defined similarly but given $\hat X_W$ as above, then a crystal on the nilpotent crystalline site is
\begin{itemize}
\item[i)]  a flat formal connection $(\hat E_W, \hat \nabla_W)$ on $\hat X_W$.
\end{itemize}
(See \cite[ I 3.1. 1 ii) p.56, III 1.3.1 p. 187, IV 1.6.6 p. 248]{Ber74}.)
There is {\it no} condition on the mod $p$ reduction. 

\subsection{The Frobenius action on  the set of crystals on the nilpotent crystalline site} \label{sec:Frob}
As $\hat X_W$ is formally smooth over $W$, the Frobenius of $X$ locally in the Zariski topology lifts to $\hat X_W$. The following proposition is due to J. de Jong. As far as I could see,  it is  not documented in the literature. However, the proof is exactly as the one in \cite[tag/07JH]{SP} where unfortunately there the condition ii) is assumed to be fulfilled, but .... not  really really used. 

\begin{prop} \label{prop:dJ}
Let $(\hat E_W,\ \hat \nabla_W)$ be a formal  flat connection on $\hat X_W$.   Assume $p\ge 3$. There is a unique formal  flat connection denoted by $F^*(\hat E_W, \hat \nabla_W)$ on $\hat X_W$ with the property that if on an open $\hat U_W\subset \hat X_W$,  the Frobenius $F$ of $U$ lifts to $F_W$, then  $$F^*(\hat E_W, \hat \nabla_W)|_{\hat U_W} =F^*_W(( \hat E_W,\ \hat \nabla_W)|_{\hat U_W}).$$

\end{prop}
\begin{proof}({\it loc. cit.}) The goal is to show that if $F_W, G_W$ are two  lifts to $\hat U_W$  of $F$ on $U$,  there is  a commutative diagram
\ga{}{\xymatrix{\ar[d]_{\hat \nabla}  F^*_W \hat E_W \ar[r]^\psi & G_W^* \hat E_W \ar[d]^{ \hat \nabla} \\
\Omega^1_{\hat U_W} {\hat \otimes}_{\sO_{\hat U_W}} F_W^* \hat E_W \ar[r]^{1\otimes \psi} & \Omega^1_{\hat U_W} {\hat \otimes}_{\sO_{\hat U_W} } G_W^* \hat E_W 
} \notag}
which is canonical.  So then it fulfils the cocycle condition.  Here $\hat{\otimes}$ denotes the completed tensor product. So we may assume that $\hat U_W$ is \'etale, finite 
onto its image,   an open of $\hat \A^d_W$,  where $d={\rm dim}X$, so as to have coordinates $(x_1, \ldots, x_d)$.  This defines the derivation $\p_i \in T_{\hat U_W}$ dual to $dx_i$, acting on $\hat E_W|_{\hat U_W}$, and the action of  the differential operator  $\p_{\underline{k}} =\p_1^{k_1}\cdots \p_{d}^{k_d}$  on $\hat E_W|_{\hat U_W}$ for all multi-indices $\underline{k}=(k_1,\ldots, k_d)$.  We write 
$F^*_W \hat E_W=
 F_W^{-1} \hat E_W \otimes_{F^{-1}\sO_{\hat U_W}} \sO_{\hat U_W}$.
Then we define
\ga{}{\psi(e\otimes 1)= \sum_{\underline{k}} \p_{\underline{k}}(e)\otimes_{G_W^{-1} \sO_{\hat U_W}}\prod_{i=1}^d (F_W(x_i)-G_W(x_i))^{k_i}/(k_i)!.\notag}
By definition $ (F_W(x_i)-G_W(x_i)) \in p\sO(\hat U_W)$  so its $p$-adic valuation is $\ge 1$. On the other hand,  the $p$-adic exponential function converges on elements of $p$-adic valuation $>1/(p-1)$. 
So $\psi$, thus $F^*$,  is defined for $p\ge 3$. This finishes the proof. 
\end{proof}
\begin{rmk}
If the connection was defined on $X_\sO$ where $\sO\supset W$ is the ring of integers of a $p$-adic field of degree $\ge 2$ over ${\rm Frac}(W)$,  the  same formula with the obvious change of notation  would yield 
$ (F_\sO(x_i)-G_\sO(x_i)) \in \pi \sO(U_W)$ where $\pi$ is a uniformizer of $\sO$. Then  the convergence would be violated and we could not write  the diagram above. In this case we need condition ii) to define the Frobenius pull-back of $(E_\sO, \nabla_\sO)$ in this way when the index of ramification of $\sO$ lies in $[2, p-1($.

\end{rmk}

\subsection{ The Frobenius induces an isomorphism on cohomology in characteristic $0$.}
We set $K={\rm Frac}(W)$ and keep the notation from  the previous paragraph.   In addition we set $$(\hat E_W, \hat \nabla_W)\otimes_W K=(\hat E_K, \hat \nabla_K)$$ and abuse notation
$$F^*(\hat E_K, \hat \nabla_K):=F^*(\hat E_W,\hat \nabla_W)\otimes_W K.$$
\begin{prop} \label{prop:uni}
Let $(\hat E_W, \hat \nabla_W)$ be a  flat formal  connection on $\hat X_W$. Then  for all cohomological degrees $i$, 
\ga{}{ F^*:  H^i(\hat X_W, (\hat E_W, \hat \nabla_W))\otimes_W K \to 
H^i(\hat X_W, F^*(\hat E_W, \hat \nabla_W)) \otimes_W K  \notag}
is an isomorphism and is compatible with cup-products.

\end{prop}
\begin{proof}
By the Mayer-Vietoris spectral sequence it is enough to prove the first part of the proposition on  an open $\hat U_W$ lifting $U$  on which the Frobenius  lifts to $\hat F_W$
  so as to have  an  \'etale map $h: \hat U_W\to  \hat \A^d_W$ finite onto its image, as in Section~\ref{sec:Frob}.  To be more precise, we first construct a finite \'etale morphism 
  $U_k\xrightarrow{h_k}  V_k$ onto an affine open $V_k \hookrightarrow  \A^d_k$,  then a lift of the  diagram  $U_k\to V_k\hookrightarrow \A^d_k$  to the diagram of  formal affine schemes 
  $\hat U_W \xrightarrow{h} \hat V_W\hookrightarrow \hat \A^d_W.$

    \medskip

  On $\hat \A^d_W$ we choose the lift $F_W$ which is defined by $F^*_W(x_i)=x_i^p$. 
  This defines $\hat V'_W:=F_W^{-1} \hat V_W$.
  We consider the cartesian diagram
\ga{}{\xymatrix{ \ar[d]^{h\times 1} \hat U_W\times_{\hat V_W}  \hat V'_W \ar[rr]^{F'_W} &  & \hat U_W \ar[d]^h \\
\hat V'_W \ar[rr]^{F_W} &  & \hat V_W} \notag }
defining $F'_W$.  Both  $\hat U_W\times_{\hat V_W}  \hat  V'_W$ and $\hat U_W$ are formal affine schemes with reduction
modulo $p$ isomorphic to $U_k$, so there is 
an isomorphism of formal schemes 
\ga{}{ \tau: \hat U_W\xrightarrow{\cong}  \hat U_W\times_{ \hat V_W}  \hat V'_W \notag}
over $W$,  see e.g. \cite[Th\'eor\`eme~2.1.3]{Ara01}, defining 
\ga{}{\xymatrix{  \hat U_W  \ar@/^2pc/[rrrr]^{G_W}  \ar[rr]^{\tau \cong}   \ar[drr]_g  &   & \ar[d]^{h\times 1} \hat U_W\times_{\hat V_W}  \hat V'_W \ar[rr]^{F'_W} &  & \hat U_W \ar[d]^h \\
& & \hat V'_W \ar[rr]^{F_W} &  & \hat V_W} \notag }
So $G_W$ 
is a  lift of Frobenius. 
By base change, it holds
\ml{}{  H^0(\hat U_W, G_K^* (\hat E_W,  \hat \nabla_W))=
H^0(\hat V'_W,  g_*G_K^* (\hat E_W, \hat  \nabla_W)) =\\  H^0(\hat V'_W, F_W^* h_* (\hat E_W, \hat \nabla_W))\notag}
so we may assume 
\ga{}{ (G_W: \hat U_W\to \hat U_W)=(F_W: \hat V'_W\to \hat V_W).\notag}
  Then
\ga{}{ H^0( \hat V'_W, F_W^* (\hat E_W, \hat \nabla_W)) \otimes_W K=
\oplus_{\chi_i, 0\le n_{ij_i}\le p-1}
 H^0(\hat V_K, (\hat E_K, \hat \nabla_K) 
  \otimes_{i=1}^d \chi^{n_{ij_i}} )\notag}
where the $\chi_i$ are the characters defining the Kummer cover $x_i \mapsto x_i^p$.
The only  tensor product  $\otimes_{i=1}^d \chi^{n_{ij_i}} $  which
extends as a formal  flat connection on $ \hat V_W$ is the trivial one, that is  the one for which all the $n_{ij_i}=0, \ i=1,\ldots, d$. Thus
\ga{}{ H^i( \hat V'_W, F_W^* (\hat E_W,  \hat \nabla_W)) \otimes_W K=H^i(\hat V_W, (\hat E_W, \hat \nabla_W)) \otimes_W K.\notag}
The second part of the proposition is also checked locally on $\hat U_W$ with Frobenius lift $F_W$, where it is trivial.
This finishes the proof.

\end{proof}

\subsection{Proof of \cite{EG20}, Theorem~1.6} \label{sec:pf}
Let $X_\C$ be a smooth projective variety. Let $(E_i, \nabla_i)_\C, \ i=1,\ldots, N$ be the finitely many irreducible rank $r$  rigid local systems with a fixed torsion determinant $\sL$.  We denote by the same letter $\sL$ the associated rank one connection $(\sL\otimes_{\C} \sO_{X_{\rm an}}, 1\otimes d)$. 
Let $M_{dR}(X,r,\sL)$   be the de Rham  moduli space of irreducible rank $r$  flat connections on $X_\C$ with determinant $\sL$.  So the connections  $(E_i,\nabla_i)_\C$ are the isolated complex points of $M_{dR}(X,r,\sL)$.  Let $S$ be a scheme such that
\begin{itemize}
\item[0)] $S\to {\rm Spec}(\Z)$ is smooth of finite type, $2$ is not in the image;
\item[1)]   $X_\C$ has a smooth projective model $X_S$;
\item[2)]  $\sL$ has a model $\sL_S$ and $M_{dR}(X,r,\sL)$ has a flat model  $M_{dR}(X,r,\sL) _S$   over $S$ (\cite[Theorem~1.1]{Lan14});
\item[3)]  the connections $(E_i,\nabla_i)_\C, \ i=1,\ldots, N$  have a model $(E_{i},\nabla_{i})_S$ defining $N$ $S$-sections of $M_{dR}(X,r,\sL)_S$;
\item[4)] $M_{dR}(X,r,\sL) _{S,red}$ is \'etale over $S$ at those $N$  pairwise different $S$-points; 
\item[5)]  the formal completion of $M_{dR}(X,r,\sL) _{S}$  at those sections is flat over $S$;
\item[6)]  the $S\otimes_{\Z} \Q$-modules $$H^j(X_S, \sE nd(E_{i}, \nabla_{i})_S), \ j=1,2,  i=1,\ldots, N$$ and 
\ml{}{  K((E_i,\nabla_i))_S= \\
 {\rm Ker}   \big( H^1(X_S, \sE nd(E_{i}, \nabla_{i})_S)
 \xrightarrow{x\mapsto x\cup x} 
H^2(X_S, \sE nd(E_{i}, \nabla_i)_S) \big)  \notag }
 satisfy base change.

\end{itemize}

\begin{thm} \label{thm:padic}
For any point ${\rm Spec}W\to S$ where $W=W(k)$ and $k$ is a perfect characteristic $p>0$ field,  for any $i=1,\ldots, N$, the $p$-completion $(\hat E_{i},\hat \nabla_{i})_W$ of $(E_{i},\nabla_{i})_S|_{X_W}$ is an isocrystal with a Frobenius structure. In particular it has nilpotent $p$-curvature. 
\end{thm}
\begin{proof}
By \cite[Proposition~4.4]{GM88} the completion of 
$ K((E,\nabla)_\C)$ at $0$  is isomorphic to the completion
at the  complex point $(E,\nabla)_\C$ of $M_{dR}(X,r, \sL)$. Grothendieck's comparison isomorphism
\ga{}{ H^i(X_W, (E_{i }, \nabla_{i})_S|_{X_W})  \xrightarrow{\cong} H^i(\hat X_W, (\hat E_i, \hat \nabla_i)_W) \notag
 }
 which implies the comparison isomorphism
 \ml{}{ H^i(X_K, (E_{i }, \nabla_{i})_S)=H^i(X_W, (E_{i }, \nabla_{i})_S|_{X_W})_\Q   \xrightarrow{\cong} \\ H^i(\hat X_W, (\hat E_i, \hat \nabla_i)_W)_\Q= H^i(\hat X_K, (\hat E_i, \hat \nabla_i)_K) \notag
 }
is compatible with the cup-product.  Here $K={\rm Frac}(W)$ and $\hat X_K=\hat X_W\otimes_W K$. Proposition~\ref{prop:uni} then implies that $F^*$ maps isomorphically the completion of $M_{dR}(X,r, \sL)_S$ at $(\hat E_i, \hat \nabla_i)_K$ to the one at $F^*(\hat E_i, \hat \nabla_i)_K$. Thus  the latter is  $0$-dimensional over $K$. We conclude
\ga{}{ F^*
(\hat E_i, \hat \nabla_i)_K = (\hat E_{\varphi(i)}, \hat \nabla_{\varphi(i)})_K 
 \ {\rm  for \ some \ } \varphi(i)\in \{1,\ldots, N\}. \notag} 
 Proposition~\ref{prop:uni} for $H^0$ implies then that $\varphi: \{1,\ldots, N\} \to \{1,\ldots, N\}$ is an injective map, thus is a bijection. Thus there is a natural number  $M$, which divides $N!$, such that $$\underbrace{\varphi \circ \ldots \circ \varphi}_{\text{M-times}}={\rm Identity}.$$
In other words
\ga{}{ (F^M)^* (\hat E_i, \hat \nabla_i)_K =(\hat E_i, \hat \nabla_i)_K \ \forall i=1,\ldots, N.\notag}
This finishes the first part of the proof. As for the second part, we apply  \cite[Proof~of~Proposition~3.1]{ES18} which shows that the nilpotency of the $p$-curvature does not depend on the lattice chosen. 

\end{proof}
\begin{rmk}
As compared to  \cite[Theorem~1.6]{EG20}, this proof has the advantage that it is $p$-adic as opposed to a characteristic $p>0$ proof. This is more natural as the result is $p$-adic. It also gives a larger $S$ on which the result holds.  It has the disadvantage to juggle with a cohomology which is nowhere studied in details (Proposition~\ref{prop:uni}) and, what is the most important point, it does not yield on a nose a Fontaine-Lafaille module on $X_W$, and its pendant which is a $p$-adic crystalline local system on $X_K$. This is the topic of Lecture~\ref{sec:crysLS}. The counting argument at the end is of course taken from the proof of \cite[Theorem~1.6]{EG20}.

\end{rmk}

\newpage
\section{Lecture 9:  Rigid Local Systems, Fontaine-Laffaille modules and Crystalline Local Systems} \label{sec:crysLS}
\begin{abstract}
As seen in Lecture~\ref{sec:Fiso}, originally proved with Michael Groechenig  in  \cite[Theorem~1.6]{EG20}, 
rigid connections  on 
 $X_{\C}$ smooth projective $\C$, while restricted to the formal $p$-completion $\hat X_W$  a non-ramified projective $p$-adic model $X_W$  of $X_\C$, yield $F$-isocrystals.   More is true.    By showing in \cite{EG20}  the existence of a {\it periodic} Higgs-de Rham flow  on  the formal connection $(\hat E_W,\ \hat \nabla_W)$  on $\hat X_W$,  we prove the existence of a Fontaine-Lafaille module structure on  $(\hat E_W,\ \hat \nabla_W)$ \cite[Section~4]{EG20}, which, via Faltings' functor, eventually yields a {\it crystalline} $\Z_{p^f}$-{\it local system} on the algebraic scheme $X_K$,  where $f$ is the period of the Higgs-de Rham flow.  This in turn implies that the rigid complex local systems on $X_\C$ , for $p>0$ large so they are integral by $p$, 
 the residual characteristic of such a good $W$,  descend to crystalline $\Z_{p^f}$-local systems, see   \cite[Section~5]{EG20}. This property remains true even if $X$ is only quasi-projective under a strong cohomological rigidity assumption,  which is  fulfilled on Shimura varieties of real rank $\ge 2$, and assuming  in addition that the local monodromies at infinity are unipotent, see \cite[Theorem~A.4, Theorem~A.22]{EG21}.  
\end{abstract}
\subsection{The main theorems}
We summarize the theorems in the projective case and then in the quasi-projective case separately as the assumptions in the latter case are more technical. 

\subsubsection{Good model in the projective case} \label{sec:goodproj}
We use the notation of Lecture~\ref{sec:Fiso}, Section~\ref{sec:pf}. We denote by $\sM$ the rank one Higgs bundle which corresponds to the torsion rank one connection $\sL$.  Explicitly, if $\sL=(L,\nabla)  $ then $\sM=(L, 0)$ as $\sL$ is assumed to be torsion. 
Recall from Simpson's theory  \cite[Lemma~4.5]{Sim92}  that on the Dolbeault moduli space $M_{Dol}(X,r,  \sM)$ of stable points over $\C$, we have the $\C^\times$-operation which acts as homotheties on the Higgs field. If $(V_i,\theta_i), \ i=1,\ldots, N$ are the $N$-Higgs bundles associated to the $N$ rigid connections $(E_i,\nabla_i), \ i=1,\ldots, N$, then $(V_i,\theta_i)$ is stable under $\C^\times$,  thus  $(E_i,\nabla_i)$ is a polarized complex variation of Hodge structure $(E_i,Fil_i,\nabla_i)$ where $Fil_i\subset E_i$ is a 
locally split  filtration  which satisfies  Griffiths transversality, and 
\ga{}{(V_i, \theta_i)=(gr (E_i),gr (\nabla_i)), \notag} see \cite[Lemma~4.5]{Sim92}.

\medskip

With reference to the conditions for the base $S$ of a {\it good model} in Lecture~\ref{sec:Fiso}, Section~\ref{sec:pf}, 
we request now $S$ to fulfil 
\begin{itemize}
\item[1a)] the same as 1) together with the existence of an $S$-point of $X_S$;
\item[2a)]  the same as 2)  for  $M_{dR}(X,r, \sL)$ and $M_{Dol}(X,r,\sM)$;
\item[3a)]  the same as 3)  for $M_{dR}(X,r, \sL)$ and   $M_{Dol}(X,r,\sM)$;
\item[4a)]  the same as 4)  for $M_{dR}(X,r, \sL)$ and  $M_{Dol}(X,r,\sM)$;
\item[5a)]  the same as 5)  for  $M_{dR}(X,r, \sL)$ and $M_{Dol}(X,r,\sM)$;
\item[6a)] the  filtrations $Fil_i\subset E_i$  have a model $Fil_{i,S} \subset E_{iS}$   for $i=1,\ldots, N$, which is locally split;
\item[7a)] for all closed points $s\in |S|$, the characteristic of the residue field of $s$ is $>r+1$ and $> {\rm dim}(X)$;
\item[8a)] the local systems $\mathbb L_i, i=1,\ldots, N$ associated by the Riemann-Hilbert correspondence to $(E_i,\nabla_i)_\C, i=1,\ldots, N$,  are integral at all residue characteristics $p$ of the closed points of $S$.
\end{itemize}
 The condition 6a) can be fulfilled as there are  finitely many filtrations $Fil_i$.  The conditions 2a)3a)4a)5a) can be fulfilled applying  the existence of the  flat moduli spaces $M_{dR}(X,r,\sL)_S$ and  $M_{Dol}(X,r,\sM)_S$ over $S$, due to Langer \cite[Theorem~1.1]{Lan14} and base change.  Note that condition 8a) implies that the $\mathbb L_i$ come from \'etale  $\bar \Z_p$- local systems on $X_\C$, which by base change, is the same as \'etale  $\bar \Z_p$- local systems on $X_{\bar K}$, where $K={\rm Frac}(W(s))$ and ${\rm Spec}(W(s))\to S$ is a Witt-vector point with residual closed point $s\in |S|$. 
  In addition they are  irreducible over $\bar \Q_p$.

\subsubsection{Theorem in the projective case}
\begin{thm}[See \cite{EG20}, Section 5] \label{thm:proj} For any closed point $s\in |S|$ in the basis of a good model, 
the  $p$-adic local systems $\mathbb L_i, i=1,\ldots, N$ on $X_{\bar K}$  descend to crystalline $p$-adic local systems on $X_K$.

\end{thm}

\subsubsection{Good model in the quasi-projective case}
We use the notation of Lecture~\ref{sec:Fiso}, Section~\ref{sec:pf} and of \ref{sec:goodproj}.  We fix a good compactification $X\hookrightarrow \bar X$ with  smooth $\bar X$ projective such that $D:= \bar X\setminus X$ is a strict normal crossings divisor. The de Rham moduli space in rank $r$ and fixed determinant $\sL$ in the projective case  is replaced by the de Rham moduli $M_{dR}(\bar X,r,\sL, D)$ of connections $(E,\nabla)$  on $\bar X$ with log-poles along $D$ and with nilpotent residues.  We remark, even if we do not  use it, that then $\sL$ extends to a torsion rank one connection on $\bar X$.
The nilpotency of the residues  implies that the (Betti or de Rham) Chern classes of the underlying bundles $E$ are $0$, see 
\cite[Proposition~B.1]{EV86}. We have $N$-complex points   $(E_i,\nabla_i)$ of $M_{dR}(\bar X,r,\sL, D)$ which describe the $0$-dimensional components (rigid objects). 
Furthermore, \cite[Lemma~4.5]{Sim92} is replaced by   \cite[Theorem~10.5]{Moc06} which guarantees that we have a locally split  filtration $Fil_i\subset E_i$ as in the projective case which satisfied Griffiths transversality. Finally
\ga{}{(V_i, \theta_i)=(gr (E_i),gr (\nabla_i)), \notag}
is a stable Higgs bundle with logarithmic poles along $D$ and nilpotent residues and determinant $\sM=(L,0)$.
With reference to the conditions for the base $S$ of a {\it good model} in Lecture~\ref{sec:Fiso}, Section~\ref{sec:pf}, and the ones in the projective case,
we request now $S$ to fulfil 
\begin{itemize} 
\item[1b)] $X_S\hookrightarrow \bar X_S$ is a model of $X\hookrightarrow \bar X $  such that $\bar X_S/S$ is smooth projective and $D_S:= \bar X_S\setminus X_S$ is a relative normal
crossings divisor, and  there is  an $S$-point of $X_S$;
\item[2b)] the same as 2) for  $M_{dR}(\bar X,r, \sL, D)$ and $M_{Dol}(\bar X,r,\sM, D)$; 
\item[3b)] the same as 3)  for  $M_{dR}(\bar X,r, \sL, D)$ and $M_{Dol}(\bar X,r,\sM, D)$;
\item[4b)] the same as 4)  for  $M_{dR}(\bar X,r, \sL, D)$ and $M_{Dol}(\bar X,r,\sM, D)$;
\item[5b)] the same as 5)  for  $M_{dR}(\bar X,r, \sL, D)$ and $M_{Dol}(\bar X,r,\sM, D)$;
\item[6b)] the same as 6a) for the Mochizuki filtrations;
\item[7b)] for all closed points $s\in |S|$, the characteristic of the residue field of $s$ in $>2(r+1)$ and $>{\rm dim}(X)$.

\end{itemize}

The condition 6b) can be fulfilled  for the same reason as for 6a): they are  finitely many filtrations $Fil_i$. 
 The conditions 2b)3b)4b)5b) can be fulfilled applying  the existence of the  flat moduli spaces $M_{dR}(\bar X,r,\sL, D)_S$ and  $M_{Dol}(\bar X,r,\sM, D)_S$ over $S$, due to Langer \cite[Theorem~1.1]{Lan14} and base change.

\subsubsection{Theorem in the quasi-projective case}
A flat connection $(E,\nabla)$  with logarithmitic poles and nilpotent residues  is in particular Deligne's extension of its restriction to $X$, so 
\ga{}{ H^j(X, (E,\nabla)|_X)=H^j(\bar X,( \Omega^\bullet_{\bar X}(\log D) \otimes_{\sO_{\bar X}} E,\nabla)).\notag}
Note that the restriction map 
\ga{}{ H^1(\bar X, j_{!*} \sE nd^0(E,\nabla))\hookrightarrow H^j(X, \sE nd^0(E,\nabla)|_X) \notag}
where $\sE nd^0$ denotes the trace free endomorphisms, is injective. In particular, if the right hand side vanishes, so does the left hand side. We can then apply the integrality theorem \cite[Theorem~1.1]{EG18}, see Lecture~\ref{sec:companions}, so all prime numbers are integral for $\mathbb L_i, i=1,\ldots, N$ on $X$. This explains why the condition which is  the pendant to 8a) is automatically fulfilled under this cohomological assumption and does not appear in the list of conditions.

\begin{thm}[See \cite{EG21}, Theorem~A.22] \label{thm:PST}
Assume that $H^1(X, \sE nd^0 (E_i,\nabla_i))=0$ for $i=1,\ldots, N$.  Then for all closed points  $s\in |S|$ of a good model, and all Witt vector points ${\rm Spec}(W(s))\to S$, the  $p$-adic local systems $\mathbb L_i, i=1,\ldots, N$  on $X_{\bar K}$ descend to a crystalline $p$-adic local systems on $X_K$ where $p$ is the residue characteristic of $s$.

\end{thm}
Theorem~\ref{thm:PST} applies for Shimura varieties of real rank $\ge 2$. This  the framework in which it is applied in the proof of the Andr\'e-Oort conjecture in \cite{PST21}. The aim of the rest of the Lecture is to formulate the main steps of the proof in the projective case. 

\subsection{Simpson's versus Ogus-Vologodsky's correspondences in the projective case}
Recall first the the Ogus-Vologodsky correspondence \cite{OV07} is a vast elaboration of Deligne-Illusie's splitting of the de Rham complex under the condition that $X$ smooth over a perfect field $k$  lifts to $W_2(k)$ (\cite{DI87}): for example assume  $(E,\nabla)$ has nilpotent $p$-curvature of level one, which means there is an exact sequence 
$0 \to (F^*S,\nabla_{can})\to (E,\nabla)\to (F^*Q, \nabla_{can})\to 0$  where $S,Q$ are coherent sheaves on the Frobenius twist $X'$ of $X$,  $F: X\to X'$ is the relative Frobenius, and $\nabla_{can}$ is the canonical connection determined by its flat sections $S,Q$ on $X'$.  Assume $(S,Q)$  are vector bundles. This defines a class in $H^1_{dR}(X, (F^*(Q^{-1}\otimes S), \nabla_{can}))$ which by \cite[Theorem~2.1]{DI87} is equal to $H^1(X', Q^{-1}\otimes S)\oplus H^0(X', \Omega^1_{X'}\otimes Q^{-1}\otimes S)$.  The class in $ H^1(X', Q^{-1}\otimes S) $ yields a vector bundle  extension
$$0\to S \to V'  \to Q\to 0$$ on $X'$ and the class in $H^0(X', \Omega^1_{X'}\otimes Q^{-1}\otimes S)$ yields a nilpotent Higgs field 
$$\theta': V'\to Q\to \Omega^1_{X'}\otimes S\to \Omega^1_{X'}\otimes V'.$$So here we need $p > {\rm dim}(X)$ to apply \cite{DI87} {\it loc. cit.}.  Ogus-Vologodksy correspondence assigns $C^{-1}(V', \theta'):=(E,\nabla)$  to $(V', \theta') $ under the assumption 7a).   Assume now that $X=X_s$ for a closed point $s\in |S|$ as in Theorem~\ref{thm:proj}. Then starting with 
$(V_i, \theta_i)= (gr (E_i), gr(\nabla_i))$ on $X_\C$,  we can restrict $(V_i, \theta_i)$  to $(V_i, \theta_i)_s$ on $X_s$, then take the pull-back  $(V'_i,  \theta'_i)_s$ under the arithmetic  Frobenius 
on ${\rm Spec} k$ where $k=\kappa(s)$ is the residue field of $s$, and then finally $C_s^{-1}(V'_i,\theta_i')_s$.  What is its relation to the restriction  $(E_i,\nabla_i)_s$ of $(E_i, \nabla_i)$? Here $C_s$ is the Cartier-Ogus-Vologodsky functor on $X_s$.  The miracle is the following theorem.
\begin{thm}[\cite{EG20}, Proposition~3.5] \label{thm:miracle}
 $C_s^{-1}(V'_i,\theta_i')_s$ is equal to one of the $(E_u,\nabla_u)_s$ where $u \in \{1, \ldots, N\}$.  

\end{thm}
\begin{proof}[Ideal of proof]
As $C_s^{-1}$ preserves the stability, see \cite[Corollary~5.10]{Lan14}, 
the key point is the  so-called Beauville-Narasimhan-Ramanan  correspondence \cite[Theorem~2.17]{EG20} as proved by Michael Groechenig  in his PhD Thesis \cite[Proposition~3.15]{Gro16}. In a sense it yields the Ogus-Vologodsky correspondence at the level of moduli spaces (on a curve) with a {\it scheme theoretic structure}.  If $C_s^{-1}(V'_i,\theta_i')$ was not rigid, there would be a deformation of its moduli point yielding a nilpotent structure around this point of order as high as we want. This nilpotent structure,  transported on the side of the Dolbeault moduli space via the Ogus-Vologodsky correspondence, yields  by rigidity over $\C$ and the property  5a)  a contradiction,  as  over $\C$ 
in $M_{dR}(X, r, \sL)$ the multiplicities of the isolated points are  given, thus bounded on $S$. 

\end{proof}
\begin{rmk}
The proof above unlocked the program we had  in order to prove the arithmetic crystalline properties associated to rigid connections  in Theorem~\ref{thm:proj}.
We relagate  to  \cite[Corollary~A.7]{EG20} a proof based on the same principle showing nilpotency  of the $p$-curvature (which we discussed via $p$-adic methods in  Theorem~\ref{thm:padic}),  which, in the words of an anonymous referee,  yields a canonical action of $\G^{\#}_m$  on the {\it category} of flat connections on $X_s$, 
where $\G^{\#}_m,$   is   the PD hull of the neutral element of $\G_m$.

\end{rmk}

\subsection{Periodic de Rham-Higgs  flow  on $X_s$   and the $GL_r(\bar \F_p)$-local systems  on $X_K$ in the projective case}
As $C_s^{-1}$ is an equivalence of categories (when it is defined),  we see that 
in  the proof of  Theorem~\ref{thm:miracle} the assignment $i \mapsto u$  is a  bijection of $\{1,\ldots, N\}$. 

\medskip

Defining a {\it de Rham-Higgs flow} as a sequence on $X_s$  of $\{(E_\iota,Fil_\iota, \nabla_\iota)\}$ with  $$(E_{\iota+1}, \nabla_{\iota+1})=C_s^{-1} (gr E_\iota, gr \nabla_\iota),$$ we see that for any $(E_\iota,\nabla_\iota)$ with $\iota=1\ldots, N$, $(E_\iota, Fil_\iota, \nabla_\iota)$ defines a {\it periodic} de Rham-Higgs flow of period $f(\iota)|N!$. 

\medskip

By \cite[Subsection~4.6]{OV07}, refined in \cite[Corollary~3.10]{LSZ19} to take into account the possibility of $f(\iota)>1$,  this defines a Fontaine-Lafaille module on $X_s$, which we do not define here as our emphasis is on the local system, see \cite[Theorem~3.3]{FL82}.  By the enhancement of the Fontaine-Lafaille construction in {\it loc. cit.} due to Faltings 
 [Theorem 2.6*]\cite{Fal88} (see \cite[Proposition~4.3]{EG20}), we obtain via Fontaine-Lafaille-Faltings' functor  the following claim.

 \begin{claim} 
 For  a closed point $s$ of a good model $S$,  we assign pairwise non-isomorphic local systems  $\sL_i(p)$ with values in $GL_r(\F_{p^{f(i)}})$ on $X_K$, for $i=1,\ldots, N$.

 \end{claim}
 
 \begin{rmk}
 The second miracle is provided by the theory of Fontaine-Lafaille-Faltings here: the local systems are defined over the {\it algebraic} $X_K$. 
 
 \end{rmk}
 
\subsection{Periodic de Rham-Higgs  flow  on $\hat X_W$   and the $GL_r(W(\bar \F_p))$-local systems  on $X_K$ in the projective case}
The functor $C_s^{-1}$ of Ogus-Vologosdsky extends to $C_n^{-1}$  on  $X_W\otimes_WW_n$ for all $n$, defining $\hat C$ on $\hat X_W$. 
The notion of de Rham-Higgs flow on the formal model $\hat X_W$ is completely similar to the one on $X_s$. See \cite[Corollary~3.10, Theorem~4.1]{LSZ19}, \cite[1.2.1]{SYZ22} and the work of Xu \cite{Xu19}  for very related methods and results. Then:

\begin{claim}
 For any $( \hat E_{\iota,W}, \hat \nabla_{\iota,W})$, the restriction of $(E_\iota, \nabla_{\iota})$ to $\hat X_W$,  with $\iota=1\ldots, N$,  the triple $(E_\iota, Fil_\iota, \nabla_\iota)$, defines a {\it periodic} de Rham-Higgs flow of period $f(\iota)|N!$.

\end{claim} 
By \cite[Subsection~4.6]{OV07}, refined in \cite[Corollary~3.10]{LSZ19} to take into account the possibility of $f(\iota)>1$,  this defines a Fontaine-Lafaille module on $\hat X_W$,  see \cite[Theorem~3.3]{FL82}.  By the enhancement of the Fontaine-Lafaille construction in {\it loc. cit.} due to Faltings 
 [Theorem 2.6*]\cite{Fal88} (see \cite[Proposition~4.3]{EG20}), we obtain via the Fontaine-Lafaille-Faltings functor the following claim.

 \begin{claim} 
 For  any closed point $s$ of a good model $S$,  we assign pairwise non-isomorphic $p$-adic  local systems  $\sL_i(W)$ with values in $GL_r(W( \F_{p^{f(i)}} ))$ on $X_K$, for $i=1,\ldots, N$, which are crystalline.

 \end{claim}

\begin{rmk}
This refinement of Faltings concerning crystallinity  (also in the quasi-projective case) is precisely the important property used in \cite{PST21}.
\end{rmk}

\subsection{From  crystalline $p$-adic local systems on $X_K$ to $p$-adic local systems on $X_{\bar K}$ in the projective case.}
The  local systems $(\mathbb L_i,  i=1,\ldots, N)$ are integral at $p$, thus by the identification $\pi_1(X_\C, x_\C)$ with $\pi_1(X_{\bar K}, x_{\bar K})$, they  become $p$-adic local systems on $X_{\bar X}$. We also have the base change $$\sL_i(W)|_{X_{\bar K}}=:\sL_i(\bar K)$$ to $X_{\bar K}$. To conclude Theorem~\ref{thm:proj} we should make sure that
the $\sL_i(\bar K)$  are not identified on  $X_{\bar K}$ and really come from the topological rigid local systems. 

\begin{thm}
The set of  $p$-adic local systems  $\{\sL_1(\bar K),\ldots, \sL_N(\bar K)\}$ is, up to order, the set $\{ \mathbb L_1, \ldots, \mathbb L_N\}$, viewed as $p$-adic local systems on $X_{\bar K}$.  Expressed the other way around, the $p$-adic local systems $\mathbb L_i$ on $X_{\bar K}$ descend to crystalline $p$-adic local systems on $X_K$, for $i=1,\ldots, N$ which in addition have values in $GL_r(W(\bar \F_p))$. 

\end{thm}
\begin{proof}[Idea of proof] The construction of the Fontaine-Lafaille modules over $\hat X_W$  is based on the fact that $\hat X_W$ is a non-ramified lift of $X_s$. We could 
repeat the construction on $\hat X_{W(\bar s)}$
 so as to obtain in   this way    local systems on $X_{{\rm Frac}(W(\bar \F_p))}$.  However,  we have to go all the way up  to $X_{\bar K}$. 
 
 \medskip
 
 To perform this, we use  in \cite[Section~5]{EG20} Faltings' $p$-adic Simpson correspondence. In order to be able to apply Faltings' theory, we have to be sure that the Higgs bundles on $X_{\bar K}$  and  the $\sL_i(\bar K)$ are small in his sense. We can find infinitely many prime numbers $p>0$ with the property that $\kappa(s)=\F_p$ and all the $f_i=1$. This is the content of \cite[Lemma~5.11]{EG20} and yields a weaker form of Theorem~\ref{thm:proj}. We do not get the result on all closed points of $S$ but on an infinity of those with infinite different residual characteristics. 
 
 \medskip

In general, we resort to \cite{SYZ22},  which relies on Faltings' $p$-adic Simpson correspondence, 
 and  ``does'' part of the Fontaine-Lafaille-Faltings program in the ramified case. We first argue that $\sL_i(\bar K)\otimes_{\bar Z_p} \bar \F_p$ is irreducible. If not, using the $W$-point in 1a) and the consequence that $\pi_1(X_K)$ is then a semi-direct product of ${\rm Gal}(\bar K/K)$ with $\pi_1(X_{\bar K})$, we can kill on ${\rm Gal}(\bar K/K)$ the  $GL_r(\bar \F_p)$-representation by a finite field extension $K'/K$. This enables us to conclude that $\sL_i(K')$ itself is reducible, which by \cite[Theorem~5.15]{SYZ22} violates the stability of the associated Higgs bundle on $X_{K'}$. We argue similarly to distinguish the $\sL_i(\bar K)\otimes_{\bar Z_p} \bar \F_p$  for $i=1,\ldots, N$. This finishes the proof. 

 
 \medskip
 

\end{proof}

\begin{rmk}
We hope in the near future  (\cite{EG22}) to strengthen the results and shorten the proofs of the existing ones using more  and in particular more recent $p$-adic methods. 
\end{rmk}

\newpage

\section{Lecture 10: Comments and Questions} \label{sec:CQ}
\begin{abstract}
The aim of this last Lecture is to list a few questions encountered  during the Lectures.
\end{abstract}

\subsection{With respect to the $p$-curvature conjecture (Lecture~\ref{sec:pcurv})}
What would be  a formulation of the $p$-curvature conjecture when we replace a smooth quasi-projective variety over $\C$ by
the formal completion along a non-normal subvariety  of a smooth variety over $\C$? The question is motivated by a discussion with Johan de Jong on \cite{LR96}. 

\subsection{With respect to the Mal\v{c}ev-Grothendieck's theorem and its shadows in characteristic $p>0$ (Lecture~\ref{sec:Malcev})}
As already mentioned in Section~\ref{sec:padic} we have two versions of Mal\v{c}ev-Grothendieck theorem over an algebraically closed field of characteristic $p>0$, one on the infinitesimal site (Gieseker's conjecture, solved, see Theorem~\ref{thm:EM10}), one on the crystalline site (unsolved) due to de Jong, but we do not have at present a formulation in the prismatic site, one difficulty being the {\it iso}-notion. 

\subsection{With respect to Lubotzky's Theorem~\ref{thm:Lubotzky}} \label{sec:Lub}   Notation as in {\it loc. cit.}.
Assume $\pi$ is  a profinite group. What can be said on 
$${\rm dim}_{\F_\ell}H^i(\pi, \rho ) \ {\rm  for \  any \ }   i\in \N ?$$

\medskip
\noindent
It holds ${\rm dim}_{\F_\ell}H^0(\pi, \rho)\le r$ for any representation, whatever $\pi$ is. 

\medskip
\noindent
It holds
${\rm dim}_{\F_\ell}H^1(\pi, \rho)\le \delta \cdot r$ for any profinite group  $\pi$  which is topologically generated by $\delta$ elements, as a $1$-continuous cocycle is uniquely defined by its value on topological generators of $\pi$.  

\medskip
We observe that this linear bound for $H^1$  is not equivalent to the finite generation of the profinite group $\pi$.  Here is an example due to Lubotzsky. Let $G$ be a 
finite  non-abelian simple group, let $\pi=\prod_0^\infty G$ be the infinite product of $G$ with itself, endowed with the profinite structure stemming from all finite quotients. As a finite quotient $\pi\surj Q$ has to factor through some $G^m$ for some $m\in \N$, $Q$ itself has to be isomorphic to some $G^n$ for some $n\in \N$.     As ${\rm dim}_{\F_\ell}H^1 (\pi, \rho ) \le {\rm dim}_{\F_\ell}H^1(\pi, \rho^{ss} )$ where $\rho^{ss}$ is the semi-simplification of $\rho$, we may assume that $\rho$ is irreducible. 
We write the exact sequence $1\to {\rm Ker}(\rho)\to \pi\to \rho(\pi)\cong G^n\to 1$.
Then  ${\rm Ker}(\rho)$, which is abstractly isomophic to $\pi$, has no abelian quotient. Thus $H^1({\rm Ker}(\rho), \rho|_{{\rm Ker}(\rho)})=0$ and the Hochshild-Serre spectral sequence yields an isomorphism $H^1(\rho(\pi), \rho) \xrightarrow{\cong} H^1(\pi, \rho)$. By  \cite[Theorem~A]{AG84},  $ {\rm dim}_{\F_\ell}H^1(\rho(\pi), \rho ) \le r$. This finishes the proof.

\medskip
\noindent
For $i=2$  Lubotzky's theorem~\ref{thm:Lubotzky}  yields even a characterization of finite presentation.  

\medskip
\noindent
What does the growth of the cohomology for $i\ge 3$ encode as a property?  

\medskip
\noindent
Do we have special properties for ${\rm dim}_{\F_\ell}H^i(\pi, \rho )$ for any $i \ge 3$ if $\pi$ is the tame fundamental group of $X$ smooth quasi-projective over any algebraically closed field?

\subsection{With respect to Theorem~\ref{thm:groupj}}
Notation as in {\it loc. cit.}.  It is wishful that   Theorem~\ref{thm:ESS22} be true under the assumption that $X$ is quasi-projective normal over an algebraically closed  characteristic $p>0$ field $k$.

\medskip
\noindent
If $j: X\hookrightarrow \bar X$ is a normal compactification of a normal  $X$ over $k$,   is there a formula which enables one to bound $H^2(\pi_1^t(X), M)$ by some  cohomological invariant of a constructible sheaf on $\bar X$ built out of the local system $\underline M$? 

\medskip
\noindent
  What about higher cohomology $H^i(\pi_1^t(X), M), \ i\ge 3$, also for $X$ normal (see  Section~\ref{sec:Lub})?

\subsection{With respect to Theorem~\ref{thm:C} and Theorem~\ref{thm:obst}}
Can we extend  those theorems  to formally smooth proper schemes, to smooth rigid spaces,  is there  a version for normal varieties, for non-normal varieties involving the pro-\'etale fundamental group  developed in \cite{Sch13} and in all generality in \cite{BS15} etc.  We can also raise similar questions concerning the finite presentation of the (tame) fundamental group. (The question for formal and rigid spaces was posed to us by Piotr Achinger and Ben Heurer). 
 
\subsection{With respect to Theorem~\ref{thm:dens} } 
Can we replace in the formulation   of  Theorem~\ref{thm:dens}   the Betti moduli in a given rank of a normal complex variety by the Mazur or Chenevier deformation space  of a given $\bar \F_\ell$-residual representation over $X$ over  $\bar \F_p$ to obtain in those spaces Zariski density of the local systems with quasi-unipotent monodromies at infinity?

\subsection{With respect to Theorem~\ref{thm:EG18} and a more elaborate version  of Theorem~\ref{thm:dJE22}, see Theorem~\cite[Theorem~1.4]{dJE22}}
The main question is ``where'' is $sp_{\C,\bar s}^{-1}(\mathbb L^\sigma_{\ell})$  located with respect to $\mathbb L_\ell$ on $M_B(X,r,\sL, T_i)$. If $\mathbb L_\ell$ lied on a $d$-dimensional component, what about  $sp_{\C,\bar s}^{-1}(\mathbb L^\sigma_{\ell})$?  For example, if $X=\P^1\setminus \Sigma$ where $\Sigma\subset \P^1$ consists of finitely many 
points, then  irreducible rigid local systems are cohomologically rigid (\cite[ Corollary~1.2.5]{Kat96}), thus the proof of  Theorem~\ref{thm:EG18} shows that then if $d=0$, then $sp_{\C,\bar s}^{-1}(\mathbb L^\sigma_{\ell})$ is a $0$-dimensional component as well. What about the higher dimensional components in this case?

\subsection{With respect to Theorem~\ref{thm:dJE22}}
Does Theorem~\ref{thm:dJE22}
hold with $X$ defined over $\bar \F_p$? To be precise, assume that in a given rank $r$, there is an irreducible $\ell$-adic local system. Can we conclude that for all $\ell'\neq p$ there is an irreducible $\ell'$-adic local system of rank $r$, and that perhaps by $\ell'=p$, there is an irreducible isocrystal of rank $r$?  Note if arithmetic $\ell$-adic local systems were dense
in the Chenevier deformation space as wished for in Conjectures~\cite[Weak~and~Strong~Conjectures]{EK22}, this much weaker problem would have a positive answer.

\subsection{With respect to Lectures~\ref{sec:Fiso} and ~\ref{sec:crysLS}}
As the de Rham-crystalline side of the theory does not request any cohomological condition on the rigid systems when $X$ is projective, if would be wishful to understand the full strength of the theorems in the non-proper case.

\end{document}